\documentclass[10pt,a4paper]{amsart}
\usepackage[utf8]{inputenc}
\usepackage[T1]{fontenc}
\usepackage{amsmath}
\usepackage{amsfonts}
\usepackage{amssymb}
\usepackage[english]{babel}

\newtheorem{defi}{Definition}
\newtheorem{prop}{Proposition}
\newtheorem{con}{Conjecture}
\newtheorem{thm}{Theorem}
\newtheorem{que}{Question}
\newtheorem{ex}{Example}
\newtheorem{cor}{Corollary}
\newtheorem{rem}{Remark}
\newtheorem*{lem*}{Lemma}

\def\cal{\mathcal}
\def\bbb{\mathbb}

\newcommand{\lcm}{\emph{lcm}}
\newcommand{\elcm}{\mbox{\scriptsize{lcm}}}
\newcommand{\eqlcm}{\mbox{lcm}}
\newcommand{\sgn}{\mbox{sgn}}
\newcommand{\bs}{\backslash}
\newcommand{\N}{\mathbb{N}}
\newcommand{\Z}{\mathbb{Z}}
\newcommand{\Q}{\mathbb{Q}}

\newcommand{\R}{\mathbb{R}}

\newcommand{\eps}{\varepsilon}

\theoremstyle{remark}

\title[Arithmetic properties of the sequence of derangements]{Arithmetic properties of the sequence of derangements and its generalizations}
\author{Piotr Miska}
\address{Institute of Mathematics \\
	Faculty of Mathematics and Computer Science \\
	Jagiellonian University in Cracow
}
\email{piotr.miska@uj.edu.pl}


\begin{document}

\setlength{\parindent}{10mm}
\maketitle

\begin{abstract}
The sequence of derangements is given by the formula $D_0 = 1, D_n = nD_{n-1} + (-1)^n, n>0$. It is a classical object appearing in combinatorics and number theory. In this paper we consider two classes of sequences: first class is given by the formulae $a_0 = h_1(0), a_n = f(n)a_{n-1} + h_1(n)h_2(n)^n, n>0$, where $f,h_1,h_2 \in\Z[X]$, and the second one is defined by $a_n = \sum_{j=0}^n \frac{n!}{j!} h(n)^j, n\in\N$, where $h\in\Z[X]$. Both classes are a generalization of the sequence of derangements. We study such arithmetic properties of these sequences as: periodicity modulo $d$, where $d\in\N_+$, $p$-adic valuations, asymptotics, boundedness, periodicity, recurrence relations and prime divisors. Particularly we focus on the properties of the sequence of derangements and use them to establish arithmetic properties of the sequences of even and odd derangements.
\end{abstract}

\tableofcontents

\section{Introduction}

By the term of \emph{derangement} we call a permutation in $S_n$ without fixed points. We define \emph{the $n$-th number of derangements} as the number of all derangements of the set with $n$ elements. We denote this number by $D_n$. The sequence $(D_n)_{n\in\N}$, can be described by the recurrence $D_0 = 1, D_n = nD_{n-1} + (-1)^n, n>0$. The sequence $(D_n)_{n\in\N}$ is a subject of reaserch of many mathematicians. It is connected to other well known sequences. In particular, the sequence of numbers of derangements (or shortly, the sequence of derangements) appears in a natural way in the paper \cite{SunZa}, devoted to the Bell numbers.

In \cite{Mi} we gave and proved a criterion for behavior of $p$-adic valuation of the Schenker sum $a_n$, given by the formula $a_n = \sum_{j=0}^n \frac{n!}{j!} n^j$, $n\in\N$. We expected that the method of proving this criterion could be generalized to other class of integer sequences. The trial of generalization of this method is one of the motivations for preparing this paper.

In Section \ref{sec1} we set conventions and recall facts which are used in further parts of the thesis.

In Section \ref{sec2} we define pseudo-polynomial decomposition modulo $p$ of a given sequence. If a sequence has this property then we can use the same method of proof as in \cite{Mi} to obtain the description of $p$-adic valuation of elements of this sequence. Furthermore, we show that a sequence with pseudo-polynomial decomposition modulo $p$ can be expressed as a product of functions $f$ and $g$, where $f:\Z_p\rightarrow\Z_p$ is a $p$-adic continuous function which can be approximated by polynomials with integer coefficients and $g:\N\rightarrow\Z_p\bs p\Z_p$. The last part of Section \ref{sec2} is devoted to description of $p$-adic valuation of the exponential function $\Z\ni n\mapsto a^n\in\Z$.

The results from Section \ref{sec2} are used in Section \ref{sec3} to study arithmetic properties of a family of sequences ${\bf a}={\bf a}(f,h_1,h_2)=(a_n)_{n\in\N}$ given by the recurrence relation
\begin{equation}\label{def}
a_0 = h_1(0), a_n = f(n)a_{n-1} + h_1(n)h_2(n)^n, n>0,
\end{equation}
where $f,h_1,h_2\in\Z[X]$. Let us define 
$$
\cal{R}:=\{{\bf a}(f,h_1,h_2)\in\R^{\N}:\;f, h_{1}, h_{2}\in\Z[X]\}.
$$ If $f=X$, $h_1=1$ and $h_2=-1$ then we obtain the sequence of derangements $(D_n)_{n\in\N}$, hence the class of sequences given by the relation (\ref{def}) can be treated as a generalization of the sequence of derangements.

Section \ref{subsec3.1} is concerned with the periodicity of the sequences $(a_n\pmod{d})_{n\in\N}$ of remainders modulo $d$ of a given sequence $(a_n)_{n\in\N}$, where $d\in\N_+$. Moreover we focus on $p$-adic valuations of the sequence $(a_n)_{n\in\N}$, when $p\mid f(n)$ for some $n\in\N$ and $h_2=\pm 1$. Due to the divisibility $n-1\mid D_n$ for all $n\in\N$, we study closer prime divisors and $p$-adic valuations of the sequence $(\frac{D_n}{n-1})_{n\in\N_2}$. We prove that the set of prime divisors of the numbers $\frac{D_n}{n-1}$, $n\geq 2$, is infinite.

Section \ref{subsec3.2} is devoted to asymptotics of a given sequence ${\bf a}\in\cal{R}$ and connection between boundedness and periodicity of this sequence. The main result of this section is that each bounded sequence $(a_n)_{n\in\N}$ is ultimately constant or ultimately periodic with period 2.

In Section \ref{subsec3.3} we obtain some recurrence relations for a sequence ${\bf a}(f,h_{1},h_{2})\in\cal{R}$, when $h_2=\pm 1$. Then we study real roots of the polynomials ocurring in these relations and we conclude that for $d\in\N_+$ the polynomial
\begin{equation*}
f_d = \sum_{j=0}^{d-1} (-1)^j \prod_{i=0}^{j-1} (X-i) \in\Z[X],
\end{equation*}
which arises in the formula
\begin{equation*}
D_n = D_{n-d}\prod_{i=0}^{d-1} (n-i) + (-1)^n f_d(n), \quad n\geq d,
\end{equation*}
has exactly $d-1$ real roots and exactly one rational root 1.

Section \ref{subsec3.4} deals with divisors of terms of a sequence ${\bf a}\in\cal{R}$. Section \ref{subsubsec3.4.1} is a trial of generalization of the result from Section \ref{subsec3.1} that there are infinitely many prime divisors of the numbers $\frac{D_n}{n-1}$, $n\geq 2$. We give some conditions for infinitude of set of prime divisors of a given sequence ${\bf a}$. The last two results in Section \ref{subsubsec3.4.1} show that if a sequence $(a_n)_{n\in\N}$ is given by the formula $a_0 = c, a_n = (b_1n+b_0)a_{n-1} + c, n>0$ for some integers $b_0,b_1,c$ and $b_0,b_1$ are not simultaneously 0 then there are infinitely many prime divisors of the numbers $a_n$, $n\in\N$. In Section \ref{subsubsec3.4.2} we generalize the property $n-1\mid D_n$, $n\in\N$. Namely, we consider sequences given by the formula $a_0=h_1(0), a_n = (n-b)a_{n-1} + h_1(n)h_2(n)^n, n>0$, where $b\in\Z$ is fixed, and study when $n-b-1\mid a_n$.

In Section \ref{sec4} we use the results on the sequence of derangements to obtain arithmetic properties of the sequences of even and odd derangements. First of all we present recurrence relations for these two sequences. We obtain relations involving numbers of even and odd derandements in order to write them as expressions dependent on numbers of derangements. Next we show their asymptotics and periodicity modulo $d$, where $d\in\N_+$. From the periodicity properties we conclude divisibilities of these numbers and describe their $p$-adic valuations.

The subject of Section \ref{sec4,5} are diophantine equations with numbers of usual, odd and even derangements, respectively. In Section \ref{subsec4,5.1} we find all the numbers of usual and odd derangements which are factorials. Meanwhile in Section \ref{subsec4,5.2} we try to establish for which indices $n$ the numbers of usual, odd and even derangements, respectively, are powers of prime numbers.

Section \ref{sec5} is devoted to the $h$-Schenker sums, given by the formula
\begin{equation*}
a_n = \sum_{j=0}^n \frac{n!}{j!} h(n)^j, n\in\N,
\end{equation*}
where $h$ is a given polynomial with integer coefficients. If $h=X$ then we obtain the sequence of Schenker sums, hence the motivation to call the mentioned class of sequences by $h$-Schenker sums. If $h=-1$ then $h$-Schenker sums are numbers of derangements, so the sequence of $h$-Schenker sums can be seen as a generalization of the sequence of derangements. In \cite{AmdCalMo} and \cite{Mi} there were established some results on $p$-adic valuations of Schenker sums and infinitude of the set of so-called Schenker primes (such prime numbers $p$ that $p\mid a_n$ for some $n\in\N$ not divisible by $p$). In Section \ref{sec5} we generalize these results.

In Section \ref{subsec5.1} we prove periodicty modulo $d$ of $h$-Schenker sums for a given $d\in\N_+$ and describe their $p$-adic valuations. During considerations on $p$-adic valuations we define $h$-Schenker prime as prime number $p$ such that $p\mid a_n$ and $p\nmid h(n)$ for some $n\in\N$.

Section \ref{subsec5.2} starts with giving bounds on absolute values of $h$-Schenker sums. Next these bounds are used to establish infinitude of the set of $h$-Schenker primes for $h\neq 0$.

\section{Definitions and conventions}\label{sec1}

We assume that $\N = \{0, 1, 2, 3, ...\}$ and $\N_+ = \{1, 2, 3, ...\}$. For a given positive integer $k$ we denote the set of all integers greater than or equal to $k$ by $\N_k$. We denote the set of all prime numbers by $\bbb{P}$.

We set a convention that $0^0 = 1$, $\sum_{i=j}^k = 0$ and $\prod_{i=j}^k = 1$, when $j,k\in\Z$ and $j>k$.

Let $a,b\in\Z$ and $b\neq 0$. Then by $a\pmod{b}$ we denote the remainder from the division of $a$ by $b$.

By $s_d(n)$ we denote the sum of digits of positive integer $n$ in base $d$, i.e. if $n=\sum_{i=0}^m c_id^i$ is an expansion of $n$ in base $d$ then $s_d(n)=\sum_{i=0}^m c_i$.

Let $A,B$ be topological spaces. The set of all continuous functions $f:A\rightarrow B$ we denote by $\cal{C}(A,B)$. If $C$ is a subset of $A$ then its closure in $A$ we denote by $\overline{C}$.

Fix a prime number $p$. Every nonzero rational number $x$ can be written in the form $x=\frac{a}{b}p^t$, where $a\in\mathbb{Z}$, $b\in\mathbb{N}_+$, $\gcd(a,b)=1$, $p\nmid ab$ and $t\in\Z$. Such a representation of $x$ is unique, thus the number $t$ is well defined. We call $t$ the $p$-adic valuation of the number $x$ and denote it by $v_p(x)$. By convention, $v_p(0)=+\infty$. In particular, if $x\in\mathbb{Q}\setminus\lbrace 0\rbrace$ then $\vert x\vert=\prod_{p \mbox{\scriptsize{ prime}}}p^{v_p(x)}$, where $v_p(x)\neq 0$ for finitely many prime numbers $p$.

For every rational number $x$ we define its $p$-adic norm $|x|_p$ by the formula
\begin{equation*}
|x|_p =
\begin{cases}
p^{-v_p(x)}, & \mbox{when } x\neq 0
\\ 0, & \mbox{when } x=0
\end{cases}.
\end{equation*}
Since for all rational numbers $x,y$ we have $|x+y|_p \leq \min\{ |x|_p, |y|_p\}$, hence $p$-adic norm gives a metric space structure on $\Q$. Namely, the distance between rational numbers $x,y$ is equal to $d_p(x,y) = |x-y|_p$.

The field $\Q$ equipped with $p$-adic metric $d_p$ is not a complete metric space. The completion of $\Q$ with respect to this metric has structure of field and this field is called the field of $p$-adic numbers $\Q_p$. We extend the $p$-adic valuation and $p$-adic norm on $\Q_p$ in the following way: $v_p(x) = \lim_{n\rightarrow +\infty} v_p(x_n)$, $|x|_p = \lim_{n\rightarrow +\infty} |x_n|_p$, where $x\in\Q_p$, $(x_n)_{n\in\N} \subset \Q$ and $x = \lim_{n\rightarrow +\infty} x_n$. The values $v_p(x)$ and $|x|_p$ do not depend on the choice of a sequence $(x_n)_{n\in\N}$, thus they are well defined (see \cite{Cas}).

We define the ring of integer $p$-adic numbers $\Z_p$ as a set of all $p$-adic numbers with nonnegative $p$-adic valuation. Note that $\Z_p$ is the completion of $\Z$ as a space with $p$-adic metric.

We assume that the expression $x \equiv y \pmod{p^k}$ means $v_p(x-y) \geq k$ for prime number $p$, an integer $k$ and $p$-adic numbers $x,y$.

By the term $p$-adic continuous function we mean function $f:S\rightarrow\Q_p$ defined on some subset $S$ of $\Q_p$, which is continuous with respect to $p$-adic metric. By the term $p$-adic contraction we mean such function $f:S\rightarrow\Q_p$ that $|f(x)-f(y)|_p \leq |x-y|_p$ for arbitrary $x,y\in S$. Assuming that $S$ is an open subset of $\Q_p$, we will say that $f$ is differentiable at a point $x_0\in S$, if there exists a limit $\lim_{x\rightarrow x_0}\frac{f(x)-f(x_0)}{x-x_0}$. In this situation this limit we will call the derivative of $f$ at the point $x_0$ and denote it by $f'(x_0)$.

We use the Landau symbol $O$ in the following sense: if $f,g:\N\rightarrow\R$ are two real-valued functions defined on $\N$ then $f(n)$ is $O(g(n))$, when there exists such a constant $M\geq 0$ and a nonnegative integer $n_0$ that $|f(n)|\leq M|g(n)|$ for $n\geq n_0$. In general, if there exists an $n_0\in\N$ such that some property holds for $n\geq n_0$ then for simplicity of notation we will write that this property is satisfied for $n\gg 0$.

\pagebreak

\section{Hensel's lemma for $p$-adic continuous functions approximated by polynomials over $\Z$}\label{sec2}

In \cite{Mi} there was presented a consideration which allows to describe $p$-adic valuation of Schenker sums, given by the formula $a_n = \sum_{j=0}^n \frac{n!}{j!} n^j$. In this section we will extend this method to more general class of sequences. The results given in the following section will be used in the sequel.

\subsection{Hensel's lemma for pseudo-polynomial decomposition modulo $p$}

\begin{defi}\label{def1}
Let $p$ be a prime number and $(a_n)_{n\in\N}\subset\Z_p$. By \emph{pseudo-polynomial decomposition of the sequence $(a_n)_{n\in\N}$ modulo $p$ on a set $S\subset\N$} we mean a sequence of pairs $(f_{p,k},g_{p,k})_{k\in\N_2}$ such that:
\begin{itemize}
\item $f_{p,k}\in\Z_p[X]$, $g_{p,k}:S\rightarrow\Z_p\bs p\Z_p$, $k\geq 2$;
\item $a_n \equiv f_{p,k}(n)g_{p,k}(n) \pmod{p^k}$ for all $n\in S$, $k\geq 2$;
\item $f'_{p,k}(n) \equiv f'_{p,2}(n) \pmod{p}$ for any $k\geq 2$ and $n\in S$,
\end{itemize}
where $f'$ means the derivative of a polynomial $f$. We say that $(a_n)_{n\in\N}$ has a pseudo-polynomial decomposition modulo $p$ if it has a pseudo-polynomial decomposition modulo $p$ on $\N$.
\end{defi}

\begin{rem}
Let $n_k\in\N$ and assume that a set $S\in\N$ is dense in the set \newline$\{n\in\N: n\equiv n_k\pmod{p^k}\}$ with respect to $p$-adic metric. Since $\{n\in\N: n\equiv n_l\pmod{p^l}\} = \{n\in\N: d_p(n,n_l)<p^{1-l}\}$ for any integer $l\geq k$ and positive integer $n_l\equiv n_k\pmod{p^k}$, hence $S\cap\{n\in\N: n\equiv n_l\pmod{p^l}\}\neq\emptyset$.
\end{rem}

\begin{thm}[Hensel's lemma for pseudo-polynomial decomposition modulo $p$]\label{thm1}
Let $p$ be a prime number, $k\in\N_+, n_k\in\N$ be such that $p^k\mid a_{n_k}$ and assume that $(a_n)_{n\in\N}\subset\Z_p$ has a pseudo-polynomial decomposition modulo $p$ on $S\subset\N$, where $n_k\in S$. Let us define $q_p(n_k)=\frac{1}{p}\left(\frac{a_{n_k+p}}{g_{p,2}(n_k+p)}-\frac{a_{n_k}}{g_{p,2}(n_k)}\right)$.
\begin{itemize}
\item If $v_p(q_p(n_k))=0$ and $S$ is dense in the set $\{n\in\N: n\equiv n_k\pmod{p^k}\}$ with respect to $p$-adic metric then there exists a unique $n_{k+1}$ modulo  for which $n_{k+1}\equiv n_k \pmod{p^k}$ and $p^{k+1}\mid{a_n}$ for all $n\in S$ congruent to $n_{k+1}$ modulo $p^{k+1}$. What is more, $n_{k+1} \equiv n_k-\frac{a_{n_k}}{g_{p,k+1}(n_k)q_p(n_k)} \pmod{p^{k+1}}$.
\item If $v_p(q_p(n_k))>0$ and $p^{k+1}\mid{a_{n_k}}$ then $p^{k+1}\mid{a_n}$ for all $n\in S$ satisfying $n\equiv n_k \pmod{p^k}$.
\item If $v_p(q_p(n_k))>0$ and $p^{k+1}\nmid{a_{n_k}}$ then $p^{k+1}\nmid{a_n}$ for any $n\in S$ satisfying $n\equiv n_k \pmod{p^k}$.
\end{itemize}
In particular, if $k=1$, $p\mid{a_{n_1}}$, $v_p(q_p(n_1))=0$ then for any $l\in\mathbb{N}_+$ there exists a unique $n_l$ modulo $p^l$ such that $n_l\equiv n_1 \pmod{p}$ and $v_p(a_n)\geq{l}$ for all $n\in S$ congruent to $n_l$ modulo $p^l$. Moreover, $n_l$ satisfies the congruence $n_l \equiv n_{l-1}-\frac{a_{n_{l-1}}}{g_{p,l}(n_{l-1})q_p(n_1)} \pmod{p^l}$ for $l>1$.
\end{thm}

In \cite{Mi} there was showed that the sequence of Schenker sums $(a_n)_{n\in\N}$ satisfies the congruence $a_n \equiv n^{n-p^k+2}f_{p,k}(n) \pmod{p^k}$ for each positive integer $k$, prime number $p$ and positive integer $n$ not divisible by $p$, where
\begin{equation*}
f_{p,k}=\sum_{j=0}^{d-1}X^{d-j-2}\prod_{i=0}^{j-1}(X-i).
\end{equation*}

Moreover, if $k_1,k_2\geq 2$ then $f'_{p,k_1}(n) \equiv f'_{p,k_2}(n) \pmod{p}$. This fact and Hensel's lemma allow to state the criterion for behavior of $p$-adic valuation of the Schenker sums.

In order to prove Theorem \ref{thm1} we will use the following version of Hensel's lemma (see \cite[p. 44]{Nar} and \cite[p. 49]{Cas}):

\begin{thm}[Hensel's lemma]\label{thm2}
Let $p$ be a prime number, $k$ be a positive integer and $f$ be a polynomial with integer $p$-adic coefficients. Assume that $f(n_0)\equiv 0 \pmod{p^k}$ for some integer $n_0$. Then the number of solutions $n$ of the congruence $f(n)\equiv 0 \pmod{p^{k+1}}$, satisfying the condition $n \equiv n_0 \pmod{p^k}$, is equal to:
\begin{itemize}
\item $1$, when $f'(n_0)\not\equiv 0 \pmod{p}$;
\item $0$, when $f'(n_0)\equiv 0 \pmod{p}$ and $f(n_0)\not\equiv 0 \pmod{p^{k+1}}$;
\item $p$, when $f'(n_0)\equiv 0 \pmod{p}$ and $f(n_0)\equiv 0 \pmod{p^{k+1}}$.
\end{itemize}
\end{thm}

Now we are ready to prove Theorem \ref{thm1}.

\begin{proof}[Proof of Theorem \ref{thm1}]
Let us note that if $f\in\Z_p[X]$ then for any $x_0\in\Z_p$ there exists an $r\in\Z_p[X]$ such that
\begin{equation}\label{eq1}
f(X) = f(x_0) + (X-x_0)f'(x_0) + (X-x_0)^2r(X).
\end{equation}

Using the equality above for $f=f_{p,2}$, $x_0=n_k$ and $X=n_k+p$, we have:
\begin{equation*}
f_{p,2}(n_k+p) \equiv f_{p,2}(n_k) + pf'_{p,2}(n_k) \pmod{p^2}.
\end{equation*}

This congruence and Definition \ref{def1} imply the following:
\begin{equation*}
\begin{split}
& f'_{p,k+1}(n_k) \equiv f'_{p,2}(n_k) \equiv \frac{1}{p}(f_{p,2}(n_k+p) - f_{p,2}(n_k)) \equiv \\
& \equiv \frac{1}{p}\left(\frac{a_{n_k+p}}{g_{p,2}(n_k+p)} - \frac{a_{n_k}}{g_{p,2}(n_k)}\right) = q_p(n_k) \pmod{p}.
\end{split}
\end{equation*}

Thus $q_p(n_k) \equiv f'_{p,k+1}(n_k) \pmod{p}$. Since $p\nmid g_{p,k+1}(n)$ for each nonnegative $n \equiv n_k \pmod{p^k}$, hence $v_p(a_n)=v_p(f_{p,k+1}(n))$ for such $n$. By Theorem \ref{thm2} we conclude that:
\begin{itemize}
\item if $v_p(q_p(n_k))=0$ then $f'(n_0)\not\equiv 0 \pmod{p}$ and there exists a unique $n_{k+1}$ modulo $p^{k+1}$ for which $p^{k+1}\mid{a_{n_{k+1}}}$ and $n_{k+1}\equiv n_k \pmod{p^k}$;
\item if $v_p(q_p(n_k))>0$ and $p^{k+1}\mid{a_{n_k}}$ then $f'(n_0)\equiv 0 \pmod{p}$ and $p^{k+1}\mid{a_{n_{k+1}}}$ for any $n_{k+1}\in S$ satisfying $n_{k+1}\equiv n_k \pmod{p^k}$;
\item if $v_p(q_p(n_k))>0$ and $p^{k+1}\nmid{a_{n_k}}$ then $f'(n_0)\not\equiv 0 \pmod{p}$ and $p^{k+1}\nmid{a_{n_{k+1}}}$ for any $n_{k+1}\in S$ satisfying $n_{k+1}\equiv n_k \pmod{p^k}$.
\end{itemize}

Let us consider the case $v_p(q_p(n_k))=0$ and write $n_{k+1} = n_k + p^k t_{k+1}$, where $t_{k+1}\in\Z$. Use (\ref{eq1}) for $f=f_{p,k+1}$, $x_0=n_k$ and $X=n_k+p^k t_{k+1}$ to obtain the sequence of congruences:
\begin{equation}\label{eq0}
\begin{split}
& 0 \equiv f_{p,k+1}(n_k) + p^k t_{k+1} f'_{p,k+1}(n_k) \pmod{p^{k+1}} \\
& t_{k+1} q_p(n_k) \equiv t_{k+1} f'_{p,k+1}(n_k) \equiv -\frac{f_{p,k+1}(n_k)}{p^k} \pmod{p} \\
& t_{k+1} \equiv -\frac{f_{p,k+1}(n_k)}{p^k q_p(n_k)} \pmod{p} \\
& n_{k+1} = n_k + p^k t_{k+1} \equiv n_k -\frac{f_{p,k+1}(n_k)}{q_p(n_k)} \pmod{p^{k+1}}.
\end{split}
\end{equation}

Since $f_{p,k+1}(n_k) \equiv \frac{a_{n_k}}{g_{p,k+1}(n_k)} \pmod{p^{k+1}}$, we get $n_{k+1} \equiv n_k-\frac{a_{n_k}}{g_{p,k+1}(n_k)q_p(n_k)} \pmod{p^{k+1}}$.

Assume now that $k=1$ and $v_p(q_p(n_1))=0$. By simple induction on $l\in\N_+$ we obtain that the inequality $v_p(a_n)\geq{l}$ has a unique solution $n_l$ modulo $p^l$ with condition $n_l\equiv n_1 \pmod{p}$ and this solution satisfies the congruence 
\begin{equation*}
n_l \equiv n_{l-1}-\frac{f_{p,l}(n_{l-1})}{q_p(n_1)}\pmod{p^l}
\end{equation*}
for $l>1$.

Certainly the statement is true for $l=1$. Now, assume that there exists a unique $n_l$ modulo $p^l$ satisfying the conditions in the statement. Note that
\begin{equation*}
q_p(n_l) \equiv f'_{p,2}(n_l) \equiv f'_{p,2}(n_1) \equiv q_p(n_1) \pmod{p}.
\end{equation*}
Since $v_p(q_p(n_1))=0$ then there exists a unique $n_{l+1}$ modulo $p^{l+1}$ such that $p^{l+1} \mid a_{n_{l+1}}$ and $n_{l+1} \equiv n_l \pmod{p^l}$. Additionaly the congruences (\ref{eq0}) showed that $n_{l+1} \equiv n_l-\frac{a_{n_l}}{g_{p,l+1}(n_l)q_p(n_1)} \pmod{p^{l+1}}$.
\end{proof}

\subsection{Connection between pseudo-polynomial decomposition modulo $p$ and $p$-adic continuous functions approximated by polynomials over $\Z$}\label{subsec2.2}

It is worth to see that if a sequence $(a_n)_{n\in\N}$ has a pseudo-polynomial decomposition modulo $p$ then there exist functions $f_{p,\infty}\in\cal{C}(\Z_p,\Z_p)$ and $g_{p,\infty}:\N\rightarrow\Z_p\bs p\Z_p$ such that $a_n = f_{p,\infty}(n)g_{p,\infty}(n)$ for each nonnegative integer $n$.

Let $k\in\N_+$, $p\in\bbb P$ and put:
\begin{equation*}
\cal{P}_{p,k} = \{F\mbox{ polynomial function on }\Z /p^k\Z: \exists_{G:\N\rightarrow\Z\bs p\Z} \forall_{n\in\N}:\;a_n \equiv F(n)G(n) \pmod{p^k}\}.
\end{equation*}
For each $k\in\N_+$ set $\cal{P}_{p,k}$ is finite (because there are only finitely many functions on $\Z /p^k\Z$) and nonempty (by existence of pseudo-polynomial decomposition). We have the map $\psi_{k+1}:\cal{P}_{p,k+1}\ni F\mapsto F\pmod{p^k}\in\cal{P}_{p,k}$ of reduction modulo $p^k$.

\begin{thm}\label{thm3}
Let $(S_k)_{k\in\N_+}$ be a sequence of finite nonempty sets with mappings $\psi_{k+1}:S_{k+1}\rightarrow S_k, k\in\N_+$. Then there exists a sequence $(s_k)_{k\in\N_+}$ such that $s_k\in S_k$ and $\psi_{k+1}(s_{k+1})=s_k$ for each $k\in\N_+$.
\end{thm}

\begin{proof}
See \cite[p. 13]{Ser}.
\end{proof}

By Theorem \ref{thm3} there exist polynomial functions $F_{p,k}: \Z /p^k\Z \rightarrow \Z /p^k\Z$ such that $F_{p,k+1}\equiv F_{p,k}\pmod{p^k}, k\in\N_+$. Furthermore, for each $k\in\N_+$ there exists $f_{p,k}\in\Z[X]$ such that $f_{p,k}\equiv F_{p,k}\pmod{p^k}$. As a result, if $k_1\leq k_2$ then $f_{p,k_2}(x)-f_{p,k_1}(x)\equiv 0 \pmod{p^{k_1}}, x\in\Z_p$, or in other words $|f_{p,k_2}(x)-f_{p,k_1}(x)|_p \leq p^{-k_1}$. Since $\Z[X]$ is a subset of $\cal{C}(\Z_p,\Z_p)$ and $\cal{C}(\Z_p,\Z_p)$ with metric $d_{\sup}(f,g)=\sup_{x\in\Z_p}|f(x)-g(x)|_p$ is a complete metric space, hence the sequence $(f_{p,k})_{k\in\N_+}\subset\cal{C}(\Z_p,\Z_p)$ is uniformly convergent to a continuous function $f_{p,\infty}:\Z_p\rightarrow\Z_p$. Each polynomial from $\Z_p[X]$ is a $p$-adic contraction, so $f_{p,\infty} = \lim_{k\rightarrow +\infty} f_{p,k}$ is a $p$-adic contraction, too. From the definition of $\cal{P}_{p,k}$ there exists $g_{p,k}:N\rightarrow\Z\bs p\Z$ such that $a_n \equiv f_{p,k}(n)g_{p,k}(n) \pmod{p^k}, n\in\N$. As a consequence of our reasoning we have that if $a_n\neq 0$ for some $n\in\N$ then $v_p(a_n)=v_p(f_{p,k}(n))$ for sufficiently large $k$ and by continuity of $p$-adic valuation (with respect to $p$-adic norm) $v_p(a_n)=v_p(f_{p,\infty}(n))$. If $a_n=0$ for some $n\in\N$ then $v_p(f_{p,k}(n))\geq k$ and going with $k$ to $+\infty$ we obtain $v_p(f_{p,\infty})=+\infty$, which means that $f_{p,\infty}(n)=0$. Then we define $g_{p,\infty}$ by the formula:
\begin{equation*}
g_{p,\infty}(n)=
\begin{cases}
\frac{a_n}{f_{p,\infty}(n)}, & \mbox{ if } a_n\neq 0
\\ 1, & \mbox{ if } a_n=0
\end{cases}.
\end{equation*}

Conversely, assume that $a_n = f_{p,\infty}(n)g_{p,\infty}(n)$ for some $f_{p,\infty}\in\overline{\Z_p[X]}\subset\cal{C}(\Z_p,\Z_p)$ (with respect to metric $d_{\sup}$) and $g_{p,\infty}:\N\rightarrow\Z_p\bs p\Z_p$. For $k\geq 2$, let $f_{p,k}\in\Z_p[X]$ be such that $d_{\sup}(f_{p,k},f_{p,\infty})\leq p^{-k}$ (replacing coefficients of $f_{p,k}$ by integers congruent to them modulo $p^k$ we can assume that $f_{p,k}\in\Z[X]$). Because there are only finitely many polynomial functions on $\Z /p\Z$ then we can choose a sequence of polynomials $(f_{p,k})_{k\in\N_2}\subset\Z_p[X]$ such that $f'_{p,k_1}(n)\equiv f'_{p,k_2} \pmod{p}$ for each $n\in\Z$ and $k_1,k_2\geq 2$.  For $k\geq 2$ and $n\in\N$ we put $g_{p,k}(n)$ as an integer congruent to $g_{p,\infty}(n)$ modulo $p^k$. Finally, we obtain pseudo-polynomial decomposition $(f_{p,k}, g_{p,k})_{k\in\N_2}$ modulo $p$ of the sequence $(a_n)_{n\in\N}$.

In particular, if $p\mid a_{n_1}$ for some $n_1\in\N$ and $p\nmid f'_{p,2}(n_1)$ (or equivalently $v_p\left(\frac{1}{p}\left(\frac{a_{n_1+p}}{g_{p,2}(n_1+p)}-\frac{a_{n_1}}{g_{p,2}(n_1)}\right)\right)=0$) then there exists a unique $n_{\infty}\in\Z_p$ such that $n_{\infty}\equiv n_1 \pmod{p}$ and $f_{p,\infty}(n_{\infty})=0$. Indeed, Theorem \ref{thm1} gives us existence and uniqueness of $n_k$ modulo $p^k$ such that $n_k\equiv n_1 \pmod{p}$ and $p^k\mid a_{n_k}$ for each $k\in\N_+$. Hence $p^k\mid f_{p,\infty}(n_k)$, which means that $|f_{p,\infty}(n_k)|_p\leq p^{-k}$. If $k_1\leq k_2$ then by uniqueness of $n_{k_1}$ modulo $p^{k_1}$ we have $n_{k_1}\equiv n_{k_2} \pmod{p^{k_1}}$, or in other words $|n_{k_1}-n_{k_2}|_p \leq p^{-k_1}$. We thus conclude that $(n_k)_{k\in\N_+}$ is a Cauchy sequence and by completeness of $\Z_p$ this sequence is convergent to some $n_{\infty}$. By continuity of $f_{p,\infty}$, $|f_{p,\infty}(n_{\infty})|_p=\lim_{k\rightarrow +\infty} |f_{p,\infty}(n_k)|_p = 0$. For all $k\in\N_+$, $n_{\infty}\equiv n_k \pmod{p^k}$ and by uniqueness of $n_k$ modulo $p^k$, such $n_{\infty}$, that $n_{\infty}\equiv n_1 \pmod{p}$ and $f_{p,\infty}(n_{\infty})=0$, is unique.

Let us note that if $f:\Z_p\rightarrow\Z_p$ is a $p$-adic contraction then $f$ can be approximated uniformly on $\Z_p$ by the sequence of polynomials $(f_{p,k})_{k\in\N_+}\subset\Z[X]$ such that $f_{p,k}(n)  = f(n), n\in\{0,1,2,...,p^k-1\}$ for each $k\in\N_+$. Indeed, for each $x\in\Z_p$ we have
\begin{equation*}
\begin{split}
& |f_{p,k}(x)-f(x)|_p = |(f_{p,k}(x)-f_{p,k}(x\pmod{p^k}))-(f_{p,k}(x\pmod{p^k})-f(x))|_p\leq \\
& \leq \max\{|f_{p,k}(x)-f_{p,k}(x\pmod{p^k})|_p, |f(x\pmod{p^k})-f(x)|_p\} \leq p^{-k}.
\end{split}
\end{equation*}
As a result, the closure of the rings $\Z[X]$ and $\Z_p[X]$ in the space $\cal{C}(\Z_p,\Z_p)$ with metric $d_{\sup}$ is the set of all $p$-adic contractions $f:\Z_p\rightarrow\Z_p$.

In particular, if a sequence $(a_n)_{n\in\N}$ is such that $(a_n \pmod{p^k})_{n\in\N}$ is periodic of period $p^k$ for each $k\in\N_+$ then $(a_n)_{n\in\N}$ (as a function mapping $\N$ to $\Z_p$) is a $p$-adic contraction (if $|n-m|_p = p^{-k}$ then $p^k\mid n-m$ and $p^k\mid a_n-a_m$, which means that $|a_n-a_m|_p\leq p^{-k}$). Since $\N$ is dense in $\Z_p$, the sequence $(a_n)_{n\in\N}$ can be extended to a function $f\in\cal{C}(\Z_p,\Z_p)$. Thus there exists a sequence $(f_{p,k})_{k\in\N_+}\in\Z[X]^{\N_+}$ converging uniformly to $f$ on $\Z_p$. Because of finiteness of the set of polynomial functions on $\Z /p\Z$ we can choose polynomials $f_{p,k}$, $k\in\N_2$, such that $f'_{p,k_1}(n)\equiv f'_{p,k_2} \pmod{p}$ for each $n\in\Z$ and $k_1,k_2\geq 2$. Finally, the sequence $(f_{p,k}, 1)_{k\in\N_2}$ is a pseudo-polynomial decomposition modulo $p$ of the sequence $(a_n)_{n\in\N}$ (where $1$ means the function mapping each nonnegative integer $n$ to $1$).

The following three examples show that there is no connection between approximability of a given function $f:\Z_p\rightarrow\Z_p$ by polynomials over $\Z_p$ and its differentiability.

\begin{ex}
Let us note that each $x\in\Z_p$ can be written uniquely as a series $\sum_{j=0}^{+\infty} a_j(x)p^j$, where $a_j(x)\in\{0,1,...,p-1\}$ for each $j\in\N$ (see \cite{Cas}). Moreover, if $x\neq 0$ then $v_p(x)$ is the least index $j$ such that $a_j(x)\neq 0$. Let $p$ be an odd prime number and let us consider a function $f:\Z_p\rightarrow\Z_p$ given by the formula
\begin{equation*}
f(x)=
\begin{cases}
\frac{x}{a_{v_p(x)}(x)}, & \mbox{ if } x\neq 0, \\
0, & \mbox{ if } x=0.
\end{cases}
\end{equation*}
The function $f$ is a $p$-adic contraction. First we see that $|f(x)|_p=|x|_p$ for each $x\in\Z_p$. Let $k\in\N$ and $x,y\in\Z_p$ be such that $|x-y|_p=p^{-k}$. Let us write $x=\sum_{j=0}^{+\infty} a_j(x)p^j$ and $y=\sum_{j=0}^{+\infty} a_j(y)p^j$, where  $a_j(x),a_j(y)\in\{0,1,...,p-1\}$ for $j\in\N$. Then $k$ is the least index $j$ such that $a_j(x)\neq a_j(y)$. Let us consider two cases.
\begin{enumerate}
\item Assume first that $v_p(x)=v_p(y)<k$. Then $a_{v_p(x)}(x)=a_{v_p(y)}(y)\neq 0$ and as a result $|f(x)-f(y)|_p=\left|\frac{x-y}{a_{v_p(x)}(x)}\right|_p=|x-y|_p=p^{-k}$.
\item Assume now that one of the numbers $x,y$ has $p$-adic valuation equal to $k$. Without loss of generality we set $v_p(x)=k$. Then $v_p(y)\geq k$ and as a consequence $v_p(f(x))=k$ and $v_p(f(y))\geq k$. We thus have $v_p(f(x)-f(y))\leq k$ or in other words $|f(x)-f(y)|_p\leq p^{-k}$.
\end{enumerate}
Hence $f\in\overline{\Z[X]}$. On the other hand, the function $f$ is not differentiable at $0$, because $\lim_{n\rightarrow +\infty}\frac{f(ap^n)-f(0)}{ap^n}=\lim_{n\rightarrow +\infty}\frac{p^n}{ap^n}=\frac{1}{a}$ for each $a\in\{1,...,p-1\}$.
\end{ex}

\begin{ex}
Let $p$ be an arbitrary prime number and function $f:\Z_p\rightarrow\Z_p$ be given by the formula
\begin{equation*}
f(x)=
\begin{cases}
(-1)^{v_p(x)}x, & \mbox{ if } x\neq 0, \\
0, & \mbox{ if } x=0.
\end{cases}
\end{equation*}
The function $f$ is a $p$-adic contraction (in fact, $f$ is an isometry, i.e. $|f(x)-f(y)|_p=|x-y|_p$ for any $x,y\in\Z_p$). First we see that $|f(x)|_p=|x|_p$ for each $x\in\Z_p$. Let $k\in\N$ and $x,y\in\Z_p$ be such that $|x-y|_p=p^{-k}$. Let us consider two cases.
\begin{enumerate}
\item Assume first that $v_p(x)=v_p(y)$. Then $|f(x)-f(y)|_p=\left|(-1)^{v_p(x)}(x-y)\right|_p=|x-y|_p=p^{-k}$.
\item Assume now that $v_p(x)\neq v_p(y)$. Then one of the numbers $x,y$ has $p$-adic valuation equal to $k$ and the second one has $p$-adic valuation greater than $k$. Without loss of generality we set $v_p(x)=k$ and $v_p(y)>k$. As a consequence $v_p(f(x))=k$ and $v_p(f(y))>k$. We thus have $v_p(f(x)-f(y))=k$ or in other words $|f(x)-f(y)|_p=p^{-k}$.
\end{enumerate}
Hence $f\in\overline{\Z[X]}$. On the other hand, the function $f$ is not differentiable at $0$, because $\lim_{n\rightarrow +\infty}\frac{f(p^{2n+r})-f(0)}{p^{2n+r}}=\lim_{n\rightarrow +\infty}\frac{(-1)^rp^{2n+r}}{p^{2n+r}}=(-1)^r$ for $r\in\{0,1\}$.
\end{ex}

\begin{ex}
Let a function $f:\Z_p\rightarrow\Z_p$ be given by the formula
\begin{equation*}
f(x)=
\begin{cases}
\frac{x}{p}, & \mbox{ if } p\mid x, \\
x, & \mbox{ if } p\nmid x.
\end{cases}
\end{equation*}
Obviously, the function $f$ is differentiable at each point $x\in\Z_p$ and its derivative is equal to $\frac{1}{p}$ for $x\in p\Z_p$ and $1$ otherwise. However, if $x,y\in p\Z_p$ then $|f(x)-f(y)|_p=\left|\frac{x-y}{p}\right|_p=p\cdot |x-y|_p$. Hence $f$ is not a $p$-adic contraction, which means that $f\not\in\overline{\Z[X]}$. 
\end{ex}

\subsection{Hensel's lemma for exponential function}

Let us fix a prime number $p$, an integer $a$ and consider now exponential function $f:\;\N\ni n\mapsto a^n\in\Z$. In general it is not a $p$-adic continuous function, but if $p\nmid a$ and $m\in\N$ is fixed then the function $g:\;\N\ni n\mapsto a^{n(p-1)+m}\in\Z$ is continuous. Indeed, by Fermat's little theorem $a^{p-1} = 1+pb$ for some $b\in\Z$ and
\begin{equation*}
a^{n(p-1)} = \sum_{j=0}^n {n \choose j}(pb)^j = \sum_{j=0}^{n} \frac{p^j b^j}{j!} \prod_{i=0}^{j-1} (n-i) = \sum_{j=0}^{+\infty} \frac{p^j b^j}{j!} \prod_{i=0}^{j-1} (n-i).
\end{equation*}

Note that $v_p(\frac{p^j b^j}{j!}) \geq j - v_p(j!) = j - \frac{j-s_p(j)}{p-1} \geq \frac{p-2}{p-1} j \geq 0$ (we use Legendre's formula $v_p(j!) = \frac{j-s_p(j)}{p-1}, j\in\N$, see \cite{Le}) and $v_p(\prod_{i=0}^{j-1} (n-i)) \geq \lfloor\frac{j}{p}\rfloor$ (between $j$ consecutive integers there are at least $\lfloor\frac{j}{p}\rfloor$ integers divisible by $p$). This suggests to define
\begin{equation*}
f_{p,k} = \sum_{j=0}^{kp-1} \frac{p^j b^j}{j!} \prod_{i=0}^{j-1} (X-i) \in\Z_p[X]
\end{equation*}
for $k\geq 2$. Then $a^{n(p-1)} \equiv f_{p,k}(n) \pmod{p^k}$ and
\begin{equation*}
f'_{p,k}(n) = \sum_{j=0}^{kp-1} \frac{p^j b^j}{j!} \sum_{s=0}^{j-1}\prod_{i=0, i\neq s}^{j-1} (n-i) \equiv \sum_{j=0}^{2p-1} \frac{p^j b^j}{j!} \sum_{s=0}^{j-1}\prod_{i=0, i\neq s}^{j-1} (n-i) \pmod{p}.
\end{equation*}

Hence $(f_{p,k},1)_{k\in\N_2}$, where $1$ means the function defined on $\N$ constantly equal to $1$, is a pseudo-polynomial decomposition of $(a^{n(p-1)})_{n\in\N}$. Additionaly, the formula $a^{x(p-1)} = f_{p,\infty}(x) = \sum_{j=0}^{+\infty} \frac{p^j b^j}{j!} \prod_{i=0}^{j-1} (x-i)$ extends the function $\N\ni n\mapsto a^{n(p-1)+m}\in\Z$ to a continuous function defined on $\Z_p$. However, this function has one more property. Namely, if $p^k\mid x-y$ then $p^{k+1}\mid a^{x(p-1)}-a^{y(p-1)}$, or in other words $|a^{x(p-1)}-a^{y(p-1)}|_p \leq \frac{1}{p}|x-y|_p$. Because $\N$ is dense in $\Z_p$, thus it suffices to show this property for $x,y\in\N$, where $x>y$. Let $x = y + p^k t$ for some positive integer $t$ not divisible by $p$. Then $a^{x(p-1)}-a^{y(p-1)} = a^{y(p-1)}(a^{tp^k(p-1)}-1)$ and by Euler's theorem $p^{k+1}\mid (a^t)^{p^k(p-1)}-1$.

The mentioned property is a motivation to state an analogue of Hensel's lemma for exponential function.

\begin{thm}[Hensel's lemma for exponential function]\label{thm4}
Let $p$ be a prime number and $k$ be a positive integer. Let $a,c$ be integers not divisible by $p$. Let $m,n_k$ be nonnegative integers such that $n_k<m$ and $s = v_p(a^{p^k(p-1)}-1)$ (by Euler's theorem $s \geq k+1$).
\begin{itemize}
\item If $p^s\mid a^{n_k(p-1)+m}-c$ then for each positive integer $l$ there exists a unique $n_{k+l}$ modulo $p^{k+l}$ such that $n_{k+l} \equiv n_k \pmod{p^k}$ and\newline $p^{s+l}\mid a^{n_{k+l}(p-1)+m}-c$.
\item If $p^s\nmid a^{n_k(p-1)+m}-c$ then $p^s\nmid a^{n(p-1)+m}-c$ for all $n\equiv n_k \pmod{p^k}$.
\end{itemize}
In particular, if $m=0$ and $c=1$ then for each positive integer $l$ there exists a unique $n_{k+l}$ modulo $p^{k+l}$ such that $p^{s+l}\mid a^{n_{k+l}(p-1)}-1$.
\end{thm}

\begin{proof}
The second case of the statement is very easy. If $n\equiv n_k \pmod{p^k}$ then $(a^{n(p-1)+m}-c)-(a^{n_k(p-1)+m}-c) = a^{n_k(p-1)}(a^{(n-n_k)(p-1)}-1)$. Since $p\nmid a$ and $p^k\mid n-n_k$, thus $a^{p^k(p-1)}-1\mid a^{(n-n_k)(p-1)}-1$ and as a consequence $p^s\mid a^{n_k(p-1)}(a^{(n-n_k)(p-1)}-1)$. Because $p^s\nmid a^{n_k(p-1)+m}-c$, hence $p^s\nmid a^{n(p-1)+m}-c$.

Now we prove the first case of the statement of our theorem.

First, we show by induction on $l\in\N$ that $v_p(a^{p^{k+l}(p-1)}-1) = s + l$. The induction hypothesis is obviously true for $l=0$. Assume that $a^{p^{k+l}(p-1)} = 1+p^{s+l}t_l$ and $p\nmid t_l$. Then
\begin{equation*}
\begin{split}
& a^{p^{k+l+1}(p-1)} = (a^{p^{k+l}(p-1)})^p = (1+p^{s+l}t_l)^p = \sum_{j=0}^p {p \choose j} p^{j(s+l)}t_l^j = 
\\ & = 1 + p^{s+l+1}t_l + p^{s+l+2}u_l = 1 + p^{s+l+1}(t_l + pu_l),
\end{split}
\end{equation*}
where $u_l\in\Z$. Assume now that $l\in\N$ and $p^{s+l}\mid a^{n_{k+l}(p-1)+m}-c$. Write $n_{k+l+1} = n_{k+l} + p^{k+l}w$ and $a^{n_{k+l}(p-1)+m}-c \equiv p^{s+l}z \pmod{p^{s+l+1}}$. Then we obtain the sequence of equivalent congruences:
\begin{equation*}
\begin{split}
& a^{n_{k+l+1}(p-1)+m}-c \equiv 0 \pmod{p^{s+l+1}} \\
\iff\quad & a^{(n_{k+l}+p^{k+l}w)(p-1)+m} \equiv c \pmod{p^{s+l+1}} \\
\iff\quad & a^{wp^{k+l}(p-1)}a^{n_{k+l}(p-1)+m} \equiv c \pmod{p^{s+l+1}} \\
\iff\quad & (1+p^{s+l}t_l)^w(c+p^{s+l}z) \equiv c \pmod{p^{s+l+1}} \\
\iff\quad & (1+wp^{s+l}t_l)(c+p^{s+l}z) \equiv c \pmod{p^{s+l+1}} \\
\iff\quad & wp^{s+l}t_l(c+p^{s+l}z) \equiv -p^{s+l}z \pmod{p^{s+l+1}} \\
\iff\quad & wt_lc \equiv -z \pmod{p}
\end{split}
\end{equation*}

Since $p\nmid t_{k+l}c$, thus the last congruence has exactly one solution $w$ modulo $p$, which means that the first congruence has exactly one solution $n_{k+l+1}$ modulo $p^{k+l+1}$ such that $n_{k+l+1} \equiv n_k \pmod{p^k}$.
\end{proof}

\pagebreak

\section{Arithmetic properties of sequences ${\bf a}\in\cal{R}$}\label{sec3}

The main point of this paper is to investigate arithmetic properties of sequence of derangements and to generalize this properties to some class of sequences. This section is devoted to the family $\cal{R}$ of sequences ${\bf a}$ given by the recurrence $a_0 = h_1(0), a_n = f(n)a_{n-1} + h_1(n)h_2(n)^n, n>0$, where $f,h_1,h_2\in\Z[X]$. In order to emphasize the polynomials $f,h_1,h_2$ appearing in this recurrence, we will denote ${\bf a}={\bf a}(f,h_1,h_2)$. Sequences from the class $\cal{R}$ are natural generalization of the sequence of derangements, which is obtained for $f=X, h_1=1, h_2=-1$. Note that many well-known sequences belong to this class:
\begin{itemize}
\item if $f,h_2 = 1, h_1 = c\in\Z$ then $(a_n)_{n\in\N} = (cn)_{n\in\N}$ is an arithmetic progression;
\item if $f = q\in\Z, h_1 = c\in\Z, h_2 = 0$ then $(a_n)_{n\in\N} = (cq^n)_{n\in\N}$ is a geometric progression;
\item if $f = 1, h_1 = c\in\Z, h_2 = q\in\Z$ then $(a_n)_{n\in\N} = (\sum_{j=0}^n cq^j)_{n\in\N}$ is a sequence of partial sums of a geometric progression;
\item if $f = X, h_1 = 1, h_2 = 0$ then $(a_n)_{n\in\N} = (n!)_{n\in\N}$ is the sequence of factorials;
\item if $f = 2X+l, l\in\{0,1\}, h_1 = 1, h_2 = 0$ then $(a_n)_{n\in\N} = ((2n+l)!!)_{n\in\N}$ is the sequence of double factorials.
\end{itemize}

One can easily obtain the closed formula $a_n = \sum_{j=0}^n h_1(j)h_2(j)^j \prod_{i=j+1}^n f(i)$ for $n\in\N$. 

A particular subclasses of $\cal{R}$ are class $\cal{R}'$ of sequences for which $h_2 = 1$ and class $\cal{R}''$ of sequences for which $h_2 = -1$. It is worth to note that if  ${\bf a}={\bf a}(f,h_1,-1)=(a_n)_{n\in\N}$ then $a_n = (-1)^n \widetilde{a}_n, n\in\N$, where $\widetilde{\bf a}=\widetilde{\bf a}(-f,h_1,1)=(\widetilde{a}_n)_{n\in\N}$. Certainly the equality is true for $n=0$. Now, assume that $a_{n-1} = (-1)^{n-1} \widetilde{a}_{n-1}$ for $n>0$. Then $a_n = f(n)a_{n-1} + (-1)^nh_1(n) = (-1)^n(-f(n)\widetilde{a}_{n-1} + h_1(n)) = (-1)^n \widetilde{a}_n$. The sequence $\widetilde{\bf a}$ we will call \emph{associated} to the sequence ${\bf a}$.

Hence the study of such properties as: periodicity, $p$-adic valuations, divisors, boundedness for sequences from class $\cal{R}''$ comes down to study of this properties for sequences from $\cal{R}'$.

\subsection{Periodicity modulo $d$ and $p$-adic valuations}\label{subsec3.1}

\subsubsection{\bf Periodicity modulo $d$, when $h_2=1$ or $h_2=-1$ and $d\mid f(n_0)$ for some $n_0\in\N$}\label{subsubsec3.1.1}

Assume that ${\bf a}={\bf a}(f,h_1,1)$. Let $d\in\N_+$ be such that $d\mid f(n_0)$ for some $n_0\in\N$. Then for each $n\geq n_0$ we have
\begin{equation}\label{eq2}
\begin{split}
a_n = & \sum_{j=0}^n h_1(j) \prod_{i=j+1}^n f(i) \equiv \sum_{j=n-d+1}^n h_1(j) \prod_{i=j+1}^n f(i) \\
= & \sum_{j=0}^{d-1} h_1(n-j) \prod_{i=0}^{j-1} f(n-i) \pmod{d}.
\end{split}
\end{equation}

Making reduction modulo $d$ we can skip the summands from $0$th to $(n-d)$th because if $0\leq j\leq n-d$ then among (at least $d$) numbers $j+1, j+2, ..., n$ there is such number $i_0$ that $d\mid f(i_0)$ and $d\mid\prod_{i=j+1}^n f(i)$. If $n_0\leq n<d$ then $f(n_0)$ appears in the product $\prod_{i=j+1}^n f(i)$ for $n-d+1\leq j<0$, thus this product has no influence on the value $a_n \pmod d$.

Let us define $f_d = \sum_{j=0}^{d-1} h_1(X-j) \prod_{i=0}^{j-1} f(X-i) \in\Z[X]$. Then the equation (\ref{eq2}) takes the form $a_n \equiv f_d(n) \pmod{d}, n\geq n_0$ and because of periodicity modulo $d$ of any polynomial we conclude that the sequence $(a_n \pmod d)_{n\in\N_{n_0}}$ is periodic of period $d$. One can ask the natural question now: Is $d$ the basic period of $(a_n \pmod d)_{n\in\N_{n_0}}$? If $\gcd (d,a_n)=1$ for some $n\in\N$, $f=b_1X+b_0$, where $\gcd (d,b_1)=1$ and $h_1=c\in\Z$ then the answer is positive.

\begin{prop}\label{prop1}
Let as assume ${\bf a}={\bf a}(b_1X+b_0,c,1)$, where $b_0,b_1,c\in\Z$. Let $d\in\N_+$ divide $b_1n_0+b_0$ for some $n_0\in\N$ and $\gcd (d,a_{n_1})=1$ for some $n_1\geq n_0$. Then the sequence $(a_n \pmod d)_{n\in\N_{n_0}}$ has the basic period divisible by $\frac{d}{\gcd(d,b_1)}$. In particular, if $\gcd(d,b_1)=1$ then the basic period is equal to $d$.
\end{prop}

\begin{proof}
Denote the basic period of $(a_n \pmod d)_{n\in\N_{n_0}}$ by $per$. We thus obtain the following chain of equivalences:
\begin{equation*}
\begin{split}
& a_{n_1+per+1} \equiv a_{n_1+1} \pmod{d} \\
\iff\quad & (b_1(n_1+per+1)+b_0)a_{n_1+per}+c \equiv (b_1(n_1+1)+b_0)a_{n_1}+c \pmod{d} \\
\iff\quad & (b_1(n_1+per+1)+b_0)a_{n_1} \equiv (b_1(n_1+1)+b_0)a_{n_1} \pmod{d} \\
\iff\quad & b_1(n_1+per+1)+b_0 \equiv b_1(n_1+1)+b_0 \pmod{d} \\
\iff\quad & b_1(n_1+per+1) \equiv b_1(n_1+1) \pmod{d} \\
\iff\quad & n_1+per+1 \equiv n_1+1 \left(\mbox{mod }\frac{d}{\gcd(d,b_1)}\right) \\
\iff\quad & per \equiv 0 \left(\mbox{mod }\frac{d}{\gcd(d,b_1)}\right).
\end{split}
\end{equation*}
Our proposition is proved.
\end{proof}

\begin{ex}
Let $c$ be as in the statement of {\rm Proposition \ref{prop1}} and suppose that $|c|>1$. Then by simple induction one can prove that $c\mid a_n$ for all $n\in\N$. In other words, the sequence $(a_n \pmod c)_{n\in\N}$ is constant and equal to $0$. This means that the assumption $\gcd (d,a_{n_1})=1$ for some $n_1\geq n_0$ in Proposition \ref{prop1} is essential.
\end{ex}

\begin{ex}
Let us fix $c,d\in\Z$ such that $d>1$. Consider the sequence given by the formula $a_0 = c, a_n = (n^{\varphi (d)+1}-n)a_{n-1} + c, n>0$ (by $\varphi$ we mean Euler's totient function). By Euler's theorem $d\mid n^{\varphi (d)+1}-n$ for all $n\in\N$, hence the sequence $(a_n \pmod c)_{n\in\N}$ is constant and equal to $c\pmod d$. Thus the assumption $f=b_1X+b_0$ in {\rm Proposition \ref{prop1}} is essential.
\end{ex}

\begin{ex}
Let us consider the sequence ${\bf a}={\bf a}(X,-aX+a,1)$, where $a\in\Z$. It is very easy to prove that $a_n = a$ for all $n\in\N$. Hence the assumption $h_1 = c$ in {\rm Proposition \ref{prop1}} is essential.
\end{ex}

The last example shows us that Proposition \ref{prop1} is no longer true if we replace a constant polynomial $h_1 = c$ with an affine polynomial $h_1 = c_1X + c_0$. However, we can modify Proposition \ref{prop1} and get the following.

\begin{prop}
Let ${\bf a}={\bf a}(b_1X+b_0,c_1X+c_0,1)$, where $b_0,b_1,c_0,c_1\in\Z$. Let $d\in\N_+$ be such that $d\mid b_1n_0+b_0$ for some $n_0\in\N$. Let us assume that $\gcd (d,b_1a_{n_1}+c_1)=1$ for some $n_1\geq n_0$. Then the sequence $(a_n \pmod d)_{n\in\N_{n_0}}$ has the basic period equal to $d$.
\end{prop}

\begin{proof}
Denote the basic period of $(a_n \pmod d)_{n\in\N_{n_0}}$ by $per$. We thus obtain the following chain of equivalences:
\begin{equation*}
\begin{split}
& a_{n_1+per+1} \equiv a_{n_1+1} \pmod{d} \\
\iff\quad & (b_1(n_1+per+1)+b_0)a_{n_1+per} + c_1(n_1+per+1) + c_0 \\
 & \equiv (b_1(n_1+1)+b_0)a_{n_1} + c_1(n_1+1) + c_0 \pmod{d} \\
\iff\quad & (b_1(n_1+per+1)+b_0)a_{n_1} + c_1per \equiv (b_1(n_1+1)+b_0)a_{n_1} \pmod{d} \\
\iff\quad & b_1a_{n_1}per + c_1per \equiv 0 \pmod{d} \\
\iff\quad & (b_1a_{n_1}+c_1)per \equiv 0 \pmod{d}.
\end{split}
\end{equation*}
Since $\gcd (d,b_1a_{n_1}+c_1)=1$, we get $d\mid per$ and we are done.
\end{proof}

Let ${\bf a}={\bf a}(f,h_1,-1)$. Let us assume that $d\in\N_+$ is such that $d\mid f(n_0)$ for some $n_0\in\N$. Then $a_n = (-1)^n \widetilde{a}_n$ for each $n\in\N$. This means that $a_n \equiv (-1)^n \widetilde{a}_n \pmod{d}$ for all $n\in\N$. Since $\widetilde{\bf a}=\widetilde{\bf a}(-f,h_1,1)$ and $d\mid -f(n_0)$ we deduce that the sequence  $(\widetilde{a}_n \pmod d)_{n\in\N_{n_0}}$ is periodic. If we denote the basic period of $(\widetilde{a}_n \pmod d)_{n\in\N_{n_0}}$ by $per>1$ then $(a_n \pmod d)_{n\in\N_{n_0}}$ has the basic period equal to:
\begin{itemize}
\item $per$, when $2\mid per$;
\item $2per$, when $2\nmid per$.
\end{itemize}

If the sequence  $(\widetilde{a}_n \pmod d)_{n\in\N_{n_0}}$ is constant then $(a_n \pmod d)_{n\in\N_{n_0}}$ has the basic period equal to:
\begin{itemize}
\item $1$, when $d\mid a_n, n\geq n_0$ or $d=2$;
\item $2$, otherwise.
\end{itemize}

In particular, since the associated sequence $(\widetilde{D}_n)_{n\in\N}$ to the sequence of derangements satisfies the assumptions of the Proposition \ref{prop1}, hence $(\widetilde{D}_n \pmod d)_{n\in\N}$ has the basic period $d$ for arbitrary $d\in\N_+$ and as a result the basic period of $(D_n \pmod d)_{n\in\N}$ is equal to:
\begin{itemize}
\item $d$, when $2\mid d$;
\item $2d$, when $2\nmid d$.
\end{itemize}

\begin{rem}\label{rem1}
It is worth to recall a well known fact that if $(a_n)_{n\in\N}$ is a sequence of integers, $d_1, d_2$ are two coprime positive integers and the sequences $(a_n\pmod{d_1})_{n\in\N},\newline (a_n\pmod{d_2})_{n\in\N}$ are periodic with basic periods $per_1, per_2$ respectively then the sequence $(a_n\pmod{d_1d_2})_{n\in\N}$ is periodic with basic period $\lcm (per_1, per_2)$.
\end{rem}

\subsubsection{\bf $p$-adic valuations of numbers $a_n, n\in\N$, when $h_2=1$ or $h_2=-1$ and $p\mid f(n_0)$ for some $n_0\in\N$}

Let us fix a prime number $p$ and assume that ${\bf a}={\bf a}(f,h_1,1)$. If $p\mid f(n_0)$ for some $n_0\in\N$ then for each $k\in\N_+$ and $n\geq n_0+(k-1)p$ we have
\begin{equation}\label{eq3}
\begin{split}
a_n & = \sum_{j=0}^n h_1(j) \prod_{i=j+1}^n f(i) \equiv \sum_{j=n-kp+1}^n h_1(j) \prod_{i=j+1}^n f(i) \\ 
& = \sum_{j=0}^{kp-1} h_1(n-j) \prod_{i=0}^{j-1} f(n-i)\equiv f_{p,k}(n) \pmod{p^k},
\end{split}
\end{equation} 
where $f_{p,k} = \sum_{j=0}^{kp-1} h_1(X-j) \prod_{i=0}^{j-1} f(X-i) \in\Z[X]$. We can skip the summands from $0$th to $(n-kp)$th because if $0\leq j\leq n-d$ then among (at least $kp$) numbers $j+1, j+2, ..., n$ there are at least $k$ numbers congruent to $n_0$ modulo $p$, thus $p$ divides at least $k$ factors in product $\prod_{i=j+1}^n f(i)$. Additionaly, if $f(n_0)=0$ and $n_0\leq n < n_0+(k-1)p$ then $\prod_{i=j+1}^n f(i)=0$ for $j<n_0$. Hence the congruence (\ref{eq3}) is satisfied for $n\geq n_0$.

What is more,
\begin{equation*}
\begin{split}
& f'_{p,k}(n) = \sum_{j=0}^{kp-1} [h'_1(n-j) \prod_{i=0}^{j-1} f(n-i) + h_1(n-j) \sum_{s=0}^{j-1}(f'(n-s)\prod_{i=0, i\neq s}^{j-1} f(n-i))] \\
\equiv & \sum_{j=0}^{2p-1} [h'_1(n-j) \prod_{i=0}^{j-1} f(n-i) + h_1(n-j) \sum_{s=0}^{j-1}(f'(n-s)\prod_{i=0, i\neq s}^{j-1} f(n-i))] \pmod{p}.
\end{split}
\end{equation*}
Hence $(f_{p,k},1)_{k\in\N_2}$ is a pseudo-polynomial decomposition modulo $p$ of ${\bf a}$ and thus we can use Theorem \ref{thm1} to obtain the criterion for behavior of $p$-adic valuation of the number $a_n$.

\begin{thm}\label{thm5}
Assume that ${\bf a}={\bf a}(f,h_1,1)$. Let $p$ be a prime number and $k\in\N_+$. Let $f(n_0)=0$ (respectively $p\mid f(n_0)$) and $n_k\geq n_0$ (respectively $n_k\geq n_0+kp$) be such that $p^k\mid a_{n_k}$.
\begin{itemize}
\item If $v_p(a_{n_k+p}-a_{n_k})=1$ then there exists a unique $n_{k+1}$ modulo $p^{k+1}$ such that $n_{k+1}\equiv n_k \pmod{p^k}$ and $p^{k+1}\mid{a_n}$ for all $n\geq n_0$ (respectively $n\geq n_0 +{kp}$) congruent to $n_{k+1}$ modulo $p^{k+1}$. Moreover, $n_{k+1} \equiv n_k-\frac{pa_{n_k}}{a_{n_k+p}-a_{n_k}} \pmod{p^{k+1}}$.
\item If $v_p(a_{n_k+p}-a_{n_k})>1$ and $p^{k+1}\mid{a_{n_k}}$ then $p^{k+1}\mid{a_n}$ for any $n$ satisfying $n\equiv n_k \pmod{p^k}$ and $n\geq n_0$ (respectively $n\geq n_0+{kp}$).
\item If $v_p(a_{n_k+p}-a_{n_k})>1$ and $p^{k+1}\nmid{a_{n_k}}$ then $p^{k+1}\nmid{a_n}$ for any $n$ satisfying $n\equiv n_k \pmod{p^k}$ and $n\geq n_0$ (respectively $n\geq n_0+{kp}$).
\end{itemize}
In particular, if $k=1$, $p\mid{a_{n_1}}$ and $v_p(a_{n_1+p}-a_{n_1})=1$ then for any $l\in\mathbb{N}_+$ there exists a unique $n_l$ modulo $p^l$ such that $n_l\equiv n_1 \pmod{p}$ and $v_p(a_n)\geq{l}$ for all $n\geq n_0$ (respectively $n\geq n_0 +(l-1)p$) congruent to $n_l$ modulo $p^l$. Moreover, $n_l$ satisfies the congruence $n_l \equiv n_{l-1}-\frac{pa_{n_{l-1}}}{a_{n_1+p}-a_{n_1}} \pmod{p^l}$ for $l>1$.
\end{thm}

\begin{proof}
Note that $q_p(n_k) = \frac{1}{p}(a_{n_k+p}-a_{n_k})$ ($q_p(n_k)$ is as in the statement of Theorem \ref{thm1}), which implies that $v_p(a_{n_k+p}-a_{n_k})=v_p(q_p(n_k))+1$. We use Theorem \ref{thm1} for the set $S = \{n\in\N: n\geq n_0\}$ (respectively $S = \{n\in\N: n\geq n_0+kp\}$) and get the result.
\end{proof}

Let us observe that if ${\bf a}\in\cal{R}''$ then $v_p(a_n) = v_p(\widetilde{a}_n)$. Hence it suffices to apply Theorem \ref{thm5} for the sequence $\widetilde{\bf a}$ in order to obtain the description of $p$-adic valuation of numbers $a_n$, $n\in\N$.

\subsubsection{\bf Prime divisors and $p$-adic valuations of the sequence of derangements}\label{subsubsec3.1.3}

Theorem \ref{thm5} can be used to describe $p$-adic valuations of numbers of derangements, but we will study these numbers more precisely. Namely,
\begin{equation*}
\begin{split}
D_n & = nD_{n-1} + (-1)^n = D_{n-1} + (n-1)D_{n-1} + (-1)^n \\
& = (n-1)D_{n-2} + (-1)^{n-1} + (n-1)D_{n-1} + (-1)^n = (n-1)(D_{n-2}+D_{n-1})
\end{split}
\end{equation*}

for $n>1$. We thus have $n-1\mid D_n$ for $n\in\N$ and as a consequence $v_p(n-1)\leq v_p(D_n)$ for each prime $p$. Let us define two sets:
\begin{equation*}
\cal{A} = \{ p\in\bbb{P}: v_p(n-1) = v_p(D_n) \mbox{ for all } n\in\N\},\mbox{ } \cal{B} = \bbb{P}\bs\cal{A}.
\end{equation*}

Denote $E_n = \frac{D_n}{n-1} = D_{n-2}+D_{n-1}, n>1$. Hence it suffices to study $p$-adic valuations of the sequence $(E_n)_{n\in\N_2}$ for $p\in\cal{B}$. Firstly note that the set $\cal{B}$ is infinite.

\begin{prop}
The set $\cal{B}$ is infinite.
\end{prop}

\begin{proof}
Assume that $\cal{B} = \{p_1, ..., p_s\}$. Since $((-1)^nD_n \pmod p_i)_{n\in\N}$ has period $p_i$ for each $i\in\{1,...,s\}$, thus $((-1)^nE_n \pmod p_i)_{n\in\N_2}$ has period $p_i$, too. Because $E_2=1$, hence $p_1...p_s\nmid E_{p_1...p_sm+2}$ for all $m\in\N$. $\cal{B}$ is the set of all prime divisors of numbers $E_n, n>1$ and $E_n>0$ (because $D_n>0$ for $n\neq 1$), so $E_{p_1...p_sm+2}=1$ for all $m\in\N$. On the other hand, $\frac{E_n}{(n-2)!n} = \frac{D_n}{n!} = \sum_{j=0}^n \frac{(-1)^j}{j!} \rightarrow e^{-1}$, when $n\rightarrow +\infty$. This fact implies that $E_n\rightarrow +\infty$, when $n\rightarrow +\infty$, and this is a contradiction.
\end{proof}

For a given prime number $p$ it is easy to verify if $p\in\cal{A}$. Because of periodicity of the sequence $((-1)^nE_n \pmod p)_{n\in\N_2}$ it suffices to check that $p$ divides none of the numbers $E_n$, $n\in\{2,...,p+1\}$. The first numbers in $\cal{A}$ are $2,5,7,17,19,23,29$. Numerical computations show that among all prime numbers less than $10^6$ there are $28990$ numbers which belong to $\cal{A}$, while $49508$ primes belong to $\cal{B}$. This means that primes less than $10^6$ contained in $\cal{A}$ are approx. $37\%$ of all primes less than $10^6$. However, we are not able to prove that the set $\cal{A}$ is infinite.

\begin{con}\label{con1}
The set $\cal{A}$ is infinite. Moreover, $\lim_{n\rightarrow +\infty} \frac{\sharp (\cal{A}\cap\{1,...,n\})}{\sharp (\bbb{P}\cap\{1,...,n\})} = \frac{1}{e}$.
\end{con}

The following heuristic reasoning allows us to claim the second statement in the conjecture above. If we fix a prime number $p$ and choose by random a sequence $(a_n)_{n\in\N}$ such that the sequence of remainders $(a_n\pmod{p})_{n\in\N}$ has period $p$ then the probability that $p$ does not divide any term of this sequence is equal to $\left(1-\frac{1}{p}\right)^p$. When $p\rightarrow +\infty$ then this probability tends to $\frac{1}{e}$. Note that $p\in\cal{A}$ if and only if $p$ does not divide any number $\widetilde{E}_n$, $n\geq 2$ and the sequence $(\widetilde{E}_n\pmod{p})_{n\in\N_2}$ is periodic of period $p$. Therefore we suppose that the probability that $p\in\cal{A}$ tends to $\frac{1}{e}$, when $p\rightarrow +\infty$ and hence the asymptotic density of the set $\cal{A}$ in the set $\bbb{P}$ is equal to $\frac{1}{e}$.

Now we are obtaining a pseudo-polynomial decomposition modulo $p$ of the sequence $(E_n)_{n\in\N_2}$. Note, that $(-1)^nE_n = (-1)^nD_{n-2} + (-1)^nD_{n-1} = \widetilde{D}_{n-2} - \widetilde{D}_{n-1}$, $n>1$. Let $f_{p,k} = \sum_{j=0}^{kp-1} (-1)^j \prod_{i=0}^{j-1} (X-i)$, $k>1$. Then $(f_{p,k},1)_{k\in\N_2}$ is a pseudo-polynomial decomposition modulo $p$ of $(\widetilde{D}_n)_{n\in\N}$ (recall that $\widetilde{D}_0 = 1, \widetilde{D}_n = -n\widetilde{D}_{n-1}+1$, $n>0$). Hence $(f_{p,k}(X-2)-f_{p,k}(X-1),1)_{k\in\N_2}$ is a pseudo-polynomial decomposition modulo $p$ of $((-1)^nE_n)_{n\in\N_2}$ and $(f_{p,k}(X-2)-f_{p,k}(X-1),(-1)^n)_{k\in\N_2}$ (where $(-1)^n$ means the function which maps a nonnegative integer $n$ to $(-1)^n$) is a pseudo-polynomial decomposition modulo $p$ of $(E_n)_{n\in\N_2}$.

\begin{rem}\label{rem2}
We can define $E_0, E_1$ so that $(E_n)_{n\in\N}$ has a pseudo-polynomial decomposition. The sequence of functions $(f_{p,k})_{k\in\N_2}$ converges uniformly to the function $f_{p,\infty} = \sum_{j=0}^{+\infty} (-1)^j \prod_{i=0}^{j-1} (X-i)$ on $\Z_p$ and thus (see Section \ref{subsec2.2}):
\begin{equation*}
E_n = (-1)^n (f_{p,\infty}(n-2) - f_{p,\infty}(n-1)), n\geq 2,
\end{equation*}
so there must be:
\begin{equation*}
\begin{split}
& E_0 = f_{p,\infty}(-2) - f_{p,\infty}(-1) = \sum_{j=0}^{+\infty} (j+1)! - \sum_{j=0}^{+\infty} j! = -1, \\
& E_1 = f_{p,\infty}(0) - f_{p,\infty}(-1) = 1 - \sum_{j=0}^{+\infty} j! = -\sum_{j=1}^{+\infty} j!.
\end{split}
\end{equation*}

It is worth to remark that $E_0 = \frac{D_0}{-1}$. Thus the definition of $E_0$ coincides with the definition $E_n = \frac{D_n}{n-1}$ for $n\geq 2$.

One can observe that $E_1\notin\Z$. Indeed, if $E_1\in\Z$ then the sequence of remainders $(E_1\pmod{n!})_{n\in\N}$ or $(-E_1\pmod{n!})_{n\in\N}$ is ultimately constant. However, for $n>1$ we have
\begin{equation*}
E_1\pmod{n!} = n! - \sum_{j=1}^{n-1} j! > n! - (n-1)(n-1)! = (n-1)!
\end{equation*}
and
\begin{equation*}
-E_1\pmod{n!} = \sum_{j=1}^{n-1} j!,
\end{equation*}
which leads to a contradiction.
\end{rem}

\begin{thm}\label{thm6}
Let $p\in\cal{B}, k\in\N_+$ and $n_k\in\N$, $n_k\geq 2$ be such that $p^k\mid\frac{D_{n_k}}{n_k-1}$. Let us define $\widehat{q}_p(n_k) = \frac{1}{p}\left(\frac{D_{n_k+p}}{n_k+p-1}+\frac{D_{n_k}}{n_k-1}\right)$.
\begin{itemize}
\item If $p\nmid\widehat{q}_p(n_k)$ then there exists a unique $n_{k+1}$ modulo $p^{k+1}$ such that $n_{k+1}\equiv n_k \pmod{p^k}$ and $p^{k+1}\mid\frac{D_n}{n-1}$ for all $n\geq 2$ congruent to $n_{k+1}$ modulo $p^{k+1}$. What is more, $n_{k+1} \equiv n_k+\frac{D_{n_k}}{(n_k-1)\widehat{q}_p(n_k)} \pmod{p^{k+1}}$.
\item If $p\mid\widehat{q}_p(n_k)$ and $p^{k+1}\mid\frac{D_{n_k}}{n_k-1}$ then $p^{k+1}\mid\frac{D_{n}}{n-1}$ for all $n$ satisfying $n\equiv n_k \pmod{p^k}$ and $n\geq 2$.
\item If $p\mid\widehat{q}_p(n_k)$ and $p^{k+1}\nmid\frac{D_{n_k}}{n_k-1}$ then $p^{k+1}\nmid\frac{D_{n}}{n-1}$ for any $n$ satisfying $n\equiv n_k \pmod{p^k}$ and $n\geq 2$.
\end{itemize}
In particular, if $k=1$, $p\mid \frac{D_{n_1}}{n_1-1}$ and $p^2\nmid\left(\frac{D_{n_1+p}}{n_1+p-1}+\frac{D_{n_1}}{n_1-1}\right)$ then for any $l\in\mathbb{N}_+$ there exists a unique $n_l$ modulo $p^l$ such that $n_l\equiv n_1 \pmod{p}$ and $v_p\left(\frac{D_n}{n-1}\right)\geq{l}$ for all $n\geq 2$ congruent to $n_l$ modulo $p^l$. Moreover, $n_l$ satisfies the congruence $n_l \equiv n_{l-1}+\frac{D_{n_{l-1}}}{(n_{l-1}-1)\widehat{q}_p(n_1)} \pmod{p^l}$ for $l>1$.
\end{thm}

\begin{proof}
Let $q_p(n_k)$ be as specified in Theorem \ref{thm1}. Note that
\begin{equation*}
\begin{split}
q_p(n_k) = & \frac{1}{p}\left(\frac{D_{n_k+p}}{(-1)^{n_k+p}(n_k+p-1)} - \frac{D_{n_k}}{(-1)^{n_k}(n_k-1)}\right) \\
= & \frac{(-1)^{n_k+p}}{p}\left(\frac{D_{n_k+p}}{n_k+p-1}+\frac{D_{n_k}}{n_k-1}\right) = -(-1)^{n_k}\widehat{q}_p(n_k),
\end{split}
\end{equation*}
where the equalities above hold, because $p\neq 2$ ($2\in\cal{A}$). Thus $v_p(\widehat{q}_p(n_k))=v_p(q_p(n_k))$ and if $v_p(\widehat{q}_p(n_k))=0$ then $n_{k+1} \equiv n_k-\frac{\frac{D_{n_k}}{n_k-1}}{(-1)^{n_k}q_p(n_k)} = n_k+\frac{D_{n_k}}{(n_k-1)\widehat{q}_p(n_k)} \pmod{p^{k+1}}$.
\end{proof}

According to numerical computations based on the theorem above, among primes less than $10^6$ there are three primes $p$ with the property that there exists an $n_1\geq 2$ such that $p\mid \frac{D_{n_1}}{n_1-1}$ and $p\mid\widehat{q}_p(n_1)$. Namely, they are:
\begin{itemize}
\item $p=2633$ with $n_1=1578$,
\item $p=429943$ with $n_1=317291$,
\item $p=480143$ with $n_1=121716$.
\end{itemize}

In addition, if a tuple $(p,n_1)$ is one of these tree tuples above then $v_p\left(\frac{D_{n_1}}{n_1-1}\right) = 1$. Therefore, by Theorem \ref{thm6}, $v_p\left(\frac{D_n}{n-1}\right) = 1$ for all integers $n\geq 2$ congruent to $n_1$ modulo $p$. Since $2633\mid \frac{D_n}{n-1}$ if and only if $n\equiv 1578\pmod{2633}$, thus the 2633-adic valuation of numbers $\frac{D_n}{n-1}, n\geq 2$, is bounded by 1. 

For $p=480143$ we have $429943\mid \frac{D_n}{n-1}$ if and only if $n\equiv 172017, 223393, 317291\pmod{429943}$ and $429943\nmid \widehat{q}_{429943}(172017), \widehat{q}_{429943}(223393)$, and $480143\mid \frac{D_n}{n-1}$ if and only if $n\equiv 121716, 265745\pmod{480143}$ and $480143\nmid \widehat{q}_{480143}(265745)$, so the 429943-adic valuation and 480143-adic valuation of numbers $\frac{D_n}{n-1}, n\geq 2$, are unbounded.

Hence it is not true that the $p$-adic valuation of numbers $\frac{D_n}{n-1}, n\geq 2$, is unbounded for all $p\in\cal{B}$. In the light of these results it is natural to ask the following questions:

\begin{que}
Are there infinitely many primes $p$ with the property that there exists $n_1\in\N_2$ such that $p\mid \frac{D_{n_1}}{n_1-1}$ and $p\mid\widehat{q}_p(n_1)$?
\end{que}

\begin{que}
Are there infinitely many primes $p\in\cal{B}$ such that the set $$\left\{v_p\left(\frac{D_n}{n-1}\right): n\in\N_2\right\}$$ is finite?
\end{que}

\subsubsection{\bf Periodicity modulo $d$ in case when $d$ divides $f(n_0)$ for some integer $n_0$ and $h_2$ is arbitrary}

\begin{prop}\label{prop4}
Let us consider a sequence ${\bf a}(f,h_1,h_2)$. Let $d\in\N_+$ and $n_0\in\N$ be such that $d\mid f(n_0)$ and $n_0\geq (p+1)v_p(d)-1$ for all primes $p$. Then the sequence $(a_n\pmod{d})_{n\in\N_{n_0}}$ is periodic of period $\lcm\{p_i^{k_i}(p_i-1): i\in\{1,2,...,s\}\}$, where $d = p_1^{k_1}\cdot... \cdot p_s^{k_s}$ is the factorization of the number $d$. In particular, the number $d\prod_{i=0}^s (p_i-1)$ is a period of $(a_n\pmod{d})_{n\in\N_{n_0}}$.
\end{prop}

\begin{proof}
By Remark \ref{rem1}, it suffices to prove this fact for $d=p^k$, where $p$ is a prime number and $k$ is a positive integer.

For each $n\geq n_0$ we have congruence similar to (\ref{eq2}).
\begin{equation}\label{eq7}
\begin{split}
a_n & = \sum_{j=0}^n h_1(j)h_2(j)^j \prod_{i=j+1}^n f(i) \equiv \sum_{j=n-kp+1}^n h_1(j)h_2(j)^j \prod_{i=j+1}^n f(i) \\
& = \sum_{j=0}^{kp-1} h_1(n-j)h_2(n-j)^{n-j} \prod_{i=0}^{j-1} f(n-i) \pmod{p^k}.
\end{split}
\end{equation}
Since $n\geq n_0$, $n_0\geq k(p+1)-1$ and $j\leq kp-1$, thus $n-j\geq k$ and if $p\mid h_2(n-j)$ then $p^k\mid h_2(n-j)^{n-j}$. Let us define $\cal{N} = \{n\in\N: p\nmid h_2(n)\}$. Then the congruence (\ref{eq7}) takes the form
\begin{equation*}
a_n \equiv \sum_{0\leq j\leq kp-1, n-j\in\cal{N}} h_1(n-j)h_2(n-j)^{n-j} \prod_{i=0}^{j-1} f(n-i) \pmod{p^k}.
\end{equation*}
If $n_1 \equiv n_2 \pmod{p^k(p-1)}$ then $n_1-j\in\cal{N}$ if and only if $n_2-j\in\cal{N}$. Since $n_1 \equiv n_2 \pmod{p^k}$, thus $h_1(n_1-j)\equiv h_1(n_2-j) \pmod{p^k}$, $h_2(n_1-j)\equiv h_2(n_2-j) \pmod{p^k}$ and $f(n_1-j)\equiv f(n_2-j) \pmod{p^k}$ for any $j\in\N$. Since $n_1 \equiv n_2 \pmod{p^{k-1}(p-1)}$, hence by Euler's theorem $h_2(n_1-j)^{n_1-j} \equiv h_2(n_2-j)^{n_2-j} \pmod{p^k}$ for $j$ such that $n_1-j\in\cal{N}$. Finally
\begin{equation}
\begin{split}
a_{n_1} & \equiv \sum_{0\leq j\leq kp-1, n_1-j\in\cal{N}} h_1(n_1-j)h_2(n_1-j)^{n_1-j} \prod_{i=0}^{j-1} f(n_1-i) \equiv \\
& \equiv \sum_{0\leq j\leq kp-1, n_2-j\in\cal{N}} h_1(n_2-j)h_2(n_2-j)^{n_2-j} \prod_{i=0}^{j-1} f(n_2-i) \equiv a_{n_2} \pmod{p^k},
\end{split}
\end{equation}
which means that $p^k(p-1)$ is a period of the sequence $(a_n\pmod{p^k})_{n\in\N_{n_0}}$.
\end{proof}

\begin{ex}
Let ${\bf a}={\bf a}(b,h_1,d)$, where $h_1\in\Z[X]$, $b,d\in\Z$ and $\gcd(b,d) = 1$. Then the sequence $(a_n\pmod{d})_{n\in\N} = (b^na_0)_{n\in\N}$ has period $\lambda(d) = \lcm\{p_i^{k_i-1}(p_i-1): i\in\{1,2,...,s\}\}$, where $d = p_1^{k_1}\cdot... \cdot p_s^{k_s}$ is the factorization of the number $d$ and $\lambda$ is Carmichael's function. This means that in general $\lcm\{p_i^{k_i}(p_i-1): i\in\{1,2,...,s\}\}$ is not the basic period of $(a_n\pmod{d})_{n\in\N_{n_0}}$.
\end{ex}

\begin{ex}
The sequence ${\bf a}={\bf a}(X,1,2)$ is the example that $\lcm\{p_i^{k_i}(p_i-1): i\in\{1,2,...,s\}\}$ may be the basic period of $(a_n\pmod{d})_{n\in\N_{n_0}}$, where $d = p_1^{k_1}\cdot... \cdot p_s^{k_s}$ is the factorization of the number $d$. Namely, if $d=225=3^2\cdot 5^2$ then the basic period of $(a_n\pmod{225})_{n\in\N}$ is equal to $900=\lcm\{3^2\cdot 2, 5^2\cdot 4\}$.
\end{ex}

\subsubsection{\bf Periodicity modulo $p^k$ in case when $p$ does not divide $f(n)$ for any integer $n$}

Now we are considering periodicity modulo $p^k$ of sequences given by the relation $a_0 = h_1(0), a_n = f(n)a_{n-1} + h_1(n)h_2(n)^n, n>0$, where $h_2\in\Z[X]$ is an arbitrary polynomial and a prime number $p$ does not divide $f(n)$ for any integer $n$.

\begin{prop}\label{prop5}
Let $p$ be a prime number, $k$ be a positive integer and a sequence ${\bf a}={\bf a}(f,h_1,h_2)$. Assume that $p\nmid f(n)$ for any integer $n$. Then the sequence $(a_n\pmod{p^k})_{n\in\N_{k-1}}$ is periodic of period of the form $tp^k(p-1)$, where $t\in\{1, 2, 3, ..., p^k\}$. Moreover, if $p\nmid h_2(n)$ for any integer $n$ then the sequence $(a_n\pmod{p^k})_{n\in\N}$ is periodic of period of the form as above. If $h_2 = -1$ and $p\neq 2$ then there is a period of the form $2tp^k$, $t\in\{1, 2, 3, ..., p^k\}$. If $h_2 = -1$ and $p=2$ or $h_2 = 1$ then there is a period of the form $tp^k$, $t\in\{1, 2, 3, ..., p^k\}$.
\end{prop}

\begin{proof}
Let us consider the numbers $a_{k-1}$, $a_{p^k(p-1)+k-1}$, $a_{2p^k(p-1)+k-1}$, ..., $a_{p^kp^k(p-1)+k-1}$. By pigeon hole principle there are $s_1, s_2\in\{0,1, 2, 3, ..., p^k\}$, $s_1<s_2$ such that $a_{s_1p^k(p-1)+k-1} \equiv a_{s_2p^k(p-1)+k-1} \pmod{p^k}$. Let us put $t = s_2 - s_1$. We will show that
\begin{equation}\label{eq4}
a_n \equiv a_{tp^k(p-1)+n} \pmod{p^k} \mbox{ for } n\geq k-1.
\end{equation}

The congruence (\ref{eq4}) is satisfied for $n = s_1p^k(p-1)+k-1$. Assume now that $k-1 \leq n < s_1p^k(p-1)+k-1$ and (\ref{eq4}) is satisfied for $n+1$. Then we have
\begin{equation*}
\begin{split}
& f(n+1)a_n + h_1(n+1)h_2(n+1)^{n+1} \equiv f(tp^k(p-1)+n+1)a_{tp^k(p-1)+n} + \\
& + h_1(tp^k(p-1)+n+1)h_2(tp^k(p-1)+n+1)^{tp^k(p-1)+n+1} \pmod{p^k}.
\end{split}
\end{equation*}

Let us assume that $p\mid h_2(n+1)$. Then $p\mid h_2(tp^k(p-1)+n+1)$. Since $n+1\geq k$, we infer that $p^k\mid h_2(n+1)^{n+1}$, $p^k\mid h_2(tp^k(p-1)+n+1)^{tp^k(p-1)+n+1}$. Suppose now that $p\nmid h_2(n+1)$. Then $p\nmid h_2(tp^k(p-1)+n+1)$ and by Euler's theorem we obtain the following
\begin{equation*}
h_2(tp^k(p-1)+n+1)^{tp^k(p-1)+n+1} \equiv h_2(tp^k(p-1)+n+1)^{n+1} \equiv h_2(n+1)^{n+1} \pmod{p^k}.
\end{equation*}

Finally, we get
\begin{equation*}
h_1(n+1)h_2(n+1)^{n+1} \equiv h_1(tp^k(p-1)+n+1)h_2(tp^k(p-1)+n+1)^{tp^k(p-1)+n+1} \pmod{p^k}. 
\end{equation*}

As a consequence we have
\begin{equation*}
f(n+1)a_n \equiv f(tp^k(p-1)+n+1)a_{tp^k(p-1)+n} \equiv f(n+1)a_{tp^k(p-1)+n} \pmod{p^k}.
\end{equation*}
Moreover, the fact that $p\nmid f(n+1)$ implies that $a_n \equiv a_{tp^k(p-1)+n} \pmod{p^k}$.

Let us note that if $p\nmid h_2(n)$ for any integer $n$ then the consideration above allows us to conclude that
\begin{equation*}
a_n \equiv a_{tp^k(p-1)+n} \pmod{p^k} \mbox{ for any } n\in\{0,1,...,s_1p^k(p-1)+k-1\}.
\end{equation*}

Similarly we prove (\ref{eq4}) for $n > s_1p^k(p-1)+k-1$.

The proof in the cases $h_2 = -1$, $h_2 = 1$ runs in the same way: we consider the numbers $a_{k-1}$, $a_{2p^k+k-1}$, $a_{4p^k+k-1}$, ..., $a_{2p^kp^k+k-1}$ (respectively $a_{k-1}$, $a_{p^k+k-1}$, $a_{2p^k+k-1}$, ..., $a_{p^kp^k+k-1}$) and we use the fact that
\begin{equation*}
\begin{split}
& (-1)^{n+1}h_1(n+1) \equiv (-1)^{2tp^k+n+1}h_1(2tp^k+n+1) \pmod{p^k} \\
& \mbox{(respectively } h_1(n+1) \equiv h_1(tp^k+n+1) \pmod{p^k}\mbox{).}
\end{split}
\end{equation*}
\end{proof}

\begin{ex}
Let $f = X^2 - 2$, $h_1 = 1$, $h_2 = 2$, $p=5$ and $k=1$. Then the basic period of the sequence $(a_n\pmod{5})_{n\in\N}$ is equal to $100 = 5^2 \cdot 4$. Hence it is possible that the basic period of the sequence $(a_n\pmod{p^k})_{n\in\N_{k-1}}$ is exactly $p^{2k}(p-1)$.
\end{ex}

\begin{ex}
Let $f = X^2 - 2$, $h_1 = 1$, $h_2 = 3$, $p=3$ and $k\in\N_+$. Then the sequence $(a_n\pmod{3^k})_{n\in\N_{k-2}}$ is not periodic while $(a_n\pmod{3^k})_{n\in\N_{k-1}}$ is. Indeed, $2\cdot 3^k t$, where $t$ is some number from the set $\{1, 2, 3, ..., 3^k\}$, is the period of $(a_n\pmod{3^k})_{n\in\N_{k-1}}$. This fact implies that if the sequence $(a_n\pmod{3^k})_{n\in\N_{k-2}}$ is periodic then there exists its period of the form $2\cdot 3^k u$, where $u\in\N_+$. Since $a_{k-1+2\cdot 3^k u} \equiv a_{k-1} \pmod{3^k}$, hence
\begin{equation*}
\begin{split}
& ((k-1+2\cdot 3^k u)^2 - 2)a_{k-2+2\cdot 3^k u} + 3^{k-1+2\cdot 3^k u} \equiv ((k-1)^2 - 2)a_{k-2} + 3^{k-1} \pmod{3^k} \\
& ((k-1)^2 - 2)a_{k-2+2\cdot 3^k u} \equiv ((k-1)^2 - 2)a_{k-2} + 3^{k-1} \pmod{3^k}.
\end{split}
\end{equation*}
Thus $a_{k-2+2\cdot 3^k u} \not\equiv a_{k-2} \pmod{3^k}$ - a contradiction.

In addition, the basic period of $(a_n\pmod{9})_{n\in\N_+}$ equals 18. Hence it is possible that the basic period of the sequence $(a_n\pmod{p^k})_{n\in\N_{k-1}}$ is equal to $p^k(p-1)$.
\end{ex}

\begin{ex}
Let $f = 11X^4 + 7$, $h_1 = 1$, $h_2 = 7$, $p=5$ and $k=2$. Then the sequence $(a_n\pmod{25})_{n\in\N}$ has basic period equal to $500 = 5^3 \cdot 4$. This means that the basic period of the sequence $(a_n\pmod{p^k})_{n\in\N_{k-1}}$ can be strictly greater than $p^k(p-1)$ and smaller than $p^{2k}(p-1)$.
\end{ex}

\begin{ex}
Let $f = X^2 + 1$, $h_1 = 1$, $h_2 = -1$, $p=3$ and $k=1$. Then the basic period of the sequence $(a_n\pmod{3})_{n\in\N}$ is equal to $18 = 3^2 \cdot 2$. Hence it is possible that the basic period of the sequence $(a_n\pmod{p^k})_{n\in\N_{k-1}}$, where $h_2 = -1$, is exactly $2p^{2k}$.
\end{ex}

\begin{ex}
Let $f = X^2 + 1$, $h_1 = h_2 = 1$, $p=3$ and $k=1$. Then the basic period of the sequence $(a_n\pmod{3})_{n\in\N}$ is equal to 9. Hence it is possible that the basic period of the sequence $(a_n\pmod{p^k})_{n\in\N_{k-1}}$, where $h_2 = -1$, is exactly $p^{2k}$.
\end{ex}

\begin{ex}
Let $f = h_1 = 1$, $h_2 = b$ and $b \neq 1$. Then $a_n = \frac{b^{n+1}-1}{b-1}$, $n\in\N$. Assume that $p$ is such a prime number that $p\nmid b$. We consider two cases.
\begin{enumerate}
\item We assume that $b \not\equiv 1 \pmod{p}$. In this case $p$ must be odd (because each integer is divisible by $2$ or is congruent to $1$ modulo $2$). Hence the multiplicative group $(\Z /p^k\Z)^*$ is cyclic of order $p^{k-1}(p-1)$ and all the elements of order being a power of $p$ are exactly these ones which are congruent to 1 modulo $p$. This means that the basic period of the sequence $(a_n\pmod{p^k})_{n\in\N_{k-1}}$ can be any positive integer dividing $p^k(p-1)$ on condition that it is not equal to some power of $p$.
\item We suppose that $b \equiv 1 \pmod{p}$ and put $s = v_p(b-1)$. If $p$ is odd or $p=2$ and $s>1$ then we prove by induction that $p^k$ is the order of $b$ in $(\Z /p^{k+s}\Z)^*$ for each $k\in\N$ ($p^k$ is the least positive integer $r$ with property that $p^{k+s}\mid b^r-1$ - compare with the proof of Theorem \ref{thm4}). Hence the sequence $(b^{n+1}-1\pmod{p^{k+s}})_{n\in\N}$ has the basic period $p^k$ and thus the sequence $(\frac{b^{n+1}-1}{b-1}\pmod{p^k})_{n\in\N}$ has the same basic period. If $p=2$ and $s=1$ then we prove similarly that $2^{k-1}$ is the order of $b$ in $(\Z /2^{k+1}\Z)^*$ for each $k\geq 2$ and $2$ is the order of $b$ in $(\Z /4\Z)^*$. Hence the sequence $(b^{n+1}-1\pmod{2^{k+1}})_{n\in\N}$ has the basic period $2^{k-1}$ for $k\geq 2$ and $2$ for $k=1$. Thus the sequence $(\frac{b^{n+1}-1}{b-1}\pmod{2^k})_{n\in\N}$ has the same basic period.
\end{enumerate}
\end{ex}

Let us observe that if $d = p_1^{k_1}\cdot ...\cdot p_s^{k_s}$ and $n_0 = \max\{k_i - 1: i\in\{1,2,...,s\}\}$ then knowing the basic periods of the sequences $(a_n\pmod{p_i^{k_i}})_{n\in\N_{k_i-1}}$, $i\in\{1,2,...,s\}$, we can use Remark \ref{rem1} to compute the basic period of the sequence $(a_n\pmod{d})_{n\in\N_{n_0}}$. Namely, if $per_i$ is the basic period of $(a_n\pmod{p_i^{k_i}})_{n\in\N_{k_i-1}}$ then the basic period of $(a_n\pmod{d})_{n\in\N_{n_0}}$ equals lcm$\{per_i: i\in\{1,2,...,s\}\}$.

Proposition \ref{prop5} states only existence of some period of the sequence \newline$(a_n\pmod{p^k})_{n\in\N_{k-1}}$ of the form $tp^k(p-1)$, where $t\in\{1,2,...,p^k\}$. We would like to obtain an exact formula for some period of this sequence. What is more, the examples above showed that number $tp^k(p-1)$, where $t\in\{1,2,...,p^k\}$, can be the basic period of $(a_n\pmod{p^k})_{n\in\N_{k-1}}$ when $t$ is a power of $p$. Now we will show that $p^{3k-1}(p-1)$ is a period of $(a_n\pmod{p^k})_{n\in\N_{k-1}}$ and this fact together with Proposition \ref{prop5} gives the form of basic period of this sequence.

\begin{prop}\label{prop6}
If $p$ is a prime number which does not divide $f(n)$ for any integer $n$ and $k$ is a positive integer then $p^{3k-1}(p-1)$ is the period of the sequence\newline $(a_n\pmod{p^k})_{n\in\N_{k-1}}$. Moreover, if $p$ does not divide $h_2(n)$ for any integer $n$ then $p^{3k-1}(p-1)$ is the period of the sequence $(a_n\pmod{p^k})_{n\in\N}$.
\end{prop}

\begin{proof}
Let us define $\cal{P}_p = \{n\in\N: p\mid h_2(n)\}$ and $\cal{N}_p = \N\bs\cal{P}_p$. For $n\geq k-1$, we have:

\begin{equation}\label{eq5}
\begin{split}
& a_n = \sum_{j=0}^n h_1(n-j)h_2(n-j)^{n-j} \prod_{i=0}^{j-1} f(n-i) \\
= & \sum_{t=0}^{p^{2k-1}(p-1)-1} \sum_{s=0}^{\left\lfloor \frac{n-t}{p^{2k-1}(p-1)} \right\rfloor} h_1(n-sp^{2k-1}(p-1)-t) \times \\
& \times h_2(n-sp^{2k-1}(p-1)-t)^{n-sp^{2k-1}(p-1)-t}\prod_{i=0}^{sp^{2k-1}(p-1)+t-1} f(n-i) \\
\equiv & \sum_{t\in\cal{P}_p, 0\leq t<k} h_1(t)h_2(t)^t \prod_{i=0}^{n-t-1} f(n-i) + \\
& + \sum_{0\leq t<p^{2k-1}(p-1), n-t\in\cal{N}_p} \left(1+\left\lfloor\frac{n-t}{p^{2k-1}(p-1)}\right\rfloor\right)\times \\
& \times h_1(n-t)h_2(n-t)^{n-t} \prod_{i=0}^{t-1} f(n-i)\pmod{p^k}.
\end{split}
\end{equation}

The congruence above is true because for each $j\in\N$ we have
\begin{equation*}
\begin{split}
\prod_{i=0}^{j-1} f(n-i) & = \prod_{r=0}^{p^k-1} \prod_{s=0}^{\left\lfloor\frac{j-1-r}{p^k}\right\rfloor} f(n-sp^k-r) \\
& \equiv \prod_{r=0}^{p^k-1} f(n-r)^{\left\lfloor\frac{j-1-r}{p^k}\right\rfloor} \pmod{p^k}.
\end{split}
\end{equation*}

If $j_1 \equiv j_2 \pmod{p^{2k-1}(p-1)}$ then $p^{k-1}(p-1)\mid \left\lfloor\frac{j_1-1-r}{p^k}\right\rfloor - \left\lfloor\frac{j_2-1-r}{p^k}\right\rfloor$ and by Euler's theorem we obtain
\begin{equation*}
\prod_{i=0}^{j_1-1} f(n-i) \equiv \prod_{i=0}^{j_2-1} f(n-i) \pmod{p^k}.
\end{equation*}

Using Euler's theorem once again we conclude that
\begin{equation*}
h_2(n-j_1)^{n-j_1} \equiv  h_2(n-j_2)^{n-j_2} \pmod{p^k}.
\end{equation*}

Finally, if $n_1\equiv n_2\pmod{p^{3k-1}(p-1)}$ then
\begin{equation*}
\left\lfloor\frac{n_1-t}{p^{2k-1}(p-1)}\right\rfloor\equiv\left\lfloor\frac{n_2-t}{p^{2k-1}(p-1)}\right\rfloor\pmod{p^k}.
\end{equation*}

Note that (\ref{eq5}) holds for all $n\in\N$, if $\cal{P}_p = \emptyset$.
\end{proof}

Combining Propositions \ref{prop4}, \ref{prop5}, \ref{prop6} and Remark \ref{rem1} we obtain two corollaries.

\begin{cor}
Let ${\bf a}={\bf a}(f,h_1,h_2)$, $p$ be a prime number which does not divide $f(n)$ for any integer $n$ and $k$ be a positive integer. Then the basic period of the sequence $(a_n\pmod{p^k})_{n\in\N_{k-1}}$ is of the form $p^lc$, where $l\leq 2k$ and $c\mid p-1$. In particular $p^{2k}(p-1)$ is a period of the sequence $(a_n\pmod{p^k})_{n\in\N_{k-1}}$.  If $h_2 = -1$ and $p\neq 2$ then the basic period is of the form $2p^l$, $l\leq 2k$. If $h_2 = -1$ and $p=2$ or $h_2 = 1$ then the basic period is of the form $p^l$, $l\leq 2k$.

If additionally $p$ does not divide $h_2(n)$ for any integer $n$ then the sequence $(a_n\pmod{p^k})_{n\in\N}$ is periodic with basic period of the form as above.
\end{cor}

\begin{cor}
Let $d = p_1^{k_1}\cdot... \cdot p_s^{k_s} \cdot q_1^{m_1} \cdot .. \cdot q_t^{m_t}$ be the factorization of a given positive integer $d$. Let $p_i$, $i\in\{1,2,...,s\}$, does not divide $f(n)$ for any integer $n$ and $q_i$, $i\in\{1,2,...,t\}$, divides $f(n)$ for some integer $n$. Assume that $n_0\geq k_i-1$, $i\in\{1,2,...,s\}$ and $n_0\geq (q_i+1)v_{q_i}(d)-1$, $i\in\{1,2,...,t\}$. Then
\begin{equation*}
\lcm\{ p_i^{2k_i}(p_i-1), q_j^{m_j}(q_j-1): i\in\{1,2,...,s\}, j\in\{1,2,...,t\}\}
\end{equation*}
is the period of the sequence $(a_n\pmod{d})_{n\in\N_{n_0}}$.
\end{cor}

\subsection{Asymptotics and connection between boundedness and periodicity of a sequence ${\bf a}\in\mathcal{R}$}\label{subsec3.2}

\subsubsection{\bf Asymptotics of a sequence ${\bf a}\in\cal{R}$}

Let us notice that if ${\bf a}={\bf a}(f,h_1,h_2)$, $c$ is the leading coefficient of $h_2$ and $n_0\in\N$ is such that $f(n)\neq 0$ for all integers $n > n_0$ then
\footnotesize
\begin{equation}\label{eq5,5}
a_n = 
\begin{cases}
O(ne^{(|c|+\eps)n}h_1(n)h_2(n)^{n_0}\prod_{i=n_0+1}^n f(i)), & \mbox{if } |f(n)| \geq |h_2(n)| \mbox{ for } n\gg 0, \deg h_2>0 \mbox{ and } \eps>0 \\
O(nh_1(n)h_2(n)^{n_0}\prod_{i=n_0+1}^n f(i)), & \mbox{if } |f(n)| \geq |h_2(n)| \mbox{ for } n\gg 0 \mbox{ and } \deg h_2=0 \\
O(nh_1(n)h_2(n)^n), & \mbox{if } |f(n)| \leq |h_2(n)| \mbox{ for } n\gg0
\end{cases}
\end{equation}
\normalsize
when $n\rightarrow +\infty$. Indeed, when $|f(n)| \geq |h_2(n)|$ for $n\gg0$, we have
\begin{equation}\label{eq21}
\begin{split}
& \frac{a_n}{ne^{(|c|+\eps)n}h_1(n)h_2(n)^{n_0}\prod_{i=n_0+1}^n f(i))} = \frac{\sum_{j=0}^n h_1(j)h_2(j)^j \prod_{i=j+1}^n f(i)}{ne^{(|c|+\eps)n}h_1(n)h_2(n)^{n_0}\prod_{i=n_0+1}^n f(i))} \\
= & \sum_{j=0}^{n_0} \frac{h_1(j)h_2(j)^j \prod_{i=j+1}^{n_0} f(i)}{ne^{(|c|+\eps)n}h_1(n)h_2(n)^{n_0}} + \sum_{j=n_0+1}^n \frac{h_1(j)h_2(j)^j}{ne^{(|c|+\eps)n}h_1(n)h_2(n)^{n_0} \prod_{i=n_0+1}^j f(i)} \\
= & \sum_{j=0}^{n_0} \frac{h_1(j)h_2(j)^j \prod_{i=j+1}^{n_0} f(i)}{ne^{(|c|+\eps)n}h_1(n)h_2(n)^{n_0}} + \sum_{j=n_0+1}^n \frac{1}{ne^{(|c|+\eps)n}}\cdot\frac{h_1(j)}{h_1(n)}\cdot\frac{h_2(j)^{n_0}}{h_2(n)^{n_0}}\cdot\frac{h_2(j)^{j-n_0}}{\prod_{i=n_0+1}^j f(i)}.
\end{split}
\end{equation}

If $n\gg0$ then $|h_1(n)|\geq |h_1(j)|$ and $|h_2(n)|\geq |h_2(j)|$ for $0\leq j\leq n$. Moreover, the following equality holds.
\begin{equation*}
\left|\frac{h_2(j+1)^{j+1-n_0}}{\prod_{i=n_0+1}^{j+1} f(i)}\right| \cdot \left|\frac{\prod_{i=n_0+1}^j f(i)}{h_2(j)^{j-n_0}}\right| = \left|\frac{h_2(j+1)^{j-n_0}}{h_2(j)^{j-n_0}}\right| \cdot \left|\frac{h_2(j+1)}{f(j+1)}\right|
\end{equation*}
We have $\left|\frac{h_2(j+1)}{f(j+1)}\right| \leq 1$ for sufficiently large prositive integer $j$. Additionally, if $\deg h_2 > 0$ then
\scriptsize
\begin{equation*}
\begin{split}
& \left(\frac{|h_2(j+1)|}{|h_2(j)|}\right)^{j-n_0} = \left(\left(1+\frac{|h_2(j+1)|-|h_2(j)|}{|h_2(j)|}\right)^{\frac{|h_2(j)|}{|h_2(j+1)|-|h_2(j)|}}\right)^{\frac{(|h_2(j+1)|-|h_2(j)|)(j-n_0)}{|h_2(j)|}} \leq \\
& \leq e^{\frac{(|h_2(j+1)|-|h_2(j)|)(j-n_0)}{|h_2(j)|}} \leq e^{|c|+\eps},
\end{split}
\end{equation*}
\normalsize
since $\lim_{n\rightarrow +\infty}\frac{(|h_2(j+1)|-|h_2(j)|)(j-n_0)}{|h_2(j)|}\rightarrow |c|$. Hence each summand in the sum in (\ref{eq21}) is $O(\frac{1}{n})$, when $n\rightarrow +\infty$. Finally,
\begin{equation*}
\frac{a_n}{ne^{(|c|+\eps)n}h_1(n)h_2(n)^{n_0}\prod_{i=n_0+1}^n f(i))} = \sum_{j=0}^n O\left(\frac{1}{n}\right) = O(1), n\rightarrow +\infty.
\end{equation*}

The second and third equality from (\ref{eq5,5}) can be proved in the same way.

Consider now a sequence ${\bf a}(f,h_1,1)$. We assume that $f\neq b$, where $b\in\{-1,0,1\}$, and $n_0\in\N$ is such that $f(n)\neq 0$ for all integers $n > n_0$. We know that $a_n = O(nh_1(n)\prod_{i=n_0+1}^n f(i))$, when $n\rightarrow +\infty$. However, we can show something stronger. Namely, there is such a real number $\xi$ that $a_n \sim \xi\prod_{i=n_0+1}^n f(i)$, when $n\rightarrow +\infty$. Indeed,
\begin{equation*}
\begin{split}
\frac{a_n}{\prod_{i=n_0+1}^n f(i)} & = \frac{\sum_{j=0}^n h_1(j) \prod_{i=j+1}^n f(i)}{\prod_{i=n_0+1}^n f(i)} \\
& = \sum_{j=0}^{n_0} h_1(j) \prod_{i=j+1}^{n_0} f(i) + \sum_{j=n_0+1}^n \frac{h_1(j)}{\prod_{i=n_0+1}^j f(i)} \\
& = a_{n_0} + \sum_{j=n_0+1}^n \frac{h_1(j)}{\prod_{i=n_0+1}^j f(i)}
\end{split}
\end{equation*}
for $n>n_0$ and by ratio test the expression $a_{n_0} + \sum_{j=n_0+1}^n \frac{h_1(j)}{\prod_{i=n_0+1}^j f(i)}$ converges to the real number $\xi = a_{n_0} + \sum_{j=n_0+1}^{+\infty} \frac{h_1(j)}{\prod_{i=n_0+1}^j f(i)}$.

Using similar reasoning we show the asymptotic equality for a sequence ${\bf a}(f,h_1,-1)$.

In particular, for the sequence of derangements $(D_n)_{n\in\N}$ the following equality holds:
\begin{equation*}
\lim_{n\rightarrow +\infty}\frac{D_n}{\frac{n!}{e}}=1.
\end{equation*}
In order to establish the equality above, it suffices to compute the limit of $\frac{D_n}{n!}$, when $n\rightarrow +\infty$. We have
\begin{equation*}
\lim_{n\rightarrow +\infty}\frac{D_n}{n!} = \lim_{n\rightarrow +\infty}\frac{\sum_{j=0}^n (-1)^j \prod_{i=j+1}^n i}{n!} = \lim_{n\rightarrow +\infty}\sum_{j=0}^n \frac{(-1)^j}{j!} = e^{-1}
\end{equation*}
In fact, we know that for $n\in\N_+$, $D_n$ is the best integer approximation of $\frac{n!}{e}$ because the difference between these two numbers is less than $\frac{1}{n}$:
\begin{equation*}
\left| \frac{n!}{e}-D_n\right| = \left|\sum_{j=n+1}^{+\infty} \frac{n!}{j!}(-1)^j\right| < \sum_{j=n+1}^{+\infty} \frac{n!}{j!} < \sum_{j=n+1}^{+\infty} \frac{1}{(n+1)^{j-n}} = \frac{\frac{1}{n+1}}{1-\frac{1}{n+1}} = \frac{1}{n}.
\end{equation*}

\subsubsection{\bf Boundedness and periodicity of a sequence ${\bf a}\in\cal{R}$}

It is obvious that any periodic sequence of integers is bounded. On the other hand, in general, boundedness of a sequence does not imply its periodicity. In this section we will show that if a sequence ${\bf a}\in\cal{R}$ is bounded then $h_1 = 0$ or $h_2 = b$, where $b\in\{-1,0,1\}$, and we will give the form of this sequence. In particular, such a sequence is ultimately constant or ultimately periodic with period 2.

First, we prove that if there is a constant subsequence of the form $(a_{kn+l})_{n\in\N}$ for some $k\in\N_+$ and $l\in\N$ then $h_1 = 0$ or $h_2 = b$, where $b\in\{-1,0,1\}$, or $f,h_1,h_2$ are constant and $h_2=-f$. Then, assuming that $h_1 = 0$ or $h_2 = b$, where $b\in\{-1,0,1\}$, we will show that the sequence $(a_n)_{n\in\N}$ is ultimately constant or ultimately periodic with period 2. Next, assuming boundedness of $(a_n)_{n\in\N}$, we will use periodicity of $(a_n\pmod{p})_{n\in\N}$ for sufficiently large prime number $p$ to obtain the statement on ultimate periodicity of the sequence $(a_n)_{n\in\N}$.

\begin{thm}\label{thm7}
Let ${\bf a}={\bf a}(f,h_1,h_2)$, $k\in\N_+$ and $l\in\N$. If the sequence $(a_{kn+l})_{n\in\N}$ is constant then one of the following conditions is true:
\begin{itemize}
\item $h_1 = 0$ (and then the sequence $(a_n)_{n\in\N}$ is constant and equal to 0),
\item $h_2 = b$, where $b\in\{-1,0,1\}$,
\item $h_1=c\in\Z$ and $h_2=-f=b\in\Z$ (then $a_{2n}=b^{2n}c$ and $a_{2n+1}=0$ for all $n\in\N$).
\end{itemize}
\end{thm}

\begin{proof}
Let us assume that $h_1\neq 0$. If $f=0$ then $a_n = h_1(n)h_2(n)^n$ for all $n\in\N$ and thus the assumption of our theorem can be satisfied only if $h_2 = b$, where $b\in\{-1,0,1\}$. Hence we can assume that $f \neq 0$.

Let us denote $a = a_{kn+l}, n\in\N$. Then
\footnotesize
\begin{equation}\label{eq9}
\begin{split}
a = & a_{k(n+1)+l} \\
= & a_{kn+l}\prod_{i=1}^k f(kn+l+i) + \sum_{j=1}^k h_1(kn+l+j)h_2(kn+l+j)^{kn+l+j}\prod_{i=j+1}^k f(kn+l+i) \\
= & a\prod_{i=1}^k f(kn+l+i) + h_2(kn+l)^{kn+l}\sum_{j=1}^k h_1(kn+l+j)\times \\
& \times h_2(kn+l+j)^j\left(\frac{h_2(kn+l+j)}{h_2(kn+l)}\right)^{kn+l}\prod_{i=j+1}^k f(kn+l+i).
\end{split}
\end{equation}
\normalsize

Let us put $d = \deg h_2 > 0$ and write $h_2 = \sum_{i=0}^d w_i X^i$. Then for each $j\in\N$ we have
\begin{equation*}
\begin{split}
& h_2(kn+l+j) - h_2(kn+l) = \sum_{i=0}^d w_i (kn+l+j)^i - \sum_{i=0}^d w_i (kn+l)^i \\
& = w_d (kn+l)^d + dw_d j(kn+l)^{d-1} + w_{d-1} (kn+l)^{d-1} + O((kn+l)^{d-2}) - w_d (kn+l)^d - \\
& - w_{d-1} (kn+l)^{d-1} + O((kn+l)^{d-2}) = dw_d j(kn+l)^{d-1} + O((kn+l)^{d-2})
\end{split}
\end{equation*}
with $n\rightarrow +\infty$. Since
\begin{equation*}
\begin{split}
& \lim_{n\rightarrow +\infty}\frac{h_2(kn+l+j) - h_2(kn+l)}{h_2(kn+l)} = 0 \\
& \lim_{n\rightarrow +\infty}\frac{(kn+l)^d}{h_2(kn+l)} = \frac{1}{w_d} \\
& \lim_{n\rightarrow +\infty}\frac{O((kn+l)^{d-1})}{h_2(kn+l)} = 0,
\end{split}
\end{equation*}
thus
\footnotesize
\begin{equation}\label{eq8}
\begin{split}
& \lim_{n\rightarrow +\infty}\left(\frac{h_2(kn+l+j)}{h_2(kn+l)}\right)^{kn+l} \\
= & \lim_{n\rightarrow +\infty}\left(1 + \frac{h_2(kn+l+j) - h_2(kn+l)}{h_2(kn+l)}\right)^{\frac{h_2(kn+l)}{h_2(kn+l+j) - h_2(kn+l)} \cdot \frac{(kn+l)(dw_d j(kn+l)^{d-1} + O((kn+l)^{d-2}))}{h_2(kn+l)}} \\
= & \lim_{n\rightarrow +\infty}\left(\left(1 + \frac{h_2(kn+l+j) - h_2(kn+l)}{h_2(kn+l)}\right)^{\frac{h_2(kn+l)}{h_2(kn+l+j) - h_2(kn+l)}}\right)^{\frac{dw_d j(kn+l)^d + O((kn+l)^{d-1})}{h_2(kn+l)}} = e^{dj}.
\end{split}
\end{equation}
\normalsize

If $\deg f > \deg h_2$ then we compute the following limits.
\footnotesize
\begin{equation*}
\begin{split}
& \lim_{n\rightarrow +\infty}\frac{a\prod_{i=1}^k f(kn+l+i)}{h_2(kn+l)^{kn+l} h_1(kn+l+1)h_2(kn+l+1)e^d\prod_{i=2}^k f(kn+l+i)} = 0, \\
& \lim_{n\rightarrow +\infty}\frac{h_2(kn+l)^{kn+l} h_1(kn+l+1)h_2(kn+l+1)\left(\frac{h_2(kn+l+1)}{h_2(kn+l)}\right)^{kn+l}\prod_{i=2}^k f(kn+l+i)}{h_2(kn+l)^{kn+l} h_1(kn+l+1)h_2(kn+l+1)e^d\prod_{i=2}^k f(kn+l+i)} = 1, \\
& \lim_{n\rightarrow +\infty}\frac{h_2(kn+l)^{kn+l} h_1(kn+l+j)h_2(kn+l+j)^j\left(\frac{h_2(kn+l+j)}{h_2(kn+l)}\right)^{kn+l}\prod_{i=j+1}^k f(kn+l+i)}{h_2(kn+l)^{kn+l} h_1(kn+l+1)h_2(kn+l+1)e^d\prod_{i=2}^k f(kn+l+i)} = 0 \mbox{, as } j>1.
\end{split}
\end{equation*}
\normalsize
After adding these limits we obtain the following one.
\begin{equation*}
\begin{split}
& \lim_{n\rightarrow +\infty}\frac{a}{h_2(kn+l)^{kn+l} h_1(kn+l+1)h_2(kn+l+1)e^d\prod_{i=2}^k f(kn+l+i)} \\
& = \lim_{n\rightarrow +\infty}\left(\frac{a\prod_{i=1}^k f(kn+l+i)}{h_2(kn+l)^{kn+l} h_1(kn+l+1)h_2(kn+l+1)e^d\prod_{i=2}^k f(kn+l+i)} +\right. \\
& \left.+ \frac{\sum_{j=1}^k h_1(kn+l+j) h_2(kn+l+j)^j\left(\frac{h_2(kn+l+j)}{h_2(kn+l)}\right)^{kn+l}\prod_{i=j+1}^k f(kn+l+i)}{h_1(kn+l+1)h_2(kn+l+1)e^d\prod_{i=2}^k f(kn+l+i)}\right) = 1,
\end{split}
\end{equation*}
which leads to a contradiction with the fact that
\begin{equation*}
\lim_{n\rightarrow +\infty}\frac{a}{h_2(kn+l)^{kn+l} h_1(kn+l+1)h_2(kn+l+1)e^d\prod_{i=2}^k f(kn+l+i)} = 0.
\end{equation*}

Similarly, if $\deg f < \deg h_2$ then we compute the following limits.
\footnotesize
\begin{equation*}
\begin{split}
& \lim_{n\rightarrow +\infty}\frac{a\prod_{i=1}^k f(kn+l+i)}{h_2(kn+l)^{kn+l} h_1(kn+l+k)h_2(kn+l+k)^k e^{dk}} = 0, \\
& \lim_{n\rightarrow +\infty}\frac{h_2(kn+l)^{kn+l} h_1(kn+l+k)h_2(kn+l+k)^k \left(\frac{h_2(kn+l+k)}{h_2(kn+l)}\right)^{kn+l}}{h_2(kn+l)^{kn+l} h_1(kn+l+k)h_2(kn+l+k)^k e^{dk}} = 1, \\
& \lim_{n\rightarrow +\infty}\frac{h_2(kn+l)^{kn+l} h_1(kn+l+j)h_2(kn+l+j)^j\left(\frac{h_2(kn+l+j)}{h_2(kn+l)}\right)^{kn+l}\prod_{i=j+1}^k f(kn+l+i)}{h_2(kn+l)^{kn+l} h_1(kn+l+k)h_2(kn+l+k)^k e^{dk}} = 0 \mbox{, as } j<k.
\end{split}
\end{equation*}
\normalsize
We add these limits to obtain the following.
\begin{equation*}
\begin{split}
& \lim_{n\rightarrow +\infty}\frac{a}{h_2(kn+l)^{kn+l} h_1(kn+l+k)h_2(kn+l+k)^k e^{dk}} \\
& = \lim_{n\rightarrow +\infty}\left(\frac{a\prod_{i=1}^k f(kn+l+i)}{h_2(kn+l)^{kn+l} h_1(kn+l+k)h_2(kn+l+k)^k e^{dk}} +\right. \\
& \left.+ \frac{\sum_{j=1}^k h_1(kn+l+j) h_2(kn+l+j)^j\left(\frac{h_2(kn+l+j)}{h_2(kn+l)}\right)^{kn+l}\prod_{i=j+1}^k f(kn+l+i)}{h_1(kn+l+k)h_2(kn+l+k)^k e^{dk}}\right) = 1,
\end{split}
\end{equation*}
which leads to a contradiction with the fact that
\begin{equation*}
\lim_{n\rightarrow +\infty}\frac{a}{h_2(kn+l)^{kn+l} h_1(kn+l+k)h_2(kn+l+k)^k e^{dk}} = 0.
\end{equation*}

Consider the case when $\deg f = \deg h_2$. Let us denote $f=\sum_{i=0}^d u_i X^i$. Then we compute the following limits
\scriptsize
\begin{equation*}
\begin{split}
& \lim_{n\rightarrow +\infty}\frac{a\prod_{i=1}^k f(kn+l+i)}{h_2(kn+l)^{kn+l} h_1(kn+l+k)h_2(kn+l+k)^k} = 0, \\
& \lim_{n\rightarrow +\infty}\frac{h_2(kn+l)^{kn+l} h_1(kn+l+j)h_2(kn+l+j)^j \left(\frac{h_2(kn+l+j)}{h_2(kn+l)}\right)^{kn+l}\prod_{i=j+1}^k f(kn+l+i)}{h_2(kn+l)^{kn+l} h_1(kn+l+k)h_2(kn+l+k)^k} = \left(\frac{u_d}{w_d}\right)^{k-j}e^{dj},
\end{split}
\end{equation*}
\normalsize
as $1\leq j\leq k$. We add them and obtain the following limit.
\begin{equation*}
\begin{split}
& \lim_{n\rightarrow +\infty}\frac{a}{h_2(kn+l)^{kn+l} h_1(kn+l+k)h_2(kn+l+k)^k} \\
& = \lim_{n\rightarrow +\infty}\left(\frac{a\prod_{i=1}^k f(kn+l+i)}{h_2(kn+l)^{kn+l} h_1(kn+l+k)h_2(kn+l+k)^k} +\right. \\
& \left.+ \frac{\sum_{j=1}^k h_1(kn+l+j) h_2(kn+l+j)^j\left(\frac{h_2(kn+l+j)}{h_2(kn+l)}\right)^{kn+l}\prod_{i=j+1}^k f(kn+l+i)}{h_1(kn+l+k)h_2(kn+l+k)^k}\right) \\
& = \sum_{j=1}^k \left(\frac{u_d}{w_d}\right)^{k-j}e^{dj}.
\end{split}
\end{equation*}
However
\begin{equation*}
\lim_{n\rightarrow +\infty}\frac{a}{h_2(kn+l)^{kn+l} h_1(kn+l+k)h_2(kn+l+k)^k e^{dk}} = 0
\end{equation*}
and $\sum_{j=1}^k \left(\frac{u_d}{w_d}\right)^{k-j}e^{dj} \neq 0$ because $e$ is a transcendental number (see \cite{Her}) - a contradiction.

We proved that if the sequence $(a_n)_{n\in\N}$ is bounded and $h_1 \neq 0$ then $h_2 = b$, where $b\in\Z$. In the case when $h_2 = b$ the equality (\ref{eq9}) takes the form
\begin{equation}\label{eq11}
a = a\prod_{i=1}^k f(kn+l+i) + b^{kn+l}\sum_{j=1}^k h_1(kn+l+j) b^j\prod_{i=j+1}^k f(kn+l+i).
\end{equation}

Assume that $|b| > 1$ and define
\begin{equation*}
G = \sum_{j=1}^k h_1(kX+l+j) b^j\prod_{i=j+1}^k f(kX+l+i) \in\Z[X].
\end{equation*}

If $G \neq 0$ then by (\ref{eq11}) we have
\begin{equation*}
\frac{a}{b^{kn+l}G(n)} = \frac{a\prod_{i=1}^k f(kn+l+i)}{b^{kn+l}G(n)} + 1.
\end{equation*}
Since $\lim_{n\rightarrow +\infty}\frac{a\prod_{i=1}^k f(kn+l+i)}{b^{kn+l}} = 0$ we deduce the following
\begin{equation*}
\lim_{n\rightarrow +\infty}\frac{a}{b^{kn+l}G(n)} = 1.
\end{equation*}
We get a contradiction because $\lim_{n\rightarrow +\infty}\frac{a}{b^{kn+l}G(n)} = 0$.

If $G = 0$ then $h_1 = 0$ or $f = c$, where $c\in\Z$. Indeed, if $h_1 \neq 0$ and $\deg f > 0$ then $\deg \left[h_1(kX+l+j) b^j\prod_{i=j+1}^k f(kX+l+i)\right] = (k-j)\deg f + \deg h_1$ for $j\in\{1,2,...,k\}$ and as a result we get $\deg G = k\deg f + \deg h_1 > 0$. Assume that $f = c$, where $c\in\Z$. Then
\begin{equation*}
\begin{split}
0 = & \lim_{n\rightarrow +\infty}\frac{G(n)}{bh_1(kn+l)} = \lim_{n\rightarrow +\infty}\frac{\sum_{j=1}^k h_1(kn+l+j) b^{j-1} c^{k-j}}{h_1(kn+l)} \\
= & \sum_{j=1}^k b^{j-1} c^{k-j} =
\begin{cases}
\frac{b^k-c^k}{b-c}, &\mbox{ when } b\neq c \\
kb^{k-1}, &\mbox{ when } b=c
\end{cases},
\end{split}
\end{equation*}
which means that there must be $b=c=0$ or $c=-b$ with $2\mid k$. The case $b=c=0$ contradicts with the assumption that $|b| > 1$. If $c=-b$ then by induction we obtain the following formula
\begin{equation*}
a_{l+n} = (-b)^n a_l + \sum_{j=1}^n (-1)^{n-j}b^n h_1(l+j) = (-b)^n a_l + (-b)^n\sum_{j=1}^n (-1)^j h_1(l+j), n\in\N.
\end{equation*}
Let us define
\begin{equation*}
H(n) = \sum_{j=1}^{2n} (-1)^j h_1(l+j) = \sum_{j=1}^n \left(h_1(2j)-h_1(2j-1)\right) = h_1(0) + \sum_{j=1}^n \Delta h_1(2j-1), n\in\N,
\end{equation*}
where $\Delta h_1 = h_1(X+1)-h_1(X)$ means the discrete derivative of $h_1$. The function $H$ can be seen as a polynomial in $n$ and its degree is equal to
\begin{equation*}
\deg H(X) = 1 + \deg\Delta h_1(2X-1) = 1 + \deg\Delta h_1(X) = \deg h_1(X).
\end{equation*}
Since
\begin{equation*}
a=a_l=a_{l+kn}=(-b)^{kn} a + (-b)^{kn}H\left(\frac{k}{2}n\right) = b^{kn}\left(a+H\left(\frac{k}{2}n\right)\right)
\end{equation*}
for all $n\in\N$ and $|b|>1$ we deduce that the polynomial $H$ must be constant. This implies that $h_1$ is constant.

Summing up, we showed that $h_2=b\in\{-1,0,1\}$ or $f, h_1, h_2$ are constant and $f=-h_2$.
\end{proof}

\begin{thm}\label{thm8}
Let ${\bf a}={\bf a}(f,h_1,h_2)$, $k\in\N_+$ and $l\in\N$. If the sequence $(a_{kn+l})_{n\in\N}$ is constant then the sequence $(a_n)_{n\in\N}$ is ultimately constant or of the form $(c,0,c,0,c,0,...)$ for some integer $c$.
\end{thm}

\begin{proof}
For $k=1$ the statement is obvious. Hence assume without loss of generality that $k\geq 2$. Let us denote $a = a_{kn+l}, n\in\N$. Then
\begin{equation}\label{eq11,5}
\begin{split}
a & = a_{k(n+1)+l} = a_{kn+l}\prod_{i=1}^k f(kn+l+i) + \sum_{j=1}^k h_1(kn+l+j)\prod_{i=j+1}^k f(kn+l+i) \\
& = a\prod_{i=1}^k f(kn+l+i) + \sum_{j=1}^k h_1(kn+l+j)\prod_{i=j+1}^k f(kn+l+i).
\end{split}
\end{equation}

Let us define
\begin{equation*}
G = a\prod_{i=1}^k f(kX+l+i) + \sum_{j=1}^k h_1(kX+l+j)\prod_{i=j+1}^k f(kX+l+i) \in\Z[X].
\end{equation*}

From (\ref{eq11,5}) we know that $G=a$. If $h_1 = 0$ then $a_n=0$ for all $n\in\N$, so we can assume that $h_1 \neq 0$.

If $\deg f > 0$ then $\deg \prod_{i=1}^k f(kX+l+i) = k\deg f$ and
\begin{equation*}
\deg h_1(kX+l+j)\prod_{i=j+1}^k f(kX+l+i) = (k-j)\deg f + \deg h_1
\end{equation*}
for $j\in\{1,2,...,k\}$. Since $\deg G \leq 0$ we get
\begin{equation*}
\deg \prod_{i=1}^k f(kX+l+i) = \deg h_1(kX+l+1)\prod_{i=2}^k f(kX+l+i),
\end{equation*}
which implies that $\deg f = \deg h_1$. Moreover, we have the following sequence of equivalences.
\footnotesize
\begin{equation*}
\begin{split}
& \deg(a\prod_{i=1}^k f(kX+l+i) + h_1(kX+l+1)\prod_{i=2}^k f(kX+l+i)) \\
& = \deg h_1(kX+l+2)\prod_{i=3}^k f(kX+l+i) \\
\iff & \deg(af(kX+l+1) + h_1(kX+l+1))\prod_{i=2}^k f(kX+l+i) \\
& = \deg h_1(kX+l+2)\prod_{i=3}^k f(kX+l+i) \\
\iff & \deg(af(kX+l+1) + h_1(kX+l+1)) + (k-1)\deg f \\
& = \deg h_1 + (k-2)\deg f \\
\iff & \deg(af(kX+l+1) + h_1(kX+l+1)) = 0
\end{split}
\end{equation*}
\normalsize
Hence $af+h_1=b$ for some integer $b$. Therefore we have
\footnotesize
\begin{equation*}
\begin{split}
G & = (af(kX+l+1)+h_1(kX+l+1))\prod_{i=2}^k f(kXn+l+i) + \sum_{j=2}^k h_1(kX+l+j)\prod_{i=j+1}^k f(kX+l+i) \\
& = b\prod_{i=2}^k f(kX+l+i) + \sum_{j=2}^k h_1(kX+l+j)\prod_{i=j+1}^k f(kX+l+i) \\
& = (bf(kX+l+2)+h_1(kX+l+2))\prod_{i=3}^k f(kX+l+i) + \sum_{j=3}^k h_1(kX+l+j)\prod_{i=j+1}^k f(kX+l+i).
\end{split}
\end{equation*}
\normalsize
Similarly, from the fact that $\deg G \leq 0$ we get the following chain of equivalences
\footnotesize
\begin{equation*}
\begin{split}
& \deg (bf(kX+l+2)+h_1(kX+l+2))\prod_{i=3}^k f(kX+l+i) = \deg h_1(kX+l+3)\prod_{i=4}^k f(kX+l+i) \\
& \iff \deg (bf(kX+l+2)+h_1(kX+l+2)) + (k-3)\deg f = \deg h_1 + (k-4)\deg f \\
& \iff \deg (bf(kX+l+2)+h_1(kX+l+2)) = 0
\end{split}
\end{equation*}
\normalsize
provided $n\geq 3$. If $k=2$ then $\deg (bf(kX+l+2)+h_1(kX+l+2)\leq 0$. Since $\deg (af(kX+l+2)+h_1(kX+l+2)), \deg (bf(kX+l+2)+h_1(kX+l+2)) \leq 0$, hence
\footnotesize
\begin{equation*}
\deg (a-b)f(kX+l+2) = \deg [(af(kX+l+2)+h_1(kX+l+2)) - (bf(kX+l+2)+h_1(kX+l+2))] \leq 0.
\end{equation*}
\normalsize
We made an assumption $\deg f > 0$. That is why $a=b$. We obtain the equality $af+h_1 = a$, which allows us to prove by simple induction that $a_n = a$ for each $n\geq l$.

Assume now that $f=b$ for some integer $b$. Then
\begin{equation}\label{eq12}
G = a = ab^k + \sum_{j=1}^k b^{k-j}h_1(kX+l+j).
\end{equation}

If $\deg h_1 > 0$ and $h_1 = \sum_{i=0}^d w_i X^i$ then the coefficient of $G$ near the $d$th power of variable $X$ is 0 since $\deg G \leq 0$. On the other hand, this coefficient is equal to
\begin{equation*}
k^dw_d \sum_{j=1}^k b^{k-j} =
\begin{cases}
k^dw_d \frac{b^k-1}{b-1}, &\mbox{ if } b\neq 1 \\
k^{d+1}w_d, &\mbox{ if } b = 1
\end{cases},
\end{equation*}
which means that there must be $2\mid k$ and $b=-1$. Denote $k' = \frac{k}{2}$ and take the discrete derivative $\Delta h_1 = h_1(X+1) - h_1(X)$ of the polynomial $h_1$. We know that $\deg \Delta h_1 = \deg h_1 - 1$. The equation (\ref{eq12}) takes the form:
\begin{equation*}
0 = \sum_{j=1}^{k'} h_1(kn+l+2j)-h_1(kn+l+2j-1) = \sum_{j=1}^{k'} \Delta h_1(kn+l+2j-1).
\end{equation*}
Let $H = \sum_{j=1}^{k'} \Delta h_1(kX+l+2j-1) \in\Z[X]$. Then $H = 0$. However, the coefficient of $H$ near $d-1$st power of variable $X$ is equal to $k'$ times the leading coefficient of $\Delta h_1$ - a contradiction.

We are left with the case $h_1 = c$ for some $c\in\Z\bs\{0\}$. By (\ref{eq12}) we have
\begin{equation*}
0 = a(b^k-1) + c\sum_{j=1}^{k} b^{k-j}.
\end{equation*}
On the other hand,
\begin{equation*}
a(b^k-1) + c\sum_{j=1}^{k} b^{k-j} =
\begin{cases}
a(b^k-1) + c\frac{b^k-1}{b-1}, &\mbox{ if } b\neq 1 \\
kc, &\mbox{ if } b = 1
\end{cases}.
\end{equation*}
Since $c\neq 0$, thus $b\neq 1$ and $(a+\frac{c}{b-1})(b^k-1) = 0$. Then $b=-1$ and $(a_n)_{n\in\N} = (c,0,c,0,c,0,...)$ or $c = a(1-b)$, which implies that $ba+c=a$ and $(a_n)_{n\in\N}$ is ultimately constant.
\end{proof}

\begin{ex}
Let us consider the sequence ${\bf a}(X-3,28-7X,1)$. Then $a_1 = -35$, $a_2 = 49$ and $a_n = 7$ for $n\geq 3$. This means that a sequence ${\bf a}$ satisfying assumptions of Theorem \ref{thm8} can be ultimately constant, but not constant.
\end{ex}

\begin{cor}\label{cor2}
Let ${\bf a}={\bf a}(f,h_1,-1)$. If the sequence $(a_{kn+l})_{n\in\N}$ is constant then there is such an integer $c$ that $a_n = (-1)^nc$ for almost all $n\in\N$ or $a_{2n}=c, a_{2n+1}=0$ for all $n\in\N$.
\end{cor}

\begin{proof}
Consider the associated sequence $(\widetilde{a}_n)_{n\in\N}$. Since $\widetilde{a}_n = (-1)^n a_n$ for $n\in\N$, thus the sequence $(\widetilde{a}_{2kn+l})_{n\in\N}$ is constant and by Theorem \ref{thm8} there is such an integer $c$ that $\widetilde{a}_n = c$ for almost all $n\in\N$ or $\widetilde{a}_{2n}=c, \widetilde{a}_{2n+1}=0$ for all $n\in\N$.
\end{proof}

\begin{prop}\label{prop7}
Let us consider a sequence ${\bf a}(f,h_1,0)$. Let $k\in\N_+$ and $l\in\N$. If the sequence $(a_{kn+l})_{n\in\N}$ is constant then $a_n = 0$ for almost all $n\in\N$, $a_n = h_1(0)$ for all $n\in\N$ or $a_n = (-1)^n h_1(0)$ for all $n\in\N$.
\end{prop}

\begin{proof}
If $a_{n_0} = 0$ for some $n_0\in\N$ then the sequence $(a_n)_{n\in\N}$ is ultimately constant and equal to 0, so assume that $a_n \neq 0$ for any $n\in\N$. Denote $a = a_{kn+l}$, $n\in\N$. Then
\begin{equation*}
a = a_{k(n+1)+l} = a_{kn+l}\prod_{i=1}^k f(kn+l+i) = a\prod_{i=1}^k f(kn+l+i)
\end{equation*}
and since $a\neq 0$ we get $\prod_{i=1}^k f(kn+l+i)=1$ for all $n\in\N$. Hence $|f(n)|=1$ for all but finitely many $n\in\N$, which implies that $f=1$ or $f=-1$.
\end{proof}

\begin{thm}
Let $(a_n)_{n\in\N}$ be a bounded sequence given by the formula $a_0 = h_1(0), a_n = f(n)a_{n-1} + h_1(n)h_2(n)^n, n>0$. Then one of the following conditions is true:
\begin{itemize}
\item $h_1 = 0$ (and then the sequence $(a_n)_{n\in\N}$ is constantly equal to 0),
\item $h_2 = b$, where $b\in\{-1,0,1\}$.
\end{itemize}
Moreover,
\begin{itemize}
\item if $h_2 = 1$ then there is such an integer $c$ that $a_n = c$ for almost all $n\in\N$ or $a_{2n}=c, a_{2n+1}=0$ for all $n\in\N$,
\item if $h_2 = -1$ then there is such an integer $c$ that $a_n = (-1)^nc$ for almost all $n\in\N$ or $a_{2n}=c, a_{2n+1}=0$ for all $n\in\N$,
\item if $h_2 = 0$ then $a_n = 0$ for almost all $n\in\N$, $a_n = h_1(0)$ for all $n\in\N$ or $a_n = (-1)^n h_1(0)$ for all $n\in\N$.
\end{itemize}
\end{thm}

\begin{proof}
By Theorems \ref{thm7} and \ref{thm8}, Corollary \ref{cor2} and Proposition \ref{prop7}, it suffices to show that there are such $k\in\N_+$ and $l\in\N$ that the sequence $(a_{kn+l})_{n\in\N}$ is constant.

Let $p$ be a prime number greater than $\max_{n\in\N} a_n - \min_{n\in\N} a_n$. Then the sequence of remainders $(a_n\pmod{p})_{n\in\N}$ is periodic (see Section \ref{subsec3.1}). Moreover, the values of this sequence and the value $\min_{n\in\N} a_n$ uniquely determine the values of the sequence $(a_n)_{n\in\N}$. Indeed, if $a_{n_1} \equiv a_{n_2} \pmod{p}$ then $a_{n_1} - \min_{n\in\N} a_n \equiv a_{n_2} - \min_{n\in\N} a_n \pmod{p}$ and since $a_{n_1} - \min_{n\in\N}, a_{n_2} - \min_{n\in\N} < p$, thus $a_{n_1} = a_{n_2}$. Therefore the sequence $(a_n)_{n\in\N}$ is periodic. This fact implies the existence of such $k\in\N_+$ and $l\in\N$ that the sequence $(a_{kn+l})_{n\in\N}$ is constant.
\end{proof}

\subsection{The polynomials arising in the recurrence relation for a sequence ${\bf a}\in\cal{R}$ and their real roots}\label{subsec3.3}

Let us consider a sequence ${\bf a}(f,h_1,1)\in\cal{R}'$. In Section \ref{subsubsec3.1.1} we defined polynomials
\begin{equation*}
f_d = \sum_{j=0}^{d-1} h_1(X-j) \prod_{i=0}^{j-1} f(X-i) \in\Z[X], d\in\N.
\end{equation*}

Using the closed formula for $a_n$ we can obtain the recurrence equations, which are generalizations of the recurrence definition of $a_n$.
\begin{equation}\label{eq13}
\begin{split}
a_n & = \sum_{j=0}^n h_1(j) \prod_{i=j+1}^n f(i) = a_{n-d}\prod_{i=n-d+1}^n f(i) + \sum_{j=n-d+1}^n h_1(j) \prod_{i=j+1}^n f(i) \\
& = a_{n-d}\prod_{i=n-d+1}^n f(i) + \sum_{j=0}^{d-1} h_1(n-j) \prod_{i=0}^{j-1} f(n-i) = a_{n-d}\prod_{i=n-d+1}^n f(i) + f_d(n)
\end{split}
\end{equation}
for $n\geq d$. Furthermore, we can obtain the recurrence equations for the polynomials $f_d$, $d\in\N$. For given $d_1, d_2 \in\N$, comparing the formulae for $f_{d_1}$, $f_{d_2}$ and $f_{d_1+d_2}$, we get
\footnotesize
\begin{equation}\label{eq14}
\begin{split}
f_{d_1+d_2} & = \sum_{j=0}^{d_1+d_2-1} h_1(X-j) \prod_{i=0}^{j-1} f(X-i) = \sum_{j=0}^{d_2-1} h_1(X-j) \prod_{i=0}^{j-1} f(X-i) + \\
& + \sum_{j=d_2}^{d_1+d_2-1} h_1(X-j) \prod_{i=0}^{j-1} f(X-i) = f_{d_2} + \prod_{i=0}^{d_2-1} f(X-i)\sum_{j=d_2}^{d_1+d_2-1} h_1(X-j) \prod_{i=d_2}^{j-1} f(X-i) \\
& = f_{d_2} + \prod_{i=0}^{d_2-1} f(X-i)\sum_{j=0}^{d_1-1} h_1(X-d_2-j) \prod_{i=0}^{j-1} f(X-d_2-i) = f_{d_1}(X-d_2)\prod_{i=0}^{d_2-1} f(X-i) + f_{d_2}.
\end{split}
\end{equation}
\normalsize

Similarity of (\ref{eq13}) and (\ref{eq14}) and the fact that $f_{d+1}(d) = a_d$ for $d\in\N$ allow us to say, that the polynomials $f_d$, $d\in\N$ are a generalization of the numbers $a_d$, $d\in\N$.

Analogous formulae can be obtained for a sequence ${\bf a}(f,h_1,-1)\in\cal{R}''$.
\begin{equation*}
a_n = a_{n-d}\prod_{i=n-d+1}^n f(i) + (-1)^n f_d(n), n\geq d,
\end{equation*}
\begin{equation*}
f_{d_1+d_2} = (-1)^{d_2}f_{d_1}(X-d_2)\prod_{i=0}^{d_2-1} f(X-i) + f_{d_2},
\end{equation*}
where $f_d = \sum_{j=0}^{d-1} (-1)^j h_1(X-j) \prod_{i=0}^{j-1} f(X-i), d\in\N$. Moreover, $f_{d+1}(d)=(-1)^d a_d$ for $d\in\N$.

Let us consider the sequence of derangements. Since this sequence is given by the recurrence $D_0 = 1, D_n = nD_{n-1} + (-1)^n, n>0$, hence
\begin{equation*}
f_d = \sum_{j=0}^{d-1} (-1)^j \prod_{i=0}^{j-1} (X-i)
\end{equation*}
for $d\in\N$. Define
\begin{equation*}
\widehat{f}_d = \frac{f_d}{X-1} = -1 + X\sum_{j=2}^{d-1} (-1)^j \prod_{i=2}^{j-1} (X-i), d>1.
\end{equation*}
Since $f_{d+1}(d) = (-1)^dD_d$ for $d\in\N$ we get $\widehat{f}_{d+1}(d) = (-1)^dE_d = (-1)^d\frac{D_d}{d-1}$, $d\geq 2$. It is easy to see from the definition of $\widehat{f}_d$ that
\begin{equation*}
\widehat{f}_d = -1 + X\sum_{j=0}^{d-3} (-1)^j \prod_{i=0}^{j-1} (X-2-i) = Xf_{d-2}(X-2)-1, d\geq 2.
\end{equation*}

If we substitute $n$ into the place of $X$ and $n+1$ into the place of $d$ in the equation above, we obtain the identity
\begin{equation*}
E_n = nD_{n-2}-1, n\geq 2.
\end{equation*}

Let us notice that
\begin{equation*}
\begin{split}
& f_{d-1}(X-2) - f_d(X-1) = \sum_{j=0}^{d-2} (-1)^j \prod_{i=0}^{j-1} (X-2-i) - \sum_{j=0}^{d-1} (-1)^j \prod_{i=0}^{j-1} (X-1-i) \\
= & -1 + \sum_{j=0}^{d-2} (-1)^j \left(\prod_{i=0}^{j-1} (X-2-i) + \prod_{i=0}^j (X-1-i)\right) = -1 + \sum_{j=0}^{d-2} (-1)^j X\prod_{i=0}^{j-1} (X-2-i) \\
= & -1 + X\sum_{j=0}^{d-2} (-1)^j \prod_{i=0}^{j-1} (X-2-i) = \widehat{f}_{d+1}(X), d\geq 1.
\end{split}
\end{equation*}

If we substitute $X=n$ and $d=n$ in the formula above and divide by $(-1)^n$ then we get a well-known identity
\begin{equation*}
E_n = D_{n-2} + D_{n-1}, n\geq 2.
\end{equation*}


Now we will state and prove a theorem concerning real roots of polynomials $f_d$, $d\geq 3$, related to a sequence ${\bf a}={\bf a}(f,c,1)$, where $f\in\Z[X]$ and $c\in\Z\bs\{0\}$.

\begin{thm}\label{thm9}
Assume that ${\bf a}={\bf a}(f,c,1)$, where $f\in\Z[X]$ and $c\in\Z\bs\{0\}$. Let
\begin{equation*}
f_d = c\sum_{j=0}^{d-1} \prod_{i=0}^{j-1} f(X-i) \in\Z[X]
\end{equation*}
for $d\in\N$. Assume that $d\geq 3$, there is an integer $n_0$, which is the greatest real root of $f$ and $f$ as a function is decreasing on the set $[n_0+1,+\infty)\cap\Z$. Then $f_d$ has at least $d-2$ real roots. More precisely, $f_d$ has a root in the interval $(n_0+l,n_0+l+1)$, where $l\in\{2,3,...,d-2\}$, and
\begin{itemize}
\item if $f(n_0+1) < -1$ or $f'_d(n_0+1)\neq 0$ then $f_d$ has 2 real roots in the interval $(n_0,n_0+2)$; 
\item if $f(n_0+1) = -1$ then $f_d(n_0+1) = 0$.
\end{itemize}
In particular, if $\deg f = 1$ then $f_d$ factorizes into linear polynomials with real coefficients.
\end{thm}

\begin{proof}
We will compute the signs of the values $f_d(n_0), f_d(n_0+1), ..., f_d(n_0+d-1)$ and use intermediate value theorem to conclude the existence of real roots of $f_d$. We have the following equalities:
\begin{equation*}
\begin{split}
& \sgn (f_d(n_0)) = \sgn (c), \\
& \sgn (f_d(n_0+1)) = \sgn (c(1+f(n_0+1))) =
\begin{cases}
-\sgn (c), &\mbox{ when } f(n_0+1)<-1 \\
0, &\mbox{ when } f(n_0+1)=-1
\end{cases}.
\end{split}
\end{equation*}

Let us fix $l\in\{2,3,...,d-1\}$. Then
\begin{equation*}
\sgn (f_d(n_0+l)) = \sgn (c\sum_{j=0}^l \prod_{i=0}^{j-1} f(n_0+l-i)) = \sgn (c) \sgn (\prod_{i=0}^{l-1} f(n_0+l-i)) = (-1)^l \sgn (c)
\end{equation*}
because $|\prod_{i=0}^{l-1} f(n_0+l-i)|\geq |\prod_{i=0}^{l-2} f(n_0+l-i)|$ and $|\prod_{i=0}^{l-2s+1} f(n_0+l-i)| > |\prod_{i=0}^{l-2s} f(n_0+l-i)|$ for $s\in\{2,3,...,\lfloor\frac{l}{2}\rfloor\}$ (since $f$ is decreasing on the set $[n_0+1,+\infty)\cap\Z$) and $\sgn (\prod_{i=0}^{l-1} f(n_0+l-i)) = \sgn (\prod_{i=0}^{l-2s+1} f(n_0+l-i))$ for $s\in\{2,3,...,\lfloor\frac{l}{2}\rfloor\}$).

By intermediate value theorem there is a root of $f_d$ in each of the intervals of the form $(n_0+l,n_0+l+1)$, where $l\in\{2,3,...,d-2\}$. If $f(n_0+1)<-1$ then $\sgn (f_d(n_0+1)) = -\sgn (c)$. Since $\sgn (f_d(n_0)) = \sgn (f_d(n_0+2)) = \sgn (c)$ by intermediate value theorem there are roots $f_d$ in the intervals $(n_0,n_0+1)$ and $(n_0+1,n_0+2)$. If $f(n_0+1)=-1$ then $f_d(n_0+1) = 0$. If $f_d(n_0+1) = 0$ and $f'_d(n_0+1) \neq 0$ then $\sgn(f_d(x_0)) = -\sgn(c)$ for some $x_0\in (n_0,n_0+2)$. Hence there is a root of $f_d$ in the interval $(n_0,x_0)$, if $x_0<n_0+1$, or there is a root of $f_d$ in the interval $(x_0,n_0+2)$, if $x_0>n_0+1$.

If $\deg f = 1$ then $\deg f_d = d-1$. Hence, if $f(n_0+1) < -1$ or $f'_d(n_0+1)\neq 0$ then $f_d$ has $d-1$ distinct real roots. If $f(n_0+1) = -1$ and $f'_d(n_0+1) = 0$ then $f_d$ has $d-2$ distinct real roots, where $n_0-1$ is its double root. As a consequence of the reasoning presented above, the polynomial $f_d$ factorizes into linear factors over $\R$.
\end{proof}

We can use Theorem \ref{thm9} to obtain similar result for a sequence $(a_n)_{n\in\N}$ given by the formula $a_0 = c, a_n = f(n)a_{n-1} + (-1)^n c, n>0$ for some nonzero integer $c$.

\begin{cor}\label{cor3}
Assume that ${\bf a}={\bf a}(f,c,1)$, where $f\in\Z[X]$ and $c\in\Z\bs\{0\}$. Let
\begin{equation*}
f_d = c\sum_{j=0}^{d-1} (-1)^j\prod_{i=0}^{j-1} f(X-i) \in\Z[X]
\end{equation*}
for $d\in\N$. Assume that $d\geq 3$, there is an integer $n_0$, which is the greatest real root of $f$ and $f$ as a function is increasing on the set $[n_0+1,+\infty)\cap\Z$. Then $f_d$ has at least $d-2$ real roots. More precisely, $f_d$ has a root in the interval $(n_0,n_0+2)$ and in the interval $(n_0+l,n_0+l+1)$, where $l\in\{2,3,...,d-2\}$, and
\begin{itemize}
\item if $f(n_0+1) > 1$ or $f'_d(n_0+1)\neq 0$ then $f_d$ has 2 real roots in the interval $(n_0,n_0+2)$; 
\item if $f(n_0+1) = 1$ then $f_d(n_0+1) = 0$.
\end{itemize}
In particular, if $\deg f = 1$ then $f_d$ factorizes into linear factors with real coefficients.
\end{cor}

\begin{proof}
Consider the associated sequence $(\widetilde{a}_n)_{n\in\N}$. This sequence is given by the formula $\widetilde{a}_0 = c, \widetilde{a}_n = -f(n)\widetilde{a}_{n-1} + c, n>0$. Then use Theorem \ref{thm9} substituting the sequence $(\widetilde{a}_n)_{n\in\N}$ in the place of $(a_n)_{n\in\N}$ and substituting $-f$ in the place of $f$.
\end{proof}

Let us consider the sequence of derangements. In the definition of derangements we have polynomial $f=X$. By Corollary \ref{cor3}, if $d\geq 4$ then the polynomial $f_d$ (of degree $d-1$) has exactly $d-1$ real roots and exactly one rational root 1 ($n_0$ from Corollary \ref{cor3} is equal to 0 for $f=X$). It suffices to compute $f'_d(1)$.
\begin{equation*}
f'_d(1) = \sum_{j=1}^{d-1} (-1)^j \sum_{s=0}^{j-1}\prod_{i=0, i\neq s}^{j-1} (1-i) = -1 + \sum_{j=2}^{d-1} (-1)^j \cdot (-1)^{j-2} (j-2)! = -1 + \sum_{j=0}^{d-3} j! > 0
\end{equation*}

Corollary \ref{cor3} shows us that all the roots of $f_d$ apart from 1 are noninteger. Since the leading coefficient of $f_d$ is $\pm 1$, thus by theorem on rational roots of a polynomial with integer coefficients, all the rational roots of $f_d$ must be integer. Hence all the roots of $f_d$ apart from 1 are irrational.

From the equation $\widehat{f}_d = \frac{f_d}{X-1}$ for $d\geq 3$ we see that all the complex roots of $\widehat{f}_d$ are real and if $d\geq 4$ then they are irrational. In spite of reduciblity of $\widehat{f}_d = \frac{f_d}{X-1}$, $d\geq 3$, into linear factors over $\R$, we do not know how $\widehat{f}_d = \frac{f_d}{X-1}$ factorizes over $\Q$. Numerical computations show that for each $d\leq 20$ the polynomial $\widehat{f}_d$ is irreducible over $\Q$. In the light of these results we can formulate the following question:

\begin{que}
Is the polynomial $\widehat{f}_d = -1 + X\sum_{j=0}^{d-3} (-1)^j \prod_{i=0}^{j-1} (X-2-i)$ irreducible over $\Q$ for each integer $d\geq 3$?
\end{que}

\subsection{Divisors of a sequence ${\bf a}\in\cal{R}$}\label{subsec3.4}

\subsubsection{\bf Prime divisors of a sequence ${\bf a}\in\cal{R}$}\label{subsubsec3.4.1}

In Section \ref{subsubsec3.1.3} we showed that $n-1\mid D_n$ for each nonnegative integer $n$. Moreover, we proved that there are infinitely many prime numbers $p$ such that $p\mid\frac{D_n}{n-1}$ for some integer $n>1$. Now we will give some conditions for infinitude of the set $\cal{P}_{\bf a}=\{p\in\bbb{P}:\exists_{n\in\N}\mbox{ } p\mid a_n\}$, where ${\bf a}=(a_n)_{n\in\N}\in\cal{R}$.

\begin{thm}
Let ${\bf a}={\bf a}(f,h_1,h_2)$ be an unbounded sequence. If $f$ has a nonnegative integer root then the set $\cal{P}_{\bf a}$ is infinite.
\end{thm}

\begin{proof}
Note that there must be such prime number $p$ that $p\mid a_n$ for some $n\in\N$. Otherwise $|a_n| = 1$ for all $n\in\N$. Assume that there are only finitely many prime divisors of the numbers $a_n$, $n\in\N$. Let us denote these divisors by $p_1, p_2, ..., p_s$.

Let $n_0\in\N$ be a root of $f$. Then by Proposition \ref{prop4}, for each prime number $p$ and positive integer $k$ the sequence $(a_n\pmod{p^k})_{n\in\N_{n_0}}$ is periodic of period $p^k(p-1)$. Since there are only finitely prime divisors of numbers $a_n$, $n\in\N$ then $a_n\neq 0$ for all $n\in\N$. In particular, $a_{n_0}\neq 0$. Then $a_{n_0} = \pm p_1^{\alpha_1}p_2^{\alpha_2}\cdot ...\cdot p_s^{\alpha_s}$ for some $\alpha_1, \alpha_2, ..., \alpha_s \in\N$. Without loss of generality assume that $a_{n_0} = p_1^{\alpha_1}p_2^{\alpha_2}\cdot ...\cdot p_s^{\alpha_s}$. Then, by periodicity of the sequence $(a_n\pmod{p_1^{\alpha_1+2}p_2^{\alpha_2+2}\cdot ...\cdot p_s^{\alpha_s+2}})_{n\in\N_{n_0}}$ we get that
\begin{equation*}
a_{n_0+jp_1^{\alpha_1+2}(p_1-1)p_2^{\alpha_2+2}(p_2-1)\cdot ...\cdot p_s^{\alpha_s+2}(p_s-1)} \equiv p_1^{\alpha_1}p_2^{\alpha_2}\cdot ...\cdot p_s^{\alpha_s} \pmod{p_1^{\alpha_1+2}p_2^{\alpha_2+2}\cdot ...\cdot p_s^{\alpha_s+2}}
\end{equation*}
for each $j\in\N$. Then $|a_{n_0+jp_1^{\alpha_1+2}(p_1-1)p_2^{\alpha_2+2}(p_2-1)\cdot ...\cdot p_s^{\alpha_s+2}(p_s-1)}| = p_1^{\alpha_1}p_2^{\alpha_2}\cdot ...\cdot p_s^{\alpha_s}$. However, if $a_{n_0+jp_1^{\alpha_1+2}(p_1-1)p_2^{\alpha_2+2}(p_2-1)\cdot ...\cdot p_s^{\alpha_s+2}(p_s-1)} = -p_1^{\alpha_1}p_2^{\alpha_2}\cdot ...\cdot p_s^{\alpha_s}$ then there must be
\begin{equation*}
2p_1^{\alpha_1}p_2^{\alpha_2}\cdot ...\cdot p_s^{\alpha_s} \equiv 0 \pmod{p_1^{\alpha_1+2}p_2^{\alpha_2+2}\cdot ...\cdot p_s^{\alpha_s+2}},
\end{equation*}
which means that $p_1^2p_2^2\cdot ...\cdot p_s^2\mid 2$ - a contradiction. Hence
\begin{equation*}
a_{n_0+jp_1^{\alpha_1+2}(p_1-1)p_2^{\alpha_2+2}(p_2-1)\cdot ...\cdot p_s^{\alpha_s+2}(p_s-1)} = p_1^{\alpha_1}p_2^{\alpha_2}\cdot ...\cdot p_s^{\alpha_s}
\end{equation*}
for all $j\in\N$. Then by Theorem \ref{thm8} the sequence $(a_n)_{n\in\N}$ is bounded, which is a contradiction with the assumption of unboundedness of $(a_n)_{n\in\N}$.
\end{proof}

\begin{thm}
Let ${\bf a}$ be an unbounded sequence of the form ${\bf a}(f,h_1,1)$ or ${\bf a}(f,h_1,-1)$. If for each prime number $p$ there are such integers $n,m$ that $p\mid f(n)$ and $p\nmid h_1(m)$ then the set $\cal{P}_{\bf a}$ is infinite.
\end{thm}

\begin{proof}
Note that there must be such prime number $p$ that $p\mid a_n$ for some $n\in\N$. Otherwise $|a_n| = 1$ for all $n\in\N$. Assume that there are only finitely many prime divisors of the numbers $a_n$, $n\in\N$. Let us denote these divisors by $p_1, p_2, ..., p_s$.

Because each sequence defined by the equations $a_0 = h_1(0), a_n = f(n)a_{n-1}+(-1)^n h_1(n), n>0$ has the associated sequence given by the formula $\widetilde{a}_0 = h_1(0), \widetilde{a}_n = -f(n)\widetilde{a}_{n-1}+h_1(n), n>0$ and $\widetilde{a}_n = (-1)^na_n$, $n\in\N$, thus it suffices to prove the statement for a sequence of the form ${\bf a}(f,h_1,1)$.

For each $i\in\{1,2,...,s\}$ there is such $n_i\in\{0,1,2,...,p_i-1\}$ that $p_i\mid f(n_i)$. Hence the result from Section \ref{subsubsec3.1.1} shows us that the sequence $(a_n\pmod{p_i^2})_{n\geq n_i+p_i}$ is periodic of period $p_i^2$. Let $m_i>n_i+p_i$ be such an integer that $p_i\nmid h_1(m_i)$. The formula $a_{m_i} = f(m_i)a_{m_i-1} + h_1(m_i)$ implies that $p_i\nmid a_{m_i-1}$ or $p_i\nmid a_{m_i}$.

Let $\alpha_i\in\{m_i-1,m_i\}$ be such that $p_i\nmid a_{\alpha_i}$. By Chinese remainder theorem there exists such an integer $\alpha_0\geq\max\{n_1+p_1,n_2+p_2,...,n_s+p_s\}$ that $\alpha_0 \equiv \alpha_i \pmod{p_i^2}$ for each $i\in\{1,2,...,s\}$. Then for all $j\in\N$ we have
\begin{equation*}
a_{\alpha_0 + jp_1^2p_2^2\cdot ...\cdot p_s^2} \equiv a_{\alpha_0} \pmod{p_1^2p_2^2\cdot ...\cdot p_s^2}.
\end{equation*}
Moreover, $a_{\alpha_0}\equiv a_{\alpha_i}\pmod{p_i^2}$ for $i\in\{1,2,...,s\}$. Hence $p_i\nmid a_{\alpha_0}$ for any $i\in\{1,2,...,s\}$ and this means that $a_{\alpha_0}=\pm 1$. Without loss of generality assume that $a_{\alpha_0}=1$. Then for all $j\in\N$ there must be $a_{\alpha_0 + jp_1^2p_2^2\cdot ...\cdot p_s^2} = 1$. Indeed, $a_{\alpha_0 + jp_1^2p_2^2\cdot ...\cdot p_s^2}$ has no prime divisors, so $a_{\alpha_0 + jp_1^2p_2^2\cdot ...\cdot p_s^2}=\pm 1$. Suppose that $a_{\alpha_0 + jp_1^2p_2^2\cdot ...\cdot p_s^2} = -1$ for some $j\in\N$. Then $-1 \equiv 1 \pmod{p_1^2p_2^2\cdot ...\cdot p_s^2}$, which means that $p_1^2p_2^2\cdot ...\cdot p_s^2\mid 2$ and this is a contradiction. Thus $a_{\alpha_0 + jp_1^2p_2^2\cdot ...\cdot p_s^2} = 1$ for any $j\in\N$ and by Theorem \ref{thm8} the sequence $(a_n)_{n\in\N}$ is bounded, which stays in a contradiction with the assumption of unboundedness of $(a_n)_{n\in\N}$. Hence the set $\cal{P}_{\bf a}$ is infinite.
\end{proof}

\begin{prop}
If ${\bf a}={\bf a}(f,h_1,0)$ then the set $\cal{P}_{\bf a}$ is infinite if and only if $h_1(0)=0$ or $\deg f\neq 0$.
\end{prop}

\begin{proof}
For all $n\in\N$ we have $a_n = h_1(0)\prod_{i=1}^n f(i)$. Thus if $\deg f\neq 0$ then there are infinitely many prime divisors of the numbers $f(n)$, $n\in\N$, and as a consequence, there are infinitely many prime divisors of the numbers $a_n$, $n\in\N$.

If $h_1(0)\neq 0$ and $f=b$, where $b\in\Z\bs\{0\}$ then $a_n = h_1(0)b^n$ for $n\in\N$ and all prime divisors of the numbers $a_n$, $n\in\N$, are divisors of $h_1(0)$ and $b$.
\end{proof}

\begin{prop}\label{prop9}
If ${\bf a}={\bf a}(b,c,d)$, where $b,c,d\in\Z$ and $bd\neq 0$ (i.e. $f=b$, $h_1=c$ and $h_2=d$) then each prime number $p$ divides $a_n$ for some $n\in\N$. Moreover, for any $k\in\N_+$ there is such $n_k\in\N$ that $p^k\mid a_{n_k}$.
\end{prop}

\begin{proof}
For all $n\in\N$ we have
\begin{equation*}
a_n = c\sum_{j=0}^n b^jd^{n-j} =
\begin{cases}
c\frac{b^{n+1}-d^{n+1}}{b-d}, &\mbox{ if } b\neq d \\
b^nc(n+1), &\mbox{ if } b = d
\end{cases}.
\end{equation*}

If $b=d$ then the statement is certainly true. If $b\neq d$ then by Euler's theorem, for any prime number $p$ and positive integer $k$ we have the divisibility $p^k\mid b^{p^{k-1}(p-1)}-d^{p^{k-1}(p-1)}$. Hence for any prime number $p$ and positive integer $k$ there exists such $n_k\in\N$ that $p^k\mid a_{n_k}$.
\end{proof}

\begin{thm}
If a sequence ${\bf a}$ is of the form ${\bf a}(b_1X+b_0,c,1)$ or ${\bf a}(b_1X+b_0,c,-1)$, where $b_0,b_1,c\in\Z$ and $b_1\neq 0$, then the set $\cal{P}_{\bf a}$ is infinite.
\end{thm}

\begin{proof}
Because each sequence defined by the equations $a_0 = c, a_n = (b_1n+b_0)a_{n-1}+(-1)^nc, n>0$ has the associated sequence given by the formula $\widetilde{a}_0 = c, \widetilde{a}_n = -(b_1n+b_0)\widetilde{a}_{n-1}+c, n>0$ and $\widetilde{a}_n = (-1)^na_n$, $n\in\N$, thus it suffices to prove the statement for a sequence defined by the formula $a_0 = c, a_n = (b_1n+b_0)a_{n-1}+c,\newline n>0$.

Similarly, if $a_0 = c, a_n = (b_1n+b_0)a_{n-1}+c, n>0$ then $a_n = c\widehat{a}_n$ for $n\in\N$, where $\widehat{a}_0 = 1, \widehat{a}_n = (b_1n+b_0)\widehat{a}_{n-1}+1, n>0$. Hence it suffices to prove the statement for a sequence defined by the formula $a_0 = 1, a_n = (b_1n+b_0)a_{n-1}+1, n>0$.

Assume that there are only finitely many prime divisors of numbers $a_n$, $n\in\N$. Let $p_1 < p_2 < ... < p_s$ be all prime divisors of numbers $a_n$, $n\in\N$, which do not divide $b_1$. Let $q_1 < q_2 < ... < q_t$ be all prime divisors of numbers $a_n$, $n\in\N$, which divide $b_1$ (and so do not divide $b_0$, because otherwise we would have $a_n \equiv 1 \pmod{q_i}$ for all $n\in\N$) and do not divide $b_0-1$. Let $r_1 < r_2 < ... < r_u$ be all prime divisors of numbers $a_n$, $n\in\N$, which divide $b_1$ and $b_0-1$.

For a given $i\in\{1,2,...,s\}$, since $p_i\nmid b_1$, thus there exists such $n_{p_i}\in\N$ that $p_i\mid b_1n_{p_i}+b_0$. Hence the sequence $(a_n\pmod{p_i})_{n\geq n_{p_i}}$ is periodic of period $p_i$ (see Section \ref{subsubsec3.1.1}). From the definition of $a_n$ we conclude that if $p_i\mid a_n$ then $p_i\nmid a_{n+1}$. Thus the number of solutions modulo $p_i$ of the congruence $a_n \equiv 0 \pmod{p_i}$ (where $n\geq n_{p_i}$) is less than or equal to $\left\lfloor\frac{p_i}{2}\right\rfloor$. For $p_i>2$ we have $1+\left\lfloor\frac{p_i}{2}\right\rfloor = \frac{p_i+1}{2}<p_i$. For $p_1=2$, by induction on $n$ we can prove that
\begin{itemize}
\item if $b_1 \equiv 1 \pmod{4}$ then
\begin{equation}\label{eq15}
a_n \equiv
\begin{cases}
1, &\mbox{ for } n\equiv -b_0, 2-b_0\pmod{4} \\
2, &\mbox{ for } n\equiv 1-b_0\pmod{4} \\
0, &\mbox{ for } n\equiv 3-b_0\pmod{4} \\
\end{cases}
\end{equation}
for $n\geq (-b_0\pmod{4})$;
\item if $b_1 \equiv -1 \pmod{4}$ then
\begin{equation}\label{eq16}
a_n \equiv
\begin{cases}
1, &\mbox{ for } n\equiv b_0, 2+b_0\pmod{4} \\
0, &\mbox{ for } n\equiv 1+b_0\pmod{4} \\
2, &\mbox{ for } n\equiv 3+b_0\pmod{4} \\
\end{cases}
\end{equation}
for $n\geq (b_0\pmod{4})$.
\end{itemize}

Let us observe that for a given $i\in\{1,2,...,t\}$, $a_n \equiv \sum_{j=0}^n b_0^j = \frac{b_0^{n+1}-1}{b_0-1} \pmod{q_i}$ (because \newline$q_i\nmid b_0-1$). Hence the sequence $(a_n\pmod{q_i})_{n\in\N}$ has the basic period equal to the order of $b_0$ in the multiplicative group $(\Z /q_i\Z)^*$ of nonzero remainders from division by $q_i$. Let us denote this order by $ord_i$. Then $q_i\mid a_n$ if and only if $ord_i\mid n+1$.

For a given $i\in\{1,2,...,u\}$, by induction on $n$ we get $a_n \equiv n+1 \pmod{r_i}$ for $n\in\N$.

For a given $i\in\{1,2,...,s\}$, if $p_i>2$ then we take such $\alpha_i\in\{0,1,...,p_i-2\}$ that $p_i\nmid a_{jp_i+\alpha_i}$ for any $j\in\N_+$. Since $\alpha_i\neq p_i-1$, thus for each $k\in\{1,2,...,t\}$, if $p_i\mid ord_k$ then $jp_i+\alpha_i \not\equiv -1\pmod{ord_k}$ for all $j\in\N_+$ and as a result, $q_k\nmid a_{jp_i+\alpha_i}$.

For a given $i\in\{1,2,...,u\}$ and for each $k\in\{1,2,...,t\}$, if $r_i\mid ord_k$ then $jr_i \not\equiv -1\pmod{ord_k}$ for all $j\in\N_+$ and as a result, $q_k\nmid a_{jr_i}$.

Let us consider two cases, if $p_1=2$ (it is sure that one of these cases holds):
\begin{itemize}
\item[a)] $2\nmid a_{2j}$ for any $j\in\N_+$; then we take $\alpha_1 = 0$ and for each $k\in\{1,2,...,t\}$, if $2\mid ord_k$ then $2j \not\equiv -1\pmod{ord_k}$ for all $j\in\N_+$ and as a result, $q_k\nmid a_{2j}$;
\item[b)] $2\mid a_{2j}$ for all $j\in\N_+$; then we take such $\alpha_1 \in\{0,2\}$ that $a_{4j+\alpha_1} \equiv 2 \pmod{4}$ for $j\in\N_+$, thus for each $k\in\{1,2,...,t\}$, if $2\mid ord_k$ then $4j+\alpha_1 \not\equiv -1\pmod{ord_k}$ for all $j\in\N_+$ and as a result, $q_k\nmid a_{4j+\alpha_1}$.
\end{itemize}

By Chinese remainder theorem there exists such $\alpha_0$ that $\alpha_0 \equiv \alpha_i \pmod{p_i}$ for all $i\in\{1,2,...,s\}$ (respectively $\alpha_0 \equiv \alpha_1 \pmod{4}$ and $\alpha_0 \equiv \alpha_i \pmod{p_i}$ for all $i\in\{2,3,...,s\}$, if b) holds), $\alpha_0 \equiv 0 \pmod{ord_i}$ for all $i\in\{1,2,...,s\}$ and $\alpha_0 \equiv 0 \pmod{r_i}$ for all $i\in\{1,2,...,u\}$. Then
\begin{equation*}
a_{jp_1\cdot ...\cdot p_s ord_1\cdot ...\cdot ord_t r_1\cdot ...\cdot r_u + \alpha_0} \equiv a_{\alpha_0} \pmod{p_1\cdot ...\cdot p_s q_1\cdot ...\cdot q_t r_1\cdot ...\cdot r_u}
\end{equation*}
for all $j\in\N_+$ and $|a_{\alpha_0}|=1$. Respectively, if b) holds then
\begin{equation*}
a_{j\cdot 4p_2\cdot ...\cdot p_s ord_1\cdot ...\cdot ord_t r_1\cdot ...\cdot r_u + \alpha_0} \equiv a_{\alpha_0} \pmod{4p_2\cdot ...\cdot p_s q_1\cdot ...\cdot q_t r_1\cdot ...\cdot r_u}.
\end{equation*}
for all $j\in\N_+$ and $|a_{\alpha_0}|=2$.

If 2 is not the only prime divisor of the numbers $a_n$, $n\in\N$, or 2 is not a prime divisor of a number $a_n$ for any $n\in\N$ then $a_{jp_1\cdot ...\cdot p_s ord_1\cdot ...\cdot ord_t r_1\cdot ...\cdot r_u + \alpha_0} = a_{\alpha_0}$ for all $j\in\N_+$. Indeed, if $a_{jp_1\cdot ...\cdot p_s ord_1\cdot ...\cdot ord_t r_1\cdot ...\cdot r_u + \alpha_0} = -a_{\alpha_0}$ for some $j\in\N_+$ then
\begin{equation*}
2a_{\alpha_0} \equiv 0\pmod{p_1\cdot ...\cdot p_s q_1\cdot ...\cdot q_t r_1\cdot ...\cdot r_u},
\end{equation*}
which with the fact $|a_{\alpha_0}|=1$ implies that $p_1\cdot ...\cdot p_s q_1\cdot ...\cdot q_t r_1\cdot ...\cdot r_u\mid 2$ - a contradiction. Similarly, when the case b) takes place, if $a_{j\cdot 4p_2\cdot ...\cdot p_s ord_1\cdot ...\cdot ord_t r_1\cdot ...\cdot r_u + \alpha_0} = -a_{\alpha_0}$ for some $j\in\N_+$ then
\begin{equation*}
2a_{\alpha_0} \equiv 0\pmod{4p_2\cdot ...\cdot p_s q_1\cdot ...\cdot q_t r_1\cdot ...\cdot r_u},
\end{equation*}
which with the fact $|a_{\alpha_0}|=2$ implies that $4p_2\cdot ...\cdot p_s q_1\cdot ...\cdot q_t r_1\cdot ...\cdot r_u\mid 4$ - a contradiction. Finally, we use Theorem \ref{thm8} to conclude that the sequence $(a_n)_{n\in\N}$ is bounded.

Assume now that 2 is the only prime divisor of the numbers $a_n$, $n\in\N$. Then there are two possible situations:
\begin{itemize}
\item $2\nmid b_1$; then by the equations (\ref{eq15}) and (\ref{eq16}) we know that there is such number $n_2\in\N$ that $a_{2j+n_2}\equiv 1\pmod{4}$ for all $j\in\N$ and because 2 is the only prime divisor of the numbers $a_n$, $n\in\N$, hence $a_{2j+n_2}=1$ for all $j\in\N$ and by Theorem \ref{thm8} the sequence $(a_n)_{n\in\N}$ is bounded;
\item $2\mid b_1$ and $2\nmid b_0$; then we can compute that
\begin{equation*}
\begin{split}
(a_n\pmod{4})_{n\in\N} & =
\begin{cases}
(1,2,3,0,1,2,3,0,...), &\mbox{ if } 4\mid b_1 \mbox{ and } b_0\equiv 1\pmod{4} \\
(1,0,1,0,1,0,1,0,...), &\mbox{ if } 4\mid b_1 \mbox{ and } b_0\equiv 3\pmod{4} \\
(1,0,1,0,1,0,1,0,...), &\mbox{ if } 4\nmid b_1 \mbox{ and } b_0\equiv 1\pmod{4} \\
(1,2,3,0,1,2,3,0,...), &\mbox{ if } 4\nmid b_1 \mbox{ and } b_0\equiv 3\pmod{4}
\end{cases} \\
& =\begin{cases}
(1,2,3,0,1,2,3,0,...), &\mbox{ if } 4\mid b_1+b_0-1 \\
(1,0,1,0,1,0,1,0,...), &\mbox{ if } 4\nmid b_1+b_0-1
\end{cases},
\end{split}
\end{equation*}
which with the fact that 2 is the only prime divisor of the numbers $a_n$, $n\in\N$, implies that $a_{4j}=1$ for all $j\in\N$ and by Theorem \ref{thm8} the sequence $(a_n)_{n\in\N}$ is bounded.
\end{itemize}

As a consequence of our reasoning, we obtain in all cases the sequence $(a_n)_{n\in\N}$ is bounded. On the other hand, if $a_n\neq 0$ then
\begin{equation*}
|a_{n+1}|=|(b_1(n+1)+b_0)a_n+1|\geq |(b_1(n+1)+b_0)|-1\geq |b_1|(n+1)-|b_0|-1.
\end{equation*}

If $a_n=0$ then $a_{n+1}=1$ and
\begin{equation*}
|a_{n+2}|=|(b_1(n+2)+b_0)a_{n+1}+1|\geq |(b_1(n+2)+b_0)|-1\geq |b_1|(n+2)-|b_0|-1.
\end{equation*}

This means that the sequence $(a_n)_{n\in\N}$ is unbounded - a contradiction. Hence there must be infinitely many prime divisors of the sequence $(a_n)_{n\in\N}$.
\end{proof}

The results on infinitude of the set of prime divisors of a sequence $(a_n)_{n\in\N}$ given by the formula $a_0 = h_1(0), a_n = f(n)a_{n-1} + h_1(n)h_2(n)^n, n>0$ suggest us to state the following:

\begin{con}
If a sequence ${\bf a}\in\cal{R}$ is unbounded and not of the form $(a_n)_{n\in\N} = (cb^n)_{n\in\N}$ for some $b,c\in\Z$ then the set $\cal{P}_{\bf a}$ is infinite.
\end{con}

Using a result of Luca (see \cite{Luc}) we can establish the conjecture above for sequences of the form ${\bf a}(f,h_1,c)$, where $f,h_1\in\Z[X]$ and $c\in\Z$. Namely, Luca showed that if a sequence $(a_n)_{n\in\N}$ of rational numbers satisfy a recurrence of the type
\begin{equation}\label{eq25}
F(n)a_{n+2}+G(n)a_{n+1}+H(n)a_n=0
\end{equation}
for some $F,G,H\in\Z[X]$ not all zero and do not exist such $u,v,w\in\Z$ not all zero that
\begin{equation*}
ua_{n+2}+va_{n+1}+wa_n=0
\end{equation*}
for $n\gg 0$ (we call that $(a_n)_{n\in\N}$ is not binary recurrent for sufficiently large $n$) then there exists a constant $\gamma>0$ depending only on the sequence $(a_n)_{n\in\N}$ such that the product of numerators and denominators of all the nonzero numbers $a_n$ for $n\leq N$ has at least $\gamma\log N$ prime divisors as $N\gg 0$ (by $\log$ we mean the natural logarithm).

Let a sequence $(a_n)_{n\in\N}$ be given by the formula $a_0=h_1(0), a_n=f(n)a_{n-1}+c^nh_1(n), n>0$, where $f,h_1\in\Z[X]$ and $c\in\Z$. If $h_1=0$ then the sequence $(a_n)_{n\in\N}$ is constantly equal to $0$, so assume additionaly that $h_1\neq 0$. Write the recurrence formula for $n+1$ and $n+2$:
\begin{equation*}
\begin{split}
a_{n+1}= & f(n+1)a_n+c^{n+1}h_1(n+1) \\
a_{n+2}= & f(n+2)a_{n+1}+c^{n+2}h_1(n+2)
\end{split}
\end{equation*}
After multiplication the first equality by $ch_1(n+2)$ and the second one by $-h_1(n+1)$ and adding them, we obtain
\begin{equation*}
ch_1(n+2)a_{n+1}-h_1(n+1)a_{n+2}=ch_1(n+2)f(n+1)a_n-h_1(n+1)f(n+2)a_{n+1}
\end{equation*}
or equivalently
\begin{equation*}
h_1(n+1)a_{n+2}-(ch_1(n+2)+f(n+2)h_1(n+1))a_{n+1}+cf(n+1)h_1(n+2)a_n=0.
\end{equation*}
Hence the sequence $(a_n)_{n\in\N}$ satisfy a recurrence of the form (\ref{eq25}) and $h_1(X+1)\neq 0$. By the result of Luca we know that if $(a_n)_{n\in\N}$ is not binary recurrent for sufficiently large $n$ then there exists a constant $\gamma>0$ such that the product of all the nonzero numbers $a_n$ for $n\leq N$ has at least $\gamma\log N$ prime divisors as $N\gg 0$.

Let us establish when the sequence $(a_n)_{n\in\N}$ is binary recurrent for sufficiently large $n$. Then, for $n\gg 0$ we have two relations:
\begin{equation*}
\begin{split}
& h_1(n+1)a_{n+2}-(ch_1(n+2)+f(n+2)h_1(n+1))a_{n+1}+cf(n+1)h_1(n+2)a_n=0, \\
& ua_{n+2}+va_{n+1}+wa_n=0.
\end{split}
\end{equation*}
We multiply the first equation by $u$ and the second one by $-h_1(n+1)$ and then we add them to obtain
\begin{equation}
(cuf(n+1)h_1(n+2)-wh_1(n+1))a_n=(cuh_1(n+2)+uf(n+2)h_1(n+1)+vh_1(n+1))a_{n+1}.
\end{equation}

First, let us consider the case when $cuf(X+1)h_1(X+2)-wh_1(X+1)=0$ and $cuh_1(n+2)+uf(n+2)h_1(n+1)+vh_1(n+1)=0$. Then
\begin{equation*}
cuh_1(X+2)+uf(X+2)h_1(X+1)=-vh_1(X+1)
\end{equation*}
and
\begin{equation*}
cuf(X+1)h_1(X+2)=wh_1(X+1).
\end{equation*}
Since $h_1\neq 0$, thus
\begin{equation}\label{eq26}
\frac{cuf(X+1)h_1(X+2)}{h_1(X+1)}=w,
\end{equation}
but then
\begin{equation*}
w=\lim_{n\rightarrow +\infty}\frac{cuf(n+1)h_1(n+2)}{h_1(n+1)}=\lim_{n\rightarrow +\infty} cuf(n+1).
\end{equation*}
This means that $f=b$, where $b\in Z$. The equality (\ref{eq26}) takes the form $\frac{bcuh_1(X+2)}{h_1(X+1)}=w$ and this implies that $h_1=d$, where $d\in\Z$. If $bc=0$ then the sequence $(a_n)_{n\in\N}$ becomes a geometric progression and it has only finitely many prime divisors on condition that it has no zero terms. If $b,c\neq 0$ then by Proposition \ref{prop9} any prime number divides $a_n$ for some $n\in\N$.

Next, if exactly one of the polynomials $cuf(X+1)h_1(X+2)-wh_1(X+1)$, $cuh_1(n+2)+uf(n+2)h_1(n+1)+vh_1(n+1)$ is zero then the sequence $(a_n)_{n\in\N}$ is ultimately equal to zero.

Finally, assume that both the polynomials $cuf(X+1)h_1(X+2)-wh_1(X+1)$, $cuh_1(n+2)+uf(n+2)h_1(n+1)+vh_1(n+1)$ are nonzero. For simplicity of notation let us write them as $P(X)$ and $Q(X)$, respectively. If $\frac{P(X)}{Q(X)}=\alpha\in\Q$ then $a_{n+1}=\alpha a_n$, i. e. $(a_n)_{n\in\N}$ is a geometric progression for $n\gg 0$ and then it can have only finitely many prime divisors. If $\frac{P(X)}{Q(X)}$ is a nonconstant rational function then by the following lemma there are infinitely many prime numbers $p$ such that $p$ divides numerator or denominator of irreducible form of number $\frac{P(n)}{Q(n)}$ for some $n\in\N$. Therefore, since $a_{n+1}=\frac{P(n)}{Q(n)}a_n$ for $n\gg 0$, thus the sequence $(a_n)_{n\in\N}$ has infinitely many prime divisors.

\begin{lem*}
Let $P,Q\in\Z[X]\bs\{0\}$. If $\frac{P}{Q}$ is a nonconstant rational function then there are infinitely many prime numbers $p$ such that $v_p\left(\frac{P(n)}{Q(n)}\right)\neq 0$ for some $n\in\N$.
\end{lem*}

\begin{proof}
Assume that $p_1,...,p_s$ are the only prime numbers occuring in irreducible forms of numbers $\frac{P(n)}{Q(n)}$, $n\in\N$ (if certainly $Q(n)\neq 0$). Let $n_0\in\N$ be such that all the integer roots of polynomials $P$ and $Q$ are less than $n_0$. For $i\in\{1,...,s\}$ we define $k_i=v_{p_i}(P(n_0))$, $l_i=v_{p_i}(Q(n_0))$ and $m_i=\max\{k_i,l_i\}$. Then for each $n\equiv n_0\pmod{p_1^{m_1+1}\cdot ...\cdot p_s^{m_s+1}}$ we have $P(n)=p_1^{k_1}\cdot ...\cdot p_s^{k_s}R(n)$ and $Q(n)=p_1^{l_1}\cdot ...\cdot p_s^{l_s}R(n)$, where $R(n)$ is coprime to $p_1\cdot ...\cdot p_s$. Therefore $\frac{P(n)}{Q(n)}=p_1^{k_1-l_1}\cdot ...\cdot p_s^{k_s-l_s}$ for any $n\equiv n_0\pmod{p_1^{m_1+1}\cdot ...\cdot p_s^{m_s+1}}$. Since the equation $P(n)=p_1^{k_1-l_1}\cdot ...\cdot p_s^{k_s-l_s}Q(n)$ holds for infinitely many $n\in\N$ it holds for all $n\in\N$ and the rational function $\frac{P}{Q}$ is constant. However, this is a contradiction with assumption of Lemma.
\end{proof}

\subsubsection{\bf Divisors of the form $n-b-1$}\label{subsubsec3.4.2}

Now we will consider sequences of type ${\bf a}(X-b,h_1,h_2)$, i. e. given by the recurrence
\begin{equation}\label{def2}
a_0=h_1(0), a_n = (n-b)a_{n-1} + h_1(n)h_2(n)^n, n>0,
\end{equation}
where $b\in\Z$ and $h_1,h_2\in\Z[X]$. We will try to give some conditions for division $a_n$ by $n-b-1$. Notice that for $b=0$, $h_1=1$ and $h_2=-1$ we obtain the sequence of derangements and then the divisibility $n-1\mid D_n$ holds for any $n\in\N$. Our goal is to generalize this property for the sequences defined by the relation (\ref{def2}).

If we use (\ref{def2}) twice then we obtain:
\begin{equation*}
\begin{split}
& a_n = (n-b)a_{n-1} + h_1(n)h_2(n)^n = \\
& = (n-b-1)a_{n-1} + (n-b-1)a_{n-2} + h_1(n-1)h_2(n-1)^{n-1} + h_1(n)h_2(n)^n, n\geq 2.
\end{split}
\end{equation*}
and, what is more, $a_1 = (1-b)h_1(0) + h_1(1)h_2(1)$.

Since polynomials are periodic modulo $n-b-1$, thus for $n\in\N_+$ there holds the equivalence
\begin{equation}\label{eq17}
n-b-1\mid a_n \iff n-b-1\mid h_1(b)h_2(b)^{n-1} + h_1(b+1)h_2(b+1)^n.
\end{equation}

In particular, we have the following:
\begin{prop}\label{prop10}
If $h_1(b) = h_1(b+1)h_2(b+1) = 0$ then $n-b-1\mid a_n$ for all positive integers $n$. If $h_2(b) = h_1(b+1)h_2(b+1) = 0$ then $n-b-1\mid a_n$ for all integers $n\geq 2$.
\end{prop}

Assume that $n\geq 2$ and $|n-b-1|$ is a prime number. Then, by (\ref{eq17}), $n-b-1\mid a_n$ if and only if one of the following holds:
\begin{itemize}
\item $n-b-1\mid h_1(b)h_2(b)$, $h_1(b+1)h_2(b+1)$;
\item $n-b-1\nmid h_1(b)h_2(b)$, $h_1(b+1)h_2(b+1)$ \mbox{ and } 
$v_{|n-b-1|}(h_1(b)h_2(b)^{b+1} + h_1(b+1)h_2(b+1)^{b+2})>0$ \mbox{ (by Fermat's little theorem)}.
\end{itemize}

The fact presented above implies the following:
\begin{prop}
If $h_1(b)h_2(b)^{b+1} + h_1(b+1)h_2(b+1)^{b+2} \neq 0$ then there are only finitely many nonnegative integers $n$ such that $|n-b-1|$ is a prime number and $n-b-1\mid a_n$. If $h_1(b)h_2(b)^{b+1} + h_1(b+1)h_2(b+1)^{b+2} = 0$ then $n-b-1\mid a_n$ for almost all nonnegative integers $n$ such that $|n-b-1|$ is a prime number. 
\end{prop}

\begin{proof}
First, consider the case $h_1(b)h_2(b)^{b+1} + h_1(b+1)h_2(b+1)^{b+2} \neq 0$. Since there are only finitely many prime numbers dividing simultaneously $h_1(b)h_2(b)$ and $h_1(b+1)h_2(b+1)$ and prime numbers $p$ such that
\begin{equation*}
v_p(h_1(b)h_2(b)^{b+1} + h_1(b+1)h_2(b+1)^{b+2})>0,
\end{equation*}
hence the first part of the statement is true.

Assume now that $h_1(b)h_2(b)^{b+1} + h_1(b+1)h_2(b+1)^{b+2} = 0$. If $h_1(b)h_2(b) = h_1(b+1)h_2(b+1) = 0$ then by Proposition \ref{prop10} we have $n-b-1\mid a_n$ for all integers $n>1$. If $h_1(b)h_2(b), h_1(b+1)h_2(b+1) \neq 0$ and $|n-b-1|$ is a prime number not dividing either $h_1(b)h_2(b)$ or $h_1(b+1)h_2(b+1)$ then $n-b-1\mid a_n$.
\end{proof}

Let us consider the case when $h_1(b)h_2(b)\neq 0$ and $h_1(b+1)h_2(b+2) = 0$. Try to characterize indices $n\in\N_+$ such that $n-b-1\mid a_n$. Let
\begin{equation*}
h_2(b) = p_1^{\alpha_1}p_2^{\alpha_2}\cdot ...\cdot p_s^{\alpha_s},
\end{equation*}
where $\alpha_1,...,\alpha_s\in\N_+$, be the factorization of $h_2(b)$ and
\begin{equation*}
h_1(b) = p_1^{\beta_1}p_2^{\beta_2}\cdot ...\cdot p_s^{\beta_s}q_1^{\gamma_1}q_2^{\gamma_2}\cdot ... \cdot q_t^{\gamma_t},
\end{equation*}
where $\beta_1,...,\beta_s\in\N$ and $\gamma_1,...,\gamma_s\in\N_+$, be the factorization of $h_1(b)$.

Then, by (\ref{eq17}), $n-b-1\mid a_n$ if and only if
\begin{equation*}
|n-b-1| = p_1^{\delta_1}p_2^{\delta_2}\cdot ...\cdot p_s^{\delta_s}q_1^{\eps_1}q_2^{\eps_2}\cdot ... \cdot q_t^{\eps_t},
\end{equation*}
where $\delta_i \leq (n-1)\alpha_i + \beta_i$ for $i\in\{1,2,...,s\}$ and $\eps_i \leq \gamma_i$ for $i\in\{1,2,...,t\}$.

If $b\geq 0$ then for $i\in\{1,2,...,s\}$ and $n\geq\max\{2,b\}$ we have
\begin{equation*}
\delta_i < p_i^{\delta_i} \leq |n-b-1| \leq n-1 \leq (n-1)\alpha_i + \beta_i.
\end{equation*}
Hence, if $n\geq\max\{2,b\}$ then:
\begin{equation*}
n-b-1\mid a_n\Leftrightarrow |n-b-1| = p_1^{\delta_1}p_2^{\delta_2}\cdot ...\cdot p_s^{\delta_s}q_1^{\eps_1}q_2^{\eps_2}\cdot ... \cdot q_t^{\eps_t},
\end{equation*}
where $\delta_i\in\N$ for $i\in\{1,2,...,s\}$ and $\eps_i \leq \gamma_i$ for $i\in\{1,2,...,t\}$. It remains to check indices $n\in\{1,2,3,...,b\}$ one by one.

If $b<0$ then the situation is slightly more difficult. Fix $i\in\{1,2,...,s\}$. If $p_i^{\delta_i}-\delta_i \geq -b\alpha-\beta$ then
\begin{equation*}
\delta_i \leq p_i^{\delta_i} + b\alpha_i + \beta_i \leq n-b-1 + b\alpha_i + \beta_i \leq (n-1)\alpha_i + b(\alpha_i-1) + \beta_i \leq (n-1)\alpha_i + \beta_i.
\end{equation*}
thus the inequality $\delta_i \leq (n-1)\alpha_i + \beta_i$ holds for $p_i^{\delta_i} - \delta_i\geq - b\alpha_i - \beta_i$. In particular, if $\delta_i \geq -b\alpha_i-\beta_i$ then
\begin{equation*}
p_i^{\delta_i} - \delta_i \geq p_i^{\delta_i} - p_i^{\delta_i-1} \geq p_i^{\delta_i-1} \geq \delta_i \geq - b\alpha_i - \beta_i
\end{equation*}
and the inequality $\delta_i \leq (n-1)\alpha_i + \beta_i$ holds.

We can summarize our discussion in the following:

\begin{prop}
Let us consider a sequence ${\bf a}(X-b,h_1,h_2)$. Assume that $h_1(b+1)h_2(b+1) = 0$. Let $h_2(b) = p_1^{\alpha_1}p_2^{\alpha_2}\cdot ...\cdot p_s^{\alpha_s}$ and $h_1(b) = p_1^{\beta_1}p_2^{\beta_2}\cdot ...\cdot p_s^{\beta_s}q_1^{\gamma_1}q_2^{\gamma_2}\cdot ... \cdot q_t^{\gamma_t}$ be the factorizations of numbers $h_2(b)$ and $h_1(b)$, respectively. Then for $n\in\N_+$,
\begin{equation*}
n-b-1\mid a_n\Leftrightarrow |n-b-1| = p_1^{\delta_1}p_2^{\delta_2}\cdot ...\cdot p_s^{\delta_s}q_1^{\eps_1}q_2^{\eps_2}\cdot ... \cdot q_t^{\eps_t},
\end{equation*}
where $\delta_i \leq (n-1)\alpha_i + \beta_i$ for $i\in\{1,2,...,s\}$ and $\eps_i \leq \gamma_i$ for $i\in\{1,2,...,t\}$. In particular,
\begin{itemize}
\item if $b\geq 0$ and $n\geq\max\{2,b\}$ then $n-b-1\mid a_n$ if and only if
\begin{equation*}
|n-b-1| = p_1^{\delta_1}p_2^{\delta_2}\cdot ...\cdot p_s^{\delta_s}q_1^{\eps_1}q_2^{\eps_2}\cdot ... \cdot q_t^{\eps_t},
\end{equation*}
where $\delta_i\in\N$ for $i\in\{1,2,...,s\}$ and $\eps_i \leq \gamma_i$ for $i\in\{1,2,...,t\}$;
\item if $b<0$ and $n>0$ then $n-b-1\mid a_n$ when
\begin{equation*}
|n-b-1| = p_1^{\delta_1}p_2^{\delta_2}\cdot ...\cdot p_s^{\delta_s}q_1^{\eps_1}q_2^{\eps_2}\cdot ... \cdot q_t^{\eps_t},
\end{equation*}
where $\eps_i \leq \gamma_i$ for $i\in\{1,2,...,t\}$ and for each $i\in\{1,2,...,s\}$ at least one of the equalities $p_i^{\delta_i} - \delta_i\geq - b\alpha_i - \beta_i$, $\delta_i\leq\beta_i$ holds.
\end{itemize}
\end{prop}

The same consideration allows us to obtain the characterization of indices $n$ such that $n-b-1\mid a_n$ in case when $h_1(b)h_2(b) = 0$ and $h_1(b+1)h_2(b+2)\neq 0$ and in case when $h_2(b)=h_2(b+1)=c$.

\begin{prop}
Let us consider a sequence ${\bf a}(X-b,h_1,h_2)$. Assume that $h_1(b)h_2(b) = 0$. Let $h_2(b+1) = p_1^{\alpha_1}p_2^{\alpha_2}\cdot ...\cdot p_s^{\alpha_s}$ and $h_1(b+1) = p_1^{\beta_1}p_2^{\beta_2}\cdot ...\cdot p_s^{\beta_s}q_1^{\gamma_1}q_2^{\gamma_2}\cdot ... \cdot q_t^{\gamma_t}$ be the factorizations of numbers $h_2(b+1)$ and $h_1(b+1)$, respectively. Then for $n\in\N_+$,
\begin{equation*}
n-b-1\mid a_n\Leftrightarrow |n-b-1| = p_1^{\delta_1}p_2^{\delta_2}\cdot ...\cdot p_s^{\delta_s}q_1^{\eps_1}q_2^{\eps_2}\cdot ... \cdot q_t^{\eps_t},
\end{equation*}
where $\delta_i \leq (n-1)\alpha_i + \beta_i$ for $i\in\{1,2,...,s\}$ and $\eps_i \leq \gamma_i$ for $i\in\{1,2,...,t\}$. In particular,
\begin{itemize}
\item if $b\geq 0$ and $n\geq\max\{2,b\}$ then $n-b-1\mid a_n$ if and only if
\begin{equation*}
|n-b-1| = p_1^{\delta_1}p_2^{\delta_2}\cdot ...\cdot p_s^{\delta_s}q_1^{\eps_1}q_2^{\eps_2}\cdot ... \cdot q_t^{\eps_t},
\end{equation*}
where $\delta_i\in\N$ for $i\in\{1,2,...,s\}$ and $\eps_i \leq \gamma_i$ for $i\in\{1,2,...,t\}$;
\item if $b<0$ and $n>0$ then $n-b-1\mid a_n$ when
\begin{equation*}
|n-b-1| = p_1^{\delta_1}p_2^{\delta_2}\cdot ...\cdot p_s^{\delta_s}q_1^{\eps_1}q_2^{\eps_2}\cdot ... \cdot q_t^{\eps_t},
\end{equation*}
where $\eps_i \leq \gamma_i$ for $i\in\{1,2,...,t\}$ and for each $i\in\{1,2,...,s\}$ at least one of the equalities $p_i^{\delta_i} - \delta_i\geq - b\alpha_i - \beta_i$, $\delta_i\leq\beta_i$ holds.
\end{itemize}
\end{prop}

\begin{prop}
Let us consider a sequence ${\bf a}(X-b,h_1,h_2)$. Assume that $h_2(b)=h_2(b+1)=c\in\Z$. Let $c = p_1^{\alpha_1}p_2^{\alpha_2}\cdot ...\cdot p_s^{\alpha_s}$ and $h_1(b)+ch_1(b+1) = p_1^{\beta_1}p_2^{\beta_2}\cdot ...\cdot p_s^{\beta_s}q_1^{\gamma_1}q_2^{\gamma_2}\cdot ... \cdot q_t^{\gamma_t}$ be the factorizations of numbers $c$ and $h_1(b)+ch_1(b+1)$, respectively. Then for $n\in\N_+$,
\begin{equation*}
n-b-1\mid a_n\Leftrightarrow |n-b-1| = p_1^{\delta_1}p_2^{\delta_2}\cdot ...\cdot p_s^{\delta_s}q_1^{\eps_1}q_2^{\eps_2}\cdot ... \cdot q_t^{\eps_t},
\end{equation*}
where $\delta_i \leq (n-1)\alpha_i + \beta_i$ for $i\in\{1,2,...,s\}$ and $\eps_i \leq \gamma_i$ for $i\in\{1,2,...,t\}$. In particular,
\begin{itemize}
\item if $b\geq 0$ and $n\geq\max\{2,b\}$ then $n-b-1\mid a_n$ if and only if
\begin{equation*}
|n-b-1| = p_1^{\delta_1}p_2^{\delta_2}\cdot ...\cdot p_s^{\delta_s}q_1^{\eps_1}q_2^{\eps_2}\cdot ... \cdot q_t^{\eps_t},
\end{equation*}
where $\delta_i\in\N$ for $i\in\{1,2,...,s\}$ and $\eps_i \leq \gamma_i$ for $i\in\{1,2,...,t\}$;
\item if $b<0$ and $n>0$ then $n-b-1\mid a_n$ when
\begin{equation*}
|n-b-1| = p_1^{\delta_1}p_2^{\delta_2}\cdot ...\cdot p_s^{\delta_s}q_1^{\eps_1}q_2^{\eps_2}\cdot ... \cdot q_t^{\eps_t},
\end{equation*}
where $\eps_i \leq \gamma_i$ for $i\in\{1,2,...,t\}$ and for each $i\in\{1,2,...,s\}$ at least one of the equalities $p_i^{\delta_i} - \delta_i\geq - b\alpha_i - \beta_i$, $\delta_i\leq\beta_i$ holds.
\end{itemize}
If $c=0$ or $h_1(b)+ch_1(b+1) = 0$ then $n-b-1\mid a_n$ for each $n>0$.
\end{prop}

\begin{proof}
It suffices to see that if $n>0$ then $n-b-1\mid a_n\Leftrightarrow n-b-1\mid c^{n-1}h_1(b)+c^nh_1(b+1)$.
\end{proof}

Let us assume now that $h_1(b)=h_1(b+1)=c\neq 0$ and $h_2(b+1)=\pm 1$.

\begin{prop}
Let us consider a sequence ${\bf a}(X-b,h_1,h_2)$. Assume that $h_1(b)=h_1(b+1)=c\neq 0$ and $h_2(b+1)=1$. If $n_1,n_2,n_3\in\N_+$ are such that $n_1-b-1\mid a_{n_1}$, $n_2-b-1\mid a_{n_2}$, $n_3-b-1\mid\lcm(n_1-b-1, n_2-b-1)$, $v_2(n_1-1)=v_2(n_2-1)=v_2(n_3-1)$ and $\lcm(n_1-1,n_2-1)\mid n_3-1$ then $n_3-b-1\mid a_{n_3}$. In particular, if $b>0$, $a_{b+1}=0$ and $n_2\in\N_+$ is such that $v_2(n_2-1)=v_2(b)$ and $n_2-b-1\mid a_{n_2}$ then $n_3-b-1\mid a_{n_3}$ for any $n_3\in\N_+$ such that $\lcm(b,n_2-1)\mid n_3-1$ and $v_2(n_3-1)=v_2(b)$.
\end{prop}

\begin{proof}
By (\ref{eq17}) we know that $n-b-1\mid a_n$ if and only if $n-b-1\mid ch_2(b)^{n-1}+c$. Since $v_2(n_1-1)=v_2(n_2-1)=v_2(n_3-1)$, thus $\frac{n_3-1}{n_1-1}$ and $\frac{n_3-1}{n_2-1}$  are odd numbers. Hence we get
\begin{equation*}
h_2(b)^{n_1-1}+1,h_2(b)^{n_2-1}+1\mid h_2(b)^{n_3-1}+1.
\end{equation*}
As a consequence we get
\begin{equation*}
\begin{split}
& n_3-b-1\mid\eqlcm(n_1-b-1, n_2-b-1)\mid \eqlcm(c(h_2(b)^{n_1-1}+1),c(h_2(b)^{n_2-1}+1))\mid \\
& \mid c(h_2(b)^{n_3-1}+1),
\end{split}
\end{equation*}
which means that $n_3-b-1\mid a_{n_3}$.
\end{proof}

\begin{prop}
Let us consider a sequence ${\bf a}(X-b,h_1,h_2)$. Assume that $h_1(b)=h_1(b+1)=c\neq 0$ and $h_2(b+1)=-1$. If $n_1,n_2,n_3\in\N_+$ are such that $n_1-b-1\mid a_{n_1}$, $n_2-b-1\mid a_{n_2}$ $n_3-b-1\mid\lcm(n_1-b-1, n_2-b-1)$ and $\lcm(n_1-1,n_2-1)\mid n_3-1$ then $n_3-b-1\mid a_{n_3}$. In particular, if $b>0$, $a_{b+1}=0$ and $n_2\in\N_+$ is such that $n_2-b-1\mid a_{n_2}$ then $n_3-b-1\mid a_{n_3}$ for any $n_3\in\N_+$ such that $\lcm(b,n_2-1)\mid n_3-1$.
\end{prop}

\begin{proof}
By (\ref{eq17}) we know that $n-b-1\mid a_n$ if and only if $n-b-1\mid ch_2(b)^{n-1}+(-1)^nc$. Since $\eqlcm(n_1-1, n_2-1)\mid n_3-1$, thus
\begin{equation*}
\begin{split}
& h_2(b)^{n_i-1}-(-1)^{n_i-1}\mid h_2(b)^{\elcm(n_1-1,n_2-1)}-(-1)^{\elcm(n_1-1,n_2-1)}\mid \\
& \mid h_2(b)^{n_3-1}-(-1)^{n_3-1}.
\end{split}
\end{equation*}
for $i\in\{1,2\}$. Finally,
\footnotesize
\begin{equation*}
\begin{split}
& n_3-b-1\mid\eqlcm(n_1-b-1, n_2-b-1)\mid \eqlcm(c(h_2(b)^{n_1-1}-(-1)^{n_1-1}),c(h_2(b)^{n_2-1}-(-1)^{n_2-1}))\mid \\
& \mid c(h_2(b)^{\elcm(n_1-1,n_2-1)}-(-1)^{\elcm(n_1-1,n_2-1)})\mid c(h_2(b)^{n_3-1}-(-1)^{n_3-1}),
\end{split}
\end{equation*}
\normalsize
which means that $n_3-b-1\mid a_{n_3}$.
\end{proof}

Similarly we can prove analogous propositions for $h_1(b)=h_1(b+1)=c\neq 0$ and $h_2(b)=\pm 1$.

\begin{prop}\label{prop17}
Let us consider a sequence ${\bf a}(X-b,h_1,h_2)$. Assume that $h_1(b)=h_1(b+1)=c\neq 0$ and $h_2(b)=1$. If $n_1,n_2,n_3\in\N_+$ are such that $n_1-b-1\mid a_{n_1}$, $n_2-b-1\mid a_{n_2}$, $n_3-b-1\mid\lcm(n_1-b-1, n_2-b-1)$, $v_2(n_1)=v_2(n_2)=v_2(n_3)$ and $\lcm(n_1,n_2)\mid n_3$ then $n_3-b-1\mid a_{n_3}$. In particular, if $b\geq 0$, $a_{b+1}=0$ and $n_2\in\N_+$ is such that $v_2(n_2)=v_2(b+1)$ and $n_2-b-1\mid a_{n_2}$ then $n_3-b-1\mid a_{n_3}$ for any $n_3\in\N_+$ such that $\lcm(b+1,n_2)\mid n_3$ and $v_2(n_3)=v_2(b+1)$.
\end{prop}

\begin{prop}\label{prop18}
Let us consider a sequence ${\bf a}(X-b,h_1,h_2)$. Assume that $h_1(b)=h_1(b+1)=c\neq 0$ and $h_2(b)=-1$. If $n_1,n_2,n_3\in\N_+$ are such that $n_1-b-1\mid a_{n_1}$, $n_2-b-1\mid a_{n_2}$, $n_3-b-1\mid\lcm(n_1-b-1, n_2-b-1)$ and $\lcm(n_1,n_2)\mid n_3$ then $n_3-b-1\mid a_{n_3}$. In particular, if $b\geq 0$, $a_{b+1}=0$ and $n_2\in\N_+$ is such that $n_2-b-1\mid a_{n_2}$ then $n_3-b-1\mid a_{n_3}$ for any $n_3\in\N_+$ divisible by $\lcm(b+1,n_2)$.
\end{prop}

\begin{rem}
If $b=0$ then due to Proposition \ref{prop18,1} below we see that Propositions \ref{prop17} and \ref{prop18} are useful only in the case when $a_1=0$ beacause there holds:
\end{rem}

\begin{prop}\label{prop18,1}
If $n_1,n_2,n_3\in\N_+$ are such that $n_3-1\mid\lcm(n_1-1, n_2-1)$ and $\lcm(n_1,n_2)\mid n_3$ then $n_1\mid n_2$ or $n_2\mid n_1$. Additionaly, if $n_1,n_2>1$ then $n_3=\max\{n_1,n_2\}$.
\end{prop}

\begin{proof}
Let us write $d=\gcd(n_1,n_2)$, $n_1=dn'_1$, $n_2=dn'_2$ and $n_3=kdn'_1n'_2$ for some $n'_1,n'_2,k\in\N_+$, where $\gcd(n'_1,n'_2)=1$. Since $n_3-1\mid (n_1-1)(n_2-1)$, thus we have
\begin{equation}\label{eq19}
kdn'_1n'_2 - 1 \mid d^2n'_1n'_2 - dn'_1 - dn'_2 + 1.
\end{equation}
If $d=qk+r$ for some $q\in\N$ and $r\in\{0,1,2,...,k-1\}$ then the divisibility (\ref{eq19}) takes the form
\begin{equation}\label{eq19,5}
kdn'_1n'_2 - 1 \mid qkdn'_1n'_2 + rdn'_1n'_2 - dn'_1 - dn'_2 + 1.
\end{equation}
Since $kdn'_1n'_2 - 1 \mid qkdn'_1n'_2 - q$, hence (\ref{eq19,5}) is equivalent to
\begin{equation}\label{eq20}
kdn'_1n'_2 - 1 \mid q + rdn'_1n'_2 - dn'_1 - dn'_2 + 1.
\end{equation}

If the number $kdn'_1n'_2 - 1$ is equal to 0 then $d=n'_1=n'_2=1$, which means that $n_1=n_2=1$. Hence, assume that one of the numbers $n_1,n_2$ is greater than 1. Then $kdn'_1n'_2 - 1>0$.

Let us consider the case $r=0$. If $q - dn'_1 - dn'_2 + 1 > 0$ then by (\ref{eq20}) there must be
\begin{equation*}
\begin{split}
 & kdn'_1n'_2 - 1 \leq q - dn'_1 - dn'_2 + 1 \\
\Longrightarrow\quad  & dn'_1n'_2 + dn'_1 + dn'_2 \leq kdn'_1n'_2 + dn'_1 + dn'_2 \leq q+2 \leq d+2 \\
\Longrightarrow\quad & dn'_1n'_2 + dn'_1 + dn'_2 + d \leq 2d+2 \\
\Longrightarrow\quad & d(n'_1+1)(n'_2+1) - 2d \leq 2 \\
\Longrightarrow\quad & d[(n'_1+1)(n'_2+1) - 2] \leq 2,
\end{split}
\end{equation*}
but $d[(n'_1+1)(n'_2+1) - 2] \geq d(2\cdot 2 - 2) = 2d \geq 2$ (because $\min\{d,n'_1,n'_2\}\geq 1$), hence $d=n'_1=n'_2=1$. Then $n_1=n_2=1$, but we assumed that one of the numbers $n_1,n_2$ is greater than 1 - a contradiction.

If $q - dn'_1 - dn'_2 + 1 < 0$ then
\begin{equation*}
\begin{split}
 & kdn'_1n'_2 - 1 \leq |q - dn'_1 - dn'_2 + 1| = - q + dn'_1 + dn'_2 - 1 \\
\Longrightarrow\quad & kdn'_1n'_2 - dn'_1 - dn'_2 \leq -q \leq 0 \\
\Longrightarrow\quad & (k-1)dn'_1n'_2 + dn'_1n'_2 - dn'_1 - dn'_2 + d \leq d \\
\Longrightarrow\quad & (k-1)dn'_1n'_2 + d(n'_1-1)(n'_2-1) - d \leq 0 \\
\Longrightarrow\quad & d[(k-1)n'_1n'_2 + (n'_1-1)(n'_2-1) - 1] \leq 0,
\end{split}
\end{equation*}
which is possible only for $(k,n'_1,n'_2)\in\{(1,1,x),(1,x,1),(1,2,2),(2,1,1): x\in\N_+\}$ (the condition $(k,n'_1,n'_2)=(1,2,2)$ does not hold because $n'_1,n'_2$ are coprime).

If $q - dn'_1 - dn'_2 + 1 = 0$ then $d+1\geq q+1 = d(n'_1+n'_2)$, which means that $1\geq d(n'_1+n'_2-1)$ and this holds only for $d,n'_1,n'_2=1$. Then $n_1=n_2=1$ - a contradiction with the assumption that one of the numbers $n_1,n_2$ is greater than 1.

Let us consider the case $r=1$. If $q + dn'_1n'_2 - dn'_1 - dn'_2 + 1 > 0$ then by (\ref{eq20}) there must be
\begin{equation*}
\begin{split}
 & kdn'_1n'_2 - 1 \leq q + dn'_1n'_2 - dn'_1 - dn'_2 + 1 \\
\Longrightarrow\quad & (k-1)dn'_1n'_2 + dn'_1 + dn'_2 \leq q+2.
\end{split}
\end{equation*}
Since $dn'_1 + dn'_2 \leq (k-1)dn'_1n'_2 + dn'_1 + dn'_2$ and $q+2 \leq d+2$ we obtain the following chain of inequalities.
\begin{equation*}
\begin{split}
 & dn'_1 + dn'_2 \leq d+2 \\
\Longrightarrow\quad & dn'_1 + dn'_2 - d \leq 2 \\
\Longrightarrow\quad & d(n'_1+n'_2-1) \leq 2
\end{split}
\end{equation*}
The last inequality holds only if $(d,n'_1,n'_2)\in\{(1,1,2),(1,2,1),(2,1,1)\}$ because $\min\{d,n'_1,n'_2\}\geq 1$. However, for $d=2$, since $r=1$ and $k\in\N_+$, thus $k=1\mid d$ and $r=0$ - a contradiction.

If $q + dn'_1n'_2 - dn'_1 - dn'_2 + 1 < 0$ then
\begin{equation*}
\begin{split}
 & kdn'_1n'_2 - 1 \leq |q + dn'_1n'_2 - dn'_1 - dn'_2 + 1| = - q - dn'_1n'_2 + dn'_1 + dn'_2 - 1 \\
\Longrightarrow\quad & (k+1)dn'_1n'_2 - dn'_1 - dn'_2 \leq -q \leq 0, \\
\Longrightarrow\quad & kdn'_1n'_2 + dn'_1n'_2 - dn'_1 - dn'_2 + d - d \leq 0, \\
\Longrightarrow\quad & d[kn'_1n'_2 + (n'_1-1)(n'_2-1) - 1] \leq 0,
\end{split}
\end{equation*}
which is possible only for $k=n'_1=n'_2=1$.

If $q + dn'_1n'_2 - dn'_1 - dn'_2 + 1 = 0$ then
\begin{equation*}
\begin{split}
 & dn'_1n'_2 - dn'_1 - dn'_2 + d - d + 1 = -q \leq 0 \\
\Longrightarrow\quad & d[(n'_1-1)(n'_2-1)-1]+1 \leq 0,
\end{split}
\end{equation*}
which holds only for $n'_1=1$ or $n'_2=1$. Without loss of generality assume that $n'_2=1$. Then $q + dn'_1n'_2 - dn'_1 - dn'_2 + 1 = q + dn'_1 - dn'_1 - d + 1 = q-d+1 = 0$. Hence $q+1=d=qk+1$. If $q=0$ then $d=1$ and $n_2=1$. If $q>0$ then $k=1$.

Assume now that $r>1$. Then $q + rdn'_1n'_2 - dn'_1 - dn'_2 + 1 > 0$, so by (\ref{eq20}) we have
\begin{equation*}
\begin{split}
 & kdn'_1n'_2 - 1 \leq q + rdn'_1n'_2 - dn'_1 - dn'_2 + 1 \\
\Longrightarrow\quad & (k-r)dn'_1n'_2 + dn'_1 + dn'_2 \leq q+2 \leq d+2 \\
\Longrightarrow\quad & d[(k-r)n'_1n'_2 + n'_1 + n'_2 - 1] \leq 2.
\end{split}
\end{equation*}
Since all the numbers $k-r,d,n'_1,n'_2$ are $\geq 1$ and one of the values $d,n'_1,n'_2$ is $>2$ (because one of the values $n_1,n_2$ is $>2$), thus $d[(k-r)n'_1n'_2 + n'_1 + n'_2 - 1] > 2$ - a contradiction.

Summing up our discussion, we see that if the assumptions of the proposition are satisfied then $n'_1=1$ or $n'_2=1$, which means that $\gcd(n_1,n_2)=d$ is equal to $n_1$ or $n_2$. Thus $n_1\mid n_2$ or $n_2\mid n_1$.

If $n_1,n_2>1$ then one of the following conditions holds:
\begin{itemize}
\item $r=0$ and $q - dn'_1 - dn'_2 + 1 < 0$;
\item $r=1$ and $q + dn'_1n'_2 - dn'_1 - dn'_2 + 1 = 0$.
\end{itemize}
However, then we have $k=1$, which means that $n_3 = kdn'_1n'_2 = \eqlcm(n_1,n_2) = \max\{n_1,n_2\}$.
\end{proof}

\pagebreak

\section{Arithmetic properties of the sequences of even and odd derangements}\label{sec4}

The knowledge of arithmetic properties of the sequence of derangements suggests us to explore the sequences of numbers of even and odd derangements. Let us denote by $D_n^{(e)}$ the number of all even derangements of an $n$-element set, i. e. the number of all even permutations of a set with $n$ elements which have no fixed points. Similarly, let $D_n^{(o)}$ denote the number of all odd derangements of a set with $n$ elements. The sequences $(D_n^{(e)})_{n\in\N}$ and $(D_n^{(o)})_{n\in\N}$ satisfy the following system of recurrence relations
\begin{equation}\label{eq23}
\begin{split}
D_0^{(e)}=1,\mbox{ } D_1^{(e)}=0,\mbox{ } & D_0^{(o)}=0,\mbox{ } D_1^{(o)}=0, \\
D_n^{(e)}=(n-1)(D_{n-2}^{(o)}+D_{n-1}^{(o)}),\mbox{ } & D_n^{(o)}=(n-1)(D_{n-2}^{(e)}+D_{n-1}^{(e)}), n\geq 2.
\end{split}
\end{equation}

The first terms of the sequences $(D_n^{(e)})_{n\in\N}$ and $(D_n^{(o)})_{n\in\N}$ are presented in the table below.
\begin{equation*}
\begin{tabular}{c|ccccccccccc}
n & 0 & 1 & 2 & 3 & 4 & 5 & 6 & 7 & 8 & 9 & 10 \\ \hline
$D_n^{(e)}$ & 1 & 0 & 0 & 2 & 3 & 24 & 130 & 930 & 7413 & 66752 & 667476 \\
$D_n^{(o)}$ & 0 & 0 & 1 & 0 & 6 & 20 & 135 & 924 & 7420 & 66744 & 667485
\end{tabular}
\end{equation*}

The relation $D_n^{(e)}=(n-1)(D_{n-2}^{(o)}+D_{n-1}^{(o)}), n\geq 2,$ is provided by the following argument. Each even derangement $\sigma$ of the set $\{1,2,...,n\}$ can be written in the form $\sigma = \sigma\circ (n, \sigma(n))\circ (n, \sigma(n))$, where $\sigma\circ (n, \sigma(n))$ is an odd permutation with one or two fixed points. Thus $\sigma\circ (n, \sigma(n))$ can be treated as a derangement of a set with $n-2$ or $n-1$ elements, so it can be chosen in $D_{n-2}^{(o)}+D_{n-1}^{(o)}$ ways. Furthermore, the number $\sigma(n)$ can be chosen in $n-1$ ways. The relation $D_n^{(e)}=(n-1)(D_{n-2}^{(o)}+D_{n-1}^{(o)}), n\geq 2,$ can be explained in the same way.

Certainly $D_n^{(o)}+D_n^{(e)}=D_n$ for $n\in\N$. From the relations (\ref{eq23}) we can compute $D_n^{(o)}-D_n^{(e)}$ which will allow us to write $D_n^{(o)}$ and $D_n^{(e)}$ as expressions dependent only on $D_n$.

\begin{prop}
$D_n^{(o)}-D_n^{(e)}=(-1)^n(n-1)$ for $n\in\N$.
\end{prop}

\begin{proof}
Obviously the equality is true for $n\in\{0,1\}$. Assume now that $n\geq 2$ and the equality holds for $n-2$ and $n-1$. Then
\begin{equation*}
\begin{split}
& D_n^{(o)}-D_n^{(e)}=(n-1)[D_{n-2}^{(e)}+D_{n-1}^{(e)}-D_{n-2}^{(o)}-D_{n-1}^{(o)}] \\
= & (n-1)[D_{n-2}^{(e)}-D_{n-2}^{(o)}+D_{n-1}^{(e)}-D_{n-1}^{(o)}] \\
= & (n-1)[-(-1)^{n-2}(n-3)-(-1)^{n-1}(n-2)] \\
= & (-1)^n(n-1)[-(n-3)+(n-2)]=(-1)^n(n-1).
\end{split}
\end{equation*}
\end{proof}

\begin{cor}\label{cor6}
$D_n^{(o)} = \frac{D_n+(-1)^n(n-1)}{2} = \frac{1}{2}\left[\sum_{j=0}^{n-1} \frac{n!}{j!} (-1)^j + (-1)^nn\right]$,\newline $D_n^{(e)} = \frac{D_n-(-1)^n(n-1)}{2} = \frac{1}{2}\left[\sum_{j=0}^{n-1} \frac{n!}{j!} (-1)^j - (-1)^n(n-2)\right]$ for $n\in\N$.
\end{cor}

The formulae in Corollary \ref{cor6} combined with properties of the sequence of derangements will give us properties of the sequences of even and odd derangements.

First, we establish the asymptotics of $(D_n^{(e)})_{n\in\N}$ and $(D_n^{(o)})_{n\in\N}$. Directly from Corollary \ref{cor6} and the fact that $\lim_{n\rightarrow +\infty}\frac{D_n}{n!}=\frac{1}{e}$ we obtain the following:

\begin{prop}\label{prop18,2}
$\lim_{n\rightarrow +\infty}\frac{D_n^{(e)}}{n!}=\lim_{n\rightarrow +\infty}\frac{D_n^{(o)}}{n!}=\frac{1}{2e}$
\end{prop}

Now we show periodicity of sequences of remainders $(D_n^{(o)}\pmod{d})_{n\in\N}$ and $(D_n^{(e)}\pmod{d})_{n\in\N}$ for a given positive integer $d>1$ and compute their basic period.

\begin{prop}
Let $d>1$ be a positive integer. Then the sequences $(D_n^{(o)}\pmod{d})_{n\in\N}$ and $(D_n^{(e)}\pmod{d})_{n\in\N}$ are periodic with period $2d$.
\end{prop}

\begin{proof}
We will present the proof only for the sequence $(D_n^{(o)}\pmod{d})_{n\in\N}$ because the same argument allows us to claim that $(D_n^{(e)}\pmod{d})_{n\in\N}$ is periodic of period $2d$.

Recall from Section \ref{subsec3.1} that the basic period of the sequence of remainders $(D_n\pmod{d})_{n\in\N}$ is equal to $d$ if $2\mid d$ and $2d$ otherwise. The sequence $((-1)^n\pmod{d})_{n\in\N}$ has period $2$ for $d>2$ and is constant for $d=2$. Obviously, the sequence $(n-1\pmod{d})_{n\in\N}$ has period $d$. Thus, if $2\nmid d$ then the sequence $(2D_n^{(o)}\pmod{d})_{n\in\N_2}=(D_n+(-1)^n(n-1)\pmod{d})_{n\in\N}$ is periodic of period $2d$, so does the sequence $(D_n^{(o)}\pmod{d})_{n\in\N_2}$. If $2\mid d$ then $2d$ is a period of the sequence $(2D_n^{(o)}\pmod{2d})_{n\in\N}=(D_n+(-1)^n(n-1)\pmod{2d})_{n\in\N}$, hence the sequence $(D_n^{(o)}\pmod{d})_{n\in\N_2}$ has period $2d$.
\end{proof}

The relations (\ref{eq23}) show us that $n-1\mid D_n^{(e)}, D_n^{(o)}$ for each $n\in\N$. Let us put $E_n^{(e)}=\frac{D_n^{(e)}}{n-1}=D_{n-2}^{(o)}+D_{n-1}^{(o)}$ and $E_n^{(o)}=\frac{D_n^{(o)}}{n-1}=D_{n-2}^{(e)}+D_{n-1}^{(e)}$ for $n\geq 2$. Recall that in Section \ref{subsubsec3.1.3} we defined $E_n=\frac{D_n}{n-1}$ for $n\geq 2$. Using Corollary \ref{cor6} we can write
\begin{equation*}
E_n^{(e)} = \frac{E_n-(-1)^n}{2}, E_n^{(o)} = \frac{E_n+(-1)^n}{2}, n\geq 2.
\end{equation*}
Let $\widetilde{E}_n^{(e)}=(-1)^nE_n^{(e)}$ and $\widetilde{E}_n^{(o)}=(-1)^nE_n^{(o)}$ for $n\geq 2$. Using Corollary \ref{cor6}, for $n\geq 2$ we obtain
\begin{equation*}
\begin{split}
\widetilde{E}_n^{(e)}= & (-1)^n(D_{n-2}^{(o)}+D_{n-1}^{(o)})=\frac{1}{2}(\widetilde{D}_{n-2}+(n-3)-\widetilde{D}_{n-1}-(n-2)) \\
= & \frac{1}{2}(\widetilde{D}_{n-2}-\widetilde{D}_{n-1}-1)=\frac{1}{2}(\widetilde{E}_n-1)
\end{split}
\end{equation*}
and similarly
\begin{equation*}
\begin{split}
\widetilde{E}_n^{(o)}= & (-1)^n(D_{n-2}^{(e)}+D_{n-1}^{(e)})=\frac{1}{2}(\widetilde{D}_{n-2}-(n-3)-\widetilde{D}_{n-1}+(n-2)) \\
= & \frac{1}{2}(\widetilde{D}_{n-2}-\widetilde{D}_{n-1}+1)=\frac{1}{2}(\widetilde{E}_n+1),
\end{split}
\end{equation*}
where $\widetilde{E}_n=(-1)^nE_n=\widetilde{D}_{n-2}-\widetilde{D}_{n-1}$. Since for each positive integer $d>1$ the sequence $(\widetilde{D}_n\pmod{d})_{n\in\N}$ is periodic of period $d$, hence the sequences $(\widetilde{E}_n^{(e)}\pmod{d})_{n\in\N}$ and $(\widetilde{E}_n^{(o)}\pmod{d})_{n\in\N}$ are periodic of period $d$, if $2\nmid d$, and $2d$ otherwise. Hence, if $2\nmid d$ and $n_1\equiv n_2\pmod{d}$, or $2\mid d$ and $n_1\equiv n_2\pmod{2d}$ then $d\mid E_{n_1}^{(e)}\iff d\mid E_{n_2}^{(e)}$ and $d\mid E_{n_1}^{(o)}\iff d\mid E_{n_2}^{(o)}$. If we combine this with the fact that $E_2^{(e)}=E_3^{(o)}=0$ then we obtain the following divisibilities.

\begin{prop}\label{prop18,3}
The following divisibilities hold:
\begin{itemize}
\item If $2\nmid d$ then $d\mid \frac{D_{md+2}^{(e)}}{md+1}, \frac{D_{md+3}^{(o)}}{md+2}$ for all $m\in\N$. In particular, $d(md+1)\mid D_{md+2}^{(e)}$ and $d(md+2)\mid D_{md+3}^{(o)}$.
\item If $2\mid d$ then $d\mid \frac{D_{2md+2}^{(e)}}{2md+1}, \frac{D_{2md+3}^{(o)}}{2md+2}$ for all $m\in\N$. In particular, $d(2md+1)\mid D_{2md+2}^{(e)}$ and $d(2md+2)\mid D_{2md+3}^{(o)}$.
\end{itemize}
\end{prop}

If we substitute $d=p^k$, where $p$ is a prime number and $k$ is a positive integer then we infer that $p^k$ divides numbers $E_n^{(e)}$ and $E_m^{(o)}$ for infinitely many $n,m\in\N$. This means that each prime number $p$ divides numbers $E_n^{(e)}$ and $E_m^{(o)}$ for some $n,m\in\N$ and $p$-adic valuation on the sets $\{E_n^{(e)}\}_{n\in\N_2}$ and $\{E_n^{(o)}\}_{n\in\N_2}$ is unbounded. These facts are in striking opposition to results on prime divisors and $p$-adic valuations of numbers $E_n$, $n\in\N$. Let us recall that not each prime number $p$ divides $E_n$ for some $n\in\N$ and even when $p$ does divide $E_n$ for some $n\in\N$ then $p$-adic valuation can be bounded (this situation happens for $p=2633$).

Despite the fact that $p$-adic valuation of numbers $E_n^{(e)}$ and $E_n^{(o)}$, $n\in\N$, is unbounded, one can check its behavior. This task is simple, since the sequences $(2E_n^{(e)})_{n\in\N_2}$ and $(2E_n^{(o)})_{n\in\N_2}$ have pseudo-polynomial decomposition modulo $p$ on the set $\{n\in\N: n\geq 2\}$. For $n\geq 2$ we have
\begin{equation}\label{eq23,5}
\begin{split}
& 2E_n^{(e)}=2(-1)^n\widetilde{E}_n^{(e)}=(-1)^n(\widetilde{E}_n-1) \\
& 2E_n^{(o)}=2(-1)^n\widetilde{E}_n^{(o)}=(-1)^n(\widetilde{E}_n+1) \\
\end{split}
\end{equation}
and $(f_{p,k}(X-2)-f_{p,k}(X-1),1)_{k\in\N_2}$ is a pseudo-polynomial decomposition of $(\widetilde{E}_n)_{n\in\N_2}$, where $f_{p,k} = \sum_{j=0}^{kp-1} (-1)^j \prod_{i=0}^{j-1} (X-i), k>1$ (see Section \ref{subsubsec3.1.3}). Hence $(f_{p,k}(X-2)-f_{p,k}(X-1)-1,(-1)^n)_{k\in\N_2}$ is a pseudo-polynomial decomposition of $(2E_n^{(e)})_{n\in\N_2}$ and $(f_{p,k}(X-2)-f_{p,k}(X-1)+1,(-1)^n)_{k\in\N_2}$ is a pseudo-polynomial decomposition of $(2E_n^{(o)})_{n\in\N_2}$.

\begin{thm}\label{thm16}
Let $p\in\bbb{P}$, $k\in\N_+$ and $n_k\geq 2$ be such an integer that $p^k\mid\frac{D_{n_k}^{(e)}}{n_k-1}$. Let us define $q_p(n_k) = \frac{2}{p}\left(\frac{(-1)^{n_k+p}D_{n_k+p}^{(e)}}{n_k+p-1}-\frac{(-1)^{n_k}D_{n_k}^{(e)}}{n_k-1}\right) = \frac{2}{p}\left(\frac{(-1)^{n_k+p}D_{n_k+p}^{(o)}}{n_k+p-1}-\frac{(-1)^{n_k}D_{n_k}^{(o)}}{n_k-1}\right)$ (where the last equality holds because of equations (\ref{eq23,5})).

For $p>2$ we have the following implications.
\begin{itemize}
\item If $p\nmid q_p(n_k)$ then there exists a unique $n_{k+1}$ modulo $p^{k+1}$ such that $n_{k+1}\equiv n_k \pmod{p^k}$ and $p^{k+1}\mid\frac{D_n^{(e)}}{n-1}$ for all $n\geq 2$ congruent to $n_{k+1}$ modulo $p^{k+1}$. What is more, $n_{k+1} \equiv n_k-\frac{D_{n_k}^{(e)}}{(n_k-1)q_p(n_k)} \pmod{p^{k+1}}$.
\item If $p\mid q_p(n_k)$ and $p^{k+1}\mid\frac{D_{n_k}^{(e)}}{n_k-1}$ then $p^{k+1}\mid\frac{D_{n}^{(e)}}{n-1}$ for all $n$ satisfying $n\equiv n_k \pmod{p^k}$ and $n\geq 2$.
\item If $p\mid q_p(n_k)$ and $p^{k+1}\nmid\frac{D_{n_k}^{(e)}}{n_k-1}$ then $p^{k+1}\nmid\frac{D_{n}^{(e)}}{n-1}$ for any $n$ satisfying $n\equiv n_k \pmod{p^k}$ and $n\geq 2$.
\end{itemize}
In particular, if $k=1$, $p\mid \frac{D_{n_1}^{(e)}}{n_1-1}$ and $p\nmid\frac{2}{p}\left(\frac{D_{n_1+p}^{(e)}}{n_1+p-1}+\frac{D_{n_1}^{(e)}}{n_1-1}\right)$ then for any $l\in\mathbb{N}_+$ there exists a unique $n_l$ modulo $p^l$ such that $n_l\equiv n_1 \pmod{p}$ and $v_p\left(\frac{D_n^{(e)}}{n-1}\right)\geq{l}$ for all $n\geq 2$ congruent to $n_l$ modulo $p^l$. Moreover, $n_l$ satisfies the congruence $n_l \equiv n_{l-1}-\frac{D_{n_{l-1}}^{(e)}}{(n_{l-1}-1)q_p(n_1)} \pmod{p^l}$ for $l>1$.

If $p=2$ then for each $n_1\geq 2$ and $l\in\N_+$ there exists a unique $n_l$ modulo $2^l$ such that $n_l\equiv n_1 \pmod{2}$ and $v_2\left(\frac{D_n^{(e)}}{n-1}\right)\geq l-1$ for all $n\geq 2$ congruent to $n_l$ modulo $2^l$. Moreover, $n_l$ satisfies the congruence $n_l \equiv n_{l-1}-\frac{D_{n_{l-1}}^{(e)}}{(n_{l-1}-1)q_2(n_1)} \pmod{2^l}$ for $l>1$.

The whole statement above remains true, if we change $D_n^{(e)}$ by $D_n^{(o)}$.
\end{thm}

\begin{proof}
We apply Theorem \ref{thm1} to the sequences $\left(\frac{2D_n^{(e)}}{n-1}\right)_{n\in\N_2}$ and $\left(\frac{2D_n^{(o)}}{n-1}\right)_{n\in\N_2}$ and use the obvious fact that $v_2\left(\frac{D_n^{(e)}}{n-1}\right)=v_2\left(\frac{2D_n^{(e)}}{n-1}\right)-1$ and $v_p\left(\frac{D_n^{(e)}}{n-1}\right)=v_p\left(\frac{2D_n^{(e)}}{n-1}\right)$ for prime number $p>2$.

For $p=2$, $q_2(2)=1$ and $q_2(3)=-5$ are odd, so Theorem \ref{thm1} gives us precise description of $2$-adic valuation of numbers $\frac{2D_n^{(e)}}{n-1}$ and $\frac{2D_n^{(o)}}{n-1}$, $n\in\N_2$.
\end{proof}

\begin{cor}\label{cor7}
We have the following formulae for $p$-adic valuations of numbers $\frac{D_n^{(e)}}{n-1}$ and $\frac{D_n^{(o)}}{n-1}$:
\begin{itemize}
\item $v_2\left(\frac{D_n^{(e)}}{n-1}\right) = v_2(n-2)-1$ for even $n\geq 2$.
\item $v_2\left(\frac{D_n^{(o)}}{n-1}\right) = v_2(n-3)-1$ for odd $n\geq 2$.
\item If $p$ is an odd prime number and $p\nmid q_p(2)$ then $v_p\left(\frac{D_n^{(e)}}{n-1}\right) = v_p(n-2)$ for $n\equiv 2\pmod{p}$.
\item If $p$ is an odd prime number and $p\nmid q_p(3)$ then $v_p\left(\frac{D_n^{(o)}}{n-1}\right) = v_p(n-3)$ for $n\equiv 3\pmod{p}$.
\end{itemize}
\end{cor}

\begin{proof}
By Theorem \ref{thm16}, for $l\in\N_+$ the only $n_l$ modulo $p^l$, such that $n_l\equiv 2\pmod{p}$ and $v_p\left(\frac{D_n^{(e)}}{n-1}\right)\geq l$ (respectively $v_2\left(\frac{D_n^{(e)}}{n-1}\right)\geq l-1$) for any $n\geq 2$ congruent to $n_l$ modulo $p^l$, is equal to $2$. Hence, if $n\equiv 2\pmod{p}$ and $n\geq 2$ then $v_p\left(\frac{D_n^{(e)}}{n-1}\right)\geq l$ (respectively $v_2\left(\frac{D_n^{(e)}}{n-1}\right)\geq l-1$) if and only if $v_p(n-2)\geq l$.

In a similar way we can prove the equalities of $p$-adic valuations of numbers $\frac{D_n^{(o)}}{n-1}$.
\end{proof}

According to numerical computations, among all odd prime numbers $p<10^6$, if $p\mid\frac{D_{n_1}^{(e)}}{n_1-1}$ then $p\nmid q_p(n_1)$. Thus, for any $l\in\mathbb{N}_+$ there exists a unique $n_l$ modulo $p^l$ such that $n_l\equiv n_1 \pmod{p}$ and $v_p\left(\frac{D_n^{(e)}}{n-1}\right)\geq{l}$  for all $n\geq 2$ congruent to $n_l$ modulo $p^l$. Summing up, if $p<10^6$ then we have the description of $p$-adic valuation of numbers $\frac{D_n^{(e)}}{n-1}$, $n\geq 2$. It is natural to ask the following question.

\begin{que}
Is there any prime number $p$ and a positive integer $n_1\geq 2$ such that $p\mid\frac{D_{n_1}^{(e)}}{n_1-1}$ and  $p\mid q_p(n_1)$?
\end{que}

In case of the sequence of odd derangements $(D_n^{(o)})_{n\in\N}$, the trial of description of $p$-adic valuations of numbers $D_n^{(o)}$, $n\in\N$, comes down to description of $p$-adic valuations of numbers $D_n$, $n\in\N$. In fact, the number $D_n^{(o)}$ of odd derangements of $n$-element set is equal to the number of all permutations of $n$-element set with exactly two fixed points, when $n\geq 2$ (see \cite{OEIS}). The number of all permutations of $n$-element set with exactly two fixed points equals $\frac{n(n-1)}{2}D_{n-2}$ because we choose two fixed points in $n\choose 2$ ways and we treat a permutation as an derangement of a set with $n-2$ elements.

\begin{prop}\label{prop18,5}
$D_n^{(o)} = \frac{n(n-1)}{2}D_{n-2}$ for $n\geq 2$.
\end{prop}

\begin{proof}
We know from Section \ref{subsec3.3} that
\begin{equation}\label{eq24}
D_n=n(n-1)D_{n-2}+(-1)^nf_2(n)
\end{equation}
for each $n\in\N_2$. Recalling that $f_2=1-X$ the equality (\ref{eq24}) takes the form
\begin{equation*}
D_n=n(n-1)D_{n-2}-(-1)^n(n-1), \mbox{ } n\in\N_2.
\end{equation*}
Hence
\begin{equation*}
D_n^{(o)}=\frac{D_n+(-1)^n(n-1)}{2}=\frac{n(n-1)}{2}D_{n-2}
\end{equation*}
for every $n\in\N_2$, which completes the proof.
\end{proof}

Hence, $v_p(D_n^{(o)})=v_p\left(\frac{n(n-1)}{2}\right)+v_p(D_{n-2})$ for each $n\geq 2$ and $p\in\bbb P$. What is more, since $E_n=\frac{D_n}{n-1}$ and $E_n^{(o)}=\frac{D_n^{(o)}}{n-1}$ for $n\geq 2$, thus we have $v_p(E_n^{(o)})=v_p\left(\frac{n(n-1)}{2}\right)+v_p(E_{n-2})$ for each $n\geq 4$ and $p\in\bbb P$.

\section{Some diophantine equations with numbers of derangements}\label{sec4,5}

\subsection{Diophantine equations involving factorials}\label{subsec4,5.1}

Diophantine equations involving terms of given sequences are a subject of interest of number theorists. There are many papers concerning problems of the following type: when a term of a given sequence is a factorial or sum, difference or product of factorials (see \cite{BerGal}, \cite{Dab}, \cite{ErdObl}, \cite{Gaw}, \cite{GroLuc}, \cite{Luc2}, \cite{Luc3}, \cite{LucSik}, \cite{MarLeng}, \cite{Ul}).

Using information about $2$-adic valuations of numbers of derangements we will prove that $D_0=D_2=1$ and $D_3=2$ are the unique numbers of derangements being factorials.

\begin{prop}
All the solutions $(n,m)$ of the diophantine equation $D_n=m!$, $n,m\in\N_+$, are $(2,1)$ and $(3,2)$.
\end{prop}

\begin{proof}
If $D_n=m!$ then $v_2(D_n)=v_2(m!)$. We know that $v_2(D_n)=v_2(n-1)$ for any nonnegative integer $n$. Therefore $v_2(n-1)=v_2(m!)$, which means that $n-1\geq 2^{v_2(n-1)}=2^{v_2(m!)}$. If $n>1$ then $D_n\geq D_{1+2^{v_2(m!)}}$. If we denote the number $1+2^{v_2(m!)}$ by $M(m)$ and use the formula $D_{M(m)}=\left\lfloor\frac{M(m)!}{e}+\frac{1}{2}\right\rfloor$ then we obtain $m!=D_n\geq D_{M(m)}\geq \frac{M(m)!}{e}-\frac{1}{2}$.

However, we will prove by induction on $m$ that $m!<\frac{M(m)!}{e}-\frac{1}{2}$ for $m\geq 4$. Indeed, $4!<\frac{M(4)!}{e}-\frac{1}{2}=\frac{9!}{e}-\frac{1}{2}$ and $5!<\frac{M(5)!}{e}-\frac{1}{2}=\frac{9!}{e}-\frac{1}{2}$. Assume now that $m!<\frac{M(m)!}{e}-\frac{1}{2}$ for some integer $m\geq 4$. Then
\begin{equation*}
\begin{split}
(m+2)! & =m!\cdot (m+1)(m+2)<(m+1)(m+2)\left(\frac{M(m)!}{e}-\frac{1}{2}\right)< \\
& <(m+1)(m+2)\frac{M(m)!}{e}-\frac{1}{2}.
\end{split}
\end{equation*}
It suffices to show that $M(m)!\cdot (m+1)(m+2)\leq M(m+2)!$. Since $m!<\frac{M(m)!}{e}-\frac{1}{2}$, thus $m<M(m)$. We know that $v_2((m+2)!)>v_2(m!)$, so $M(m+2)-M(m)\geq 2$. Hence
\begin{equation*}
M(m)!\cdot (m+1)(m+2)<M(m)!\cdot (M(m)+1)(M(m)+2)\leq M(m+2)
\end{equation*}
and we are done. We proved that $m!<\frac{M(m)!}{e}-\frac{1}{2}$ for $m\geq 4$.

Summing up, if $D_n=m!$ then $m\leq 3$. Finally, we check one by one for each $m\in\{0,1,2,3\}$ that $D_0=D_2=1$ and $D_3=2$ are the only factorials in the sequence of derangements.
\end{proof}

We can generalize the result above as follows.

\begin{prop}
For any positive rational number $q$ the diophantine equation $D_n=q\cdot m!$ has only finitely many solutions $(n,m)\in\N_+^2$ and these solutions satisfy the inequality $q\cdot m!>\frac{(1+2^{v_2(q)+v_2(m!)})!}{e}-\frac{1}{2}$.
\end{prop}

\begin{proof}
If $D_n=q\cdot m!$ then $v_2(n-1)=v_2(D_n)=v_2(q)+v_2(m!)$ and $n-1\geq 2^{v_2(n-1)}=2^{v_2(q)+v_2(m!)}$. The product $q\cdot m!$ is an integer, thus $v_2(q)+v_2(m!)\geq 0$. If $n>1$ then $D_n\geq D_{1+2^{v_2(q)+v_2(m!)}}$. Let us put $M(m)=1+2^{v_2(q)+v_2(m!)}$. We use the formula $D_{M(m)}=\left\lfloor\frac{M(m)!}{e}+\frac{1}{2}\right\rfloor$ and get
\begin{equation*}
q\cdot m!=D_n\geq D_{M(m)}>\frac{M(m)!}{e}-\frac{1}{2}.
\end{equation*}
However, if we apply Legendre's formula $v_2(m!)=m-s_2(m)$ (where $s_2(m)$ is the sum of binary digits of a number $m$) then we get the following limit:
\begin{equation*}
\begin{split}
\lim_{m\rightarrow +\infty}\frac{M(m)}{m}=\lim_{m\rightarrow +\infty}\frac{1+2^{v_2(q)+v_2(m!)}}{m}\geq \lim_{m\rightarrow +\infty}\frac{2^{v_2(q)+m-s_2(m)}}{m}\geq \\
\geq\lim_{m\rightarrow +\infty}\frac{2^{v_2(q)+m-\log_2 m-1}}{m}=\lim_{m\rightarrow +\infty}\frac{2^{v_2(q)+m-1}}{m^2}=+\infty
\end{split}
\end{equation*}
Therefore we conclude that $\lim_{m\rightarrow +\infty}\frac{\frac{M(m)!}{e}-\frac{1}{2}}{q\cdot m!}=+\infty$, which implies that $q\cdot m!<\frac{M(m)!}{e}-\frac{1}{2}$ for $m\gg 0$.

Hence there exists a positive integer $m_0$ such that if $D_n=q\cdot m!$ then $m<m_0$.
\end{proof}

It is worth to notice that the set of positive rational values $q$ such that the equation $D_n=q\cdot m!$ has a solution $(n,m)\in\N_+^2$ is a discrete subset of the real halfline $[0,+\infty)$ (with respect to Euclidean topology) with exactly one accumulation point $0$. It is obvious that this set is exactly the set $S:=\{\frac{D_n}{m!}: n\in\N_2, m\in\N_+\}$. Hence it suffices to prove that for each $k\in\N_+$ the set $S_k:=S\cap\left[\frac{1}{ke},+\infty\right)$ is discrete without accumulation points. First of all we show that the set $S_1$ is discrete without accumulation points. If $\frac{1}{e}\leq\frac{D_n}{m!}$ then using the fact that $D_n\in\left(\frac{n!}{e}-\frac{1}{n},\frac{n!}{e}+\frac{1}{n}\right)$ we conclude the inequality
\begin{equation*}
\frac{1}{e}\leq\frac{n!}{e\cdot m!}+\frac{1}{n\cdot m!}.
\end{equation*}
The above inequality is equivalent to the following one:
\begin{equation*}
m!\leq n!+\frac{e}{n},
\end{equation*}
which is satisfied, if $m\leq n$. Hence the set $S_1$ consists of values which are ''close'' to the multiplicities of $\frac{1}{e}$ ($\left|\frac{D_n}{m!}-\frac{n!}{e\cdot m!}\right|<\frac{1}{n\cdot m!}$) and as a result $S_1$ is discrete without accumulation points. Now it remains us to prove that for each $k\in\N_+$ there are only finitely many tuples $(n,m)$ such that $\frac{D_n}{m!}\in S_k$ and $m>n$. Such tuples satisfy the inequality
\begin{equation*}
\frac{1}{ke}\leq\frac{n!}{e\cdot m!}+\frac{1}{n\cdot m!},
\end{equation*}
which is equivalent to
\begin{equation*}
m\cdot ...\cdot (n+1)\leq k+\frac{ke}{n\cdot n!}.
\end{equation*}
Since $n\geq 2$, we have $k+\frac{ke}{n\cdot n!}<2k$ and the inequality $m\cdot ...\cdot (n+1)<2k$ has only finitely many solutions $(n,m)\in\N_2\times\N_+$ with $n<m$. Hence for any $k\in\N_+$ the set $S_k$ is discrete and has no accumulation points. Finally $S$ is discrete. Any positive real number $x$ is not an accumulation point of the set $S$ since $x\in S_k$ for some $k\in\N_+$ and $S_k$ has no accumulation points. The number $0$ is an accumulation point of $S$ because $\lim_{m\rightarrow +\infty}\frac{D_n}{m!}=0$ for arbitrary $n\in\N_2$.

On the other hand, there are infinitely many positive rational values $q$ such that the diophantine equation $D_n=q\cdot m!$ has at least two solutions of the form $(n_0,m_0)$, $(n_1,m_0+1)$. Namely, we set $q=\frac{D_{n_0}}{m_0!}$. Then $\frac{D_{n_1}}{(m_0+1)!}=\frac{D_{n_0}}{m_0!}$ if and only if $m_0+1=\frac{D_{n_1}}{D_{n_0}}$. Hence, it suffices to set $n_0,n_1\in\N_2$ such that $D_{n_0}\mid D_{n_1}$ (this condition holds for example when $D_{n_0}\mid n_1-1$, but not only in this case, which shows the example $n_0=5$ and $n_1=49$, where $D_5=44\mid D_49$ and $44\nmid 48$) and $m_0=\frac{D_{n_1}}{D_{n_0}}-1$.

However, we do not know if there exists $q\in\Q$, $q>0$ such that the equation $D_n=q\cdot m!$ has at least three solutions.

\begin{que}
Is there any $q\in\Q$, $q>0$ such that the equation $D_n=q\cdot m!$ has at least three solutions?
\end{que}

In the light of Proposition \ref{prop18,5} we are able to establish analogue results on diophantine equations involving numbers of odd derangements and factorials.

\begin{prop}
All the solutions $(n,m)$ of the diophantine equation $D_n^{(o)}=m!$, $n,m\in\N_+$, are $(2,1)$ and $(4,3)$.
\end{prop}

\begin{proof}
If $n\geq 2$ and $D_n^{(o)}=m!$ then $v_2\left(D_n^{(o)}\right)=v_2(m!)$. By Proposition \ref{prop18,5} $D_n^{(o)}=\frac{n(n-1)}{2}D_{n-2}$. Since $v_2(D_{n-2})=v_2(n-3)$ for any integer $n\geq 2$, thus $v_2\left(D_n^{(o)}\right)=v_2(n(n-1)(n-3))$. As a result $v_2(n(n-1)(n-3))=v_2(m!)$, which means that $n^3>n(n-1)(n-3)\geq 2^{v_2(n(n-1)(n-3))}=2^{v_2(m!)}$. Hence $n>2^{\frac{v_2(m!)}{3}}\geq 2^{\left\lfloor\frac{v_2(m!)}{3}\right\rfloor}$. If $n\geq 4$ then $D_n^{(o)}\geq D_{M(m)}^{(o)}$, where $M(m)=2^{\left\lfloor\frac{v_2(m!)}{3}\right\rfloor}$. By Corollary \ref{cor6}, for $t\geq 4$ we obtain the inequality
\begin{equation*}
D_{t}^{(o)}=\frac{D_t+(-1)^t(t-1)}{2}>\frac{t!}{2e}-\frac{t-1}{2}>\frac{t!}{2e}-\frac{t!}{4e}=\frac{t!}{4e}.
\end{equation*}
Then we obtain $m!=D_n^{(o)}\geq D_{M(m)}^{(o)}>\frac{M(m)!}{4e}$ provided that $M(m)\geq 4$.

However, we will prove by induction on $m$ that
\begin{equation}\label{eq27}
m!<\frac{M(m)!}{4e}
\end{equation}
for $m\geq 16$. Indeed, the inequality (\ref{eq27}) holds for $m\in\{16,17,18,19\}$. We will show that if (\ref{eq27}) is valid for $m$ then (\ref{eq27}) is true for $m+4$. In order to do this we note that
\begin{equation*}
(m+4)!=m!\cdot (m+1)(m+2)(m+3)(m+4)<(m+1)(m+2)(m+3)(m+4)\frac{M(m)!}{4e}
\end{equation*}
It suffices to show that $M(m)!\cdot (m+1)(m+2)\leq M(m+4)!$. Since $m!<\frac{M(m)!}{4e}$, thus $m<M(m)$. Because $v_2((m+4)!)-v_2(m!)\geq 3$, so $\left\lfloor\frac{v_2((m+4)!)}{3}\right\rfloor-\left\lfloor\frac{v_2((m)!)}{3}\right\rfloor\geq 1$. Since $m\geq 16$, thus $v_2(m!)\geq 15$ and $M(m)\geq 32$. Hence $M(m+4)-M(m)\geq 32>4$ and
\begin{equation*}
\begin{split}
& M(m)!\cdot (m+1)(m+2)(m+3)(m+4)< \\
& <M(m)!\cdot (M(m)+1)(M(m)+2)(M(m)+3)(M(m)+4)\leq M(m+4)!.
\end{split}
\end{equation*}
We proved that $m!<\frac{M(m)!}{e}$ for $m\geq 16$.

Summing up, if $D_n^{(o)}=m!$ then $m\leq 15$. Finally, we check one by one for each $m\in\{0,1,2,...,15\}$ that $D_2=1$ and $D_4=6$ are the only factorials in the sequence of odd derangements.
\end{proof}

\begin{prop}
For any positive rational number $q$ the diophantine equation $D_n^{(o)}=q\cdot m!$ has only finitely many solutions $(n,m)\in\N_+^2$ and these solutions satisfy the inequality $q\cdot m!>\frac{\left(2^{\left\lfloor\frac{v_2(q)+v_2(m!)}{3}\right\rfloor}\right)!}{4e}$.
\end{prop}

\begin{proof}
If $D_n^{(o)}=q\cdot m!$ and $n\geq 2$ then $v_2(n(n-1)(n-3))=v_2\left(D_n^{(o)}\right)=v_2(q)+v_2(m!)$ and $n^3>n(n-1)(n-3)\geq 2^{v_2(n(n-1)(n-3))}=2^{v_2(q)+v_2(m!)}$. Hence $n>2^{\frac{v_2(q)+v_2(m!)}{3}}\geq 2^{\left\lfloor\frac{v_2(q)+v_2(m!)}{3}\right\rfloor}$. The product $q\cdot m!$ is an integer, thus $v_2(q)+v_2(m!)\geq 0$. If $n\geq 4$ then $D_n^{(o)}\geq D_{M(m)}^{(o)}$, where $M(m)=2^{\left\lfloor\frac{v_2(q)+v_2(m!)}{3}\right\rfloor}$. We use the inequality $D_{M(m)}>\frac{M(m)!}{4e}$, valid for $M(m)\geq 4$, to obtain
\begin{equation*}
q\cdot m!=D_n^{(o)}\geq D_{M(m)}^{(o)}>\frac{M(m)!}{4e}.
\end{equation*}
However, if we apply Legendre's formula $v_2(m!)=m-s_2(m)$ then we get the following limit:
\begin{equation*}
\begin{split}
\lim_{m\rightarrow +\infty}\frac{M(m)}{m}=\lim_{m\rightarrow +\infty}\frac{2^{\left\lfloor\frac{v_2(q)+m-s_2(m)}{3}\right\rfloor}}{m}\geq \lim_{m\rightarrow +\infty}\frac{2^{\frac{v_2(q)+m-s_2(m)-3}{3}}}{m}\geq \\
\geq\lim_{m\rightarrow +\infty}\frac{2^{\frac{v_2(q)+m-\log_2 m-4}{3}}}{m}=\lim_{m\rightarrow +\infty}\frac{2^{\frac{v_2(q)+m-4}{3}}}{m^{\frac{4}{3}}}=+\infty
\end{split}
\end{equation*}
Therefore we conclude that $\lim_{m\rightarrow +\infty}\frac{M(m)!}{4eq\cdot m!}=+\infty$, which implies that $q\cdot m!<\frac{M(m)!}{4e}$ for $m\gg 0$.

Hence there exists a positive integer $m_0$ such that if $D_n^{(o)}=q\cdot m!$ then $m<m_0$.
\end{proof}

In case of numbers $D_n^{(e)}$ we will use knowledge about their $3$-adic valuation to establish results on diophantine equations involving these numbers and factorials.

\begin{prop}
All the solutions $(n,m)$ of the diophantine equation $D_n^{(e)}=m!$, $n,m\in\N$, are $(0,0)$, $(0,1)$, $(3,2)$ and $(5,4)$.
\end{prop}

\begin{proof}
By Corollary \ref{cor7} and the fact that $3\nmid\frac{D_n^{(e)}}{n-1}$ we know that $v_3\left(\frac{D_n^{(e)}}{n-1}\right)=v_3(n-2)$ for every $n\in\N_2$. Hence $v_3(D_n^{(e)}=v_3((n-1)(n-2))=\max\{v_3(n-1), v_3(n-2)\}$ for each $n\in\N$. Thus if $D_n^{(e)}=m!$ then $\max\{v_3(n-1), v_3(n-2)\}=v_3(m!)$, which means that $n-1\geq 3^{\max\{v_3(n-1), v_3(n-2)\}}=3^{v_3(m!)}$. Hence $n\geq 1+3^{v_3(m!)}$. If $n\geq 3$ then $D_n^{(e)}\geq D_{M(m)}^{(e)}$, where $M(m)=1+3^{v_3(m!)}$. By Corollary \ref{cor6}, for $t\geq 4$ we obtain the inequality
\begin{equation*}
D_{t}^{(o)}=\frac{D_t-(-1)^t(t-1)}{2}>\frac{t!}{2e}-\frac{t-1}{2}>\frac{t!}{2e}-\frac{t!}{4e}=\frac{t!}{4e}.
\end{equation*}
Then we obtain $m!=D_n^{(e)}\geq D_{M(m)}^{(e)}>\frac{M(m)!}{4e}$ provided that $M(m)\geq 4$. However, we will prove by induction on $m$ that
\begin{equation}\label{eq27,5}
m!<\frac{M(m)!}{4e}
\end{equation}
for $m\geq 6$. Indeed, the inequality (\ref{eq27,5}) holds for $m\in\{6,7,8\}$. We will show that if (\ref{eq27,5}) is valid for $m$ then (\ref{eq27,5}) is true for $m+3$. In order to do this we note that
\begin{equation*}
(m+3)!=m!\cdot (m+1)(m+2)(m+3)<(m+1)(m+2)(m+3)\frac{M(m)!}{4e}
\end{equation*}
It suffices to show that $M(m)!\cdot (m+1)(m+2)(m+3)\leq M(m+3)!$. Since $m!<\frac{M(m)!}{4e}$, thus $m<M(m)$. We have $v_3((m+3)!)-v_3(m!)\geq 1$. Since $m\geq 6$, thus $v_3(m!)\geq 2$ and $M(m)\geq 10>1$. Hence $M(m+3)-M(m)\geq 3$ and
\begin{equation*}
\begin{split}
& M(m)!\cdot (m+1)(m+2)(m+3)< \\
& <M(m)!\cdot (M(m)+1)(M(m)+2)(M(m)+3)\leq M(m+3)!.
\end{split}
\end{equation*}
We proved that $m!<\frac{M(m)!}{e}$ for $m\geq 6$.

Summing up, if $D_n^{(e)}=m!$ then $m\leq 5$. Finally, we check one by one for each $m\in\{0,1,2,3,4,5\}$ that $D_0=1$, $D_3=2$ and $D_5=24$ are the only factorials in the sequence of even derangements.
\end{proof}

\begin{prop}
For any positive rational number $q$ diophantine equation $D_n^{(e)}=q\cdot m!$ has only finitely many solutions $(n,m)\in\N^2$ and these solutions satisfy the inequality $q\cdot m!>\frac{\left(1+3^{v_3(q)+v_3(m!)}\right)!}{4e}$.
\end{prop}

\begin{proof}
If $D_n^{(e)}=q\cdot m!$ then $\max\{v_3(n-1), v_3(n-2)\}=v_3\left(D_n^{(e)}\right)=v_3(q)+v_3(m!)$ and $n-1\geq 3^{\max\{v_3(n-1), v_3(n-2)\}}=3^{v_3(q)+v_3(m!)}$. The product $q\cdot m!$ is an integer, thus $v_3(q)+v_3(m!)\geq 0$. If $n\geq 3$ then $D_n^{(e)}\geq D_{M(m)}^{(e)}$, where $M(m)=1+3^{v_3(q)+v_3(m!)}$. We use the inequality $D_{M(m)}>\frac{M(m)!}{4e}$, valid for $M(m)\geq 4$, to obtain
\begin{equation*}
q\cdot m!=D_n^{(e)}\geq D_{M(m)}^{(e)}>\frac{M(m)!}{4e}.
\end{equation*}
However, if we apply Legendre's formula $v_3(m!)=\frac{m-s_3(m)}{2}$ then we get the following limit:
\begin{equation*}
\begin{split}
& \lim_{m\rightarrow +\infty}\frac{M(m)}{m}=\lim_{m\rightarrow +\infty}\frac{1+3^{v_3(q)+v_3(m!)}}{m}\geq \lim_{m\rightarrow +\infty}\frac{3^{v_3(q)+v_3(m!)}}{m}\geq \\
& \geq\lim_{m\rightarrow +\infty}\frac{3^{v_3(q)+\frac{m-\log_3 m-1}{2}}}{m}=\lim_{m\rightarrow +\infty}\frac{3^{\frac{2v_3(q)+m-\log_3 m-1}{2}}}{m}=\lim_{m\rightarrow +\infty}\frac{3^{\frac{2v_3(q)+m-1}{2}}}{m^{\frac{3}{2}}}=+\infty
\end{split}
\end{equation*}
Therefore we conclude that $\lim_{m\rightarrow +\infty}\frac{M(m)!}{4eq\cdot m!}=+\infty$, which implies that $q\cdot m!<\frac{M(m)!}{4e}$ for $m\gg 0$. Hence there exists a positive integer $m_0$ such that if $D_n^{(e)}=q\cdot m!$ then $m<m_0$.
\end{proof}

Using analogous reasoning as in case of numbers $D_n$ one can prove that the set of all positive rational numbers $q$ such that the equation $D_n^{(o)}=q\cdot m!$ ($D_n^{(e)}=q\cdot m!$, respectively) has a solution $(n,m)\in\N_+^2$ is discrete subset of the real half-line $[0, +\infty)$ and it has exactly one accumulation point $0$. However, there are infinitely many values $q$ such that the equation $D_n^{(o)}=q\cdot m!$ ($D_n^{(e)}=q\cdot m!$, respectively) has at least two solutions of the form $(n_0,m_0)$ and $(n_1,m_0+1)$. Analogously as for numbers $D_n$, it suffices to put $n_0,n_1\in\N_2$ such that $D_{n_0}^{(o)}\mid D_{n_1}^{(o)}$ and $m_0=\frac{D_{n_1}^{(o)}}{D_{n_0}^{(o)}}-1$ ($D_{n_0}^{(e)}\mid D_{n_1}^{(e)}$ and $m_0=\frac{D_{n_1}^{(e)}}{D_{n_0}^{(e)}}-1$, respectively). As well as in case of classical numbers of derangements we can ask the following question.

\begin{que}
Is there any $q\in\Q$, $q>0$ such that the equation $D_n^{(o)}=q\cdot m!$ ($D_n^{(e)}=q\cdot m!$, respectively) has at least three solutions?
\end{que}

\subsection{When a number of derangements is a power of a prime number?}\label{subsec4,5.2}

Let us consider diophantine equation $D_n=p^k$, where $p$ is a given prime number and $n,k\in\N_+$ are unknowns.

\begin{prop}\label{prop29}
For any prime number $p$ the diophantine equation $D_n=p^k$, $n,k\in\N_+$, has only finitely many solutions $(n,k)$. More precisely, the number of solutions is at most equal to $v_p(\sum_{j=1}^{+\infty} j!)$ ($\sum_{j=1}^{+\infty} j!\in\Z_p\bs\{0\}$, so its $p$-adic valuation is well defined and finite).
\end{prop}

\begin{proof}
If $D_n=p^k$, $n,k\in\N_+$ then obviously $n\geq 2$. Since $D_n=(n-1)E_n$, thus $n-1=p^l$ for some $l\in\N$. Let us recall that for each integer $n\geq 2$ there holds the equality $(-1)^nE_n=f_{p,\infty}(n-2)-f_{p,\infty}(n-1)$, where the function $f_{p,\infty}:\Z_p\rightarrow\Z_p$ is given by the formula $f_{p,\infty}(x)=\sum_{j=0}^{+\infty} (-1)^j \prod_{i=0}^{j-1} (x-i)$ (see Section \ref{subsubsec3.1.3}, Remark \ref{rem2}). Then
\begin{equation}\label{eq28}
(-1)^{1+p^l}E_{1+p^l}=f_{p,\infty}(p^l-1)-f_{p,\infty}(p^l)\equiv f_{p,\infty}(-1)-f_{p,\infty}(0)=\sum_{j=1}^{+\infty} j!\pmod{p^l}.
\end{equation}
By Remark \ref{rem2} we know that $\sum_{j=1}^{+\infty} j!\in\Z_p\bs\Z$. In particular $\sum_{j=1}^{+\infty} j!\neq 0$ and $k_0:=v_p(\sum_{j=1}^{+\infty} j!)\in\N$. Hence, if $l>k_0$ then by (\ref{eq28}), $v_p(E_{1+p^l})=v_p(\sum_{j=1}^{+\infty} j!)=k_0$. However, if $p^l\geq 3$ then $E_{1+p^l}\geq p^l>p^{k_0}$, which means that $E_{1+p^l}$ is not a power of $p$. As a result $D_{1+p^l}$ is not a power of $p$ for any integer $l>k_0$.
\end{proof}

Numerical computations show that if $p$ is a prime number less than $10^6$ then $v_p(\sum_{j=1}^{+\infty} j!)>0$ only for $p\in\{3,11\}$. This means that the equation $D_n=p^k$, $n,k\in\N_+$, has no solutions for prime number $p$ less than $10^6$ and not equal to 3 or 11. For $p=3$ we have $v_3(\sum_{j=1}^{+\infty} j!)=v_3(\sum_{j=1}^{8} j!)=2$ and there is one solution $n=1+3^1=4$, $k=2$. If $p=11$ then $v_{11}(\sum_{j=1}^{+\infty} j!)=v_{11}(\sum_{j=1}^{21} j!)=1$ and there is no solution $(n,k)\in\N_+^2$.

It is easy to note that $p|\sum_{j=1}^{+\infty} j!$ if and only if $p|\sum_{j=1}^{p-1} j!$.

\begin{con}\label{con5}
$p\nmid\sum_{j=1}^{p-1} j!$ for any prime number $p>11$.
\end{con}

The conjecture above resembles Kurepa's conjecture that $p\nmid\sum_{j=0}^{p-1} j!$ for any prime number $p$ (see \cite[Section B44]{Guy} and \cite{Kur}). If Conjecture \ref{con5} is true then by Proposition \ref{prop29} for any prime number $p>11$ there are no solutions of the equation $D_n=p^k$, $n,k\in\N_+$. Therefore the diophantine equation $D_n=p^k$ with variables $p\in\bbb{P}$ and $n,k\in\N_+$ has only one solution $(p,n,k)=(3,4,2)$.

Despite the equation $D_n=p^k$ with variables $(p,n,k)$, problem of solving the equations $D_n^{(o)}=p^k$ and $D_n^{(e)}=p^k$ with unknowns $p\in\bbb P$ and $n,k\in\N_+$ is very easy.

\begin{prop}
The diophantine equation $D_n^{(o)}=p^k$ with unknowns $p\in\bbb{P}$ and $n,k\in\N_+$ has no solutions while the equation $D_n^{(e)}=p^k$ with unknowns $p\in\bbb{P}$ and $n,k\in\N_+$ has one solution $(p,n,k)=(2,3,1)$.
\end{prop}

\begin{proof}
By Proposition \ref{prop18,3}, if $n$ is even then $n-3\mid D_n^{(o)}$. Moreover $n-1\mid D_n^{(o)}$. Thus $n-3$ and $n-1$ are powers of $p$. This gives three posibilities:
\begin{itemize}
\item $n-1=1$, but then $D_n^{(o)}=D_2^{(o)}=1$;
\item $n-3=1$, but then $D_n^{(o)}=D_4^{(o)}=6$;
\item $n-1\neq 1, n-3\neq 1$; in this case $p\mid (n-1)-(n-3)=2$, which means that $p=2$ and there must be $n-3=2$ and $n-1=4$; but then $D_n^{(o)}=D_5^{(o)}=20$.
\end{itemize}
If $n$ is odd then by Proposition \ref{prop18,3}, $\frac{n-3}{2}\mid D_n^{(o)}$. Then $\frac{n-3}{2}$ and $n-1$ are powers of $p$ and we have three possibilities:
\begin{itemize}
\item $n-1=1$, but then $D_n^{(o)}=D_2^{(o)}=1$;
\item $\frac{n-3}{2}=1$, but then $D_n^{(o)}=D_5^{(o)}=20$;
\item $n-1\neq 1, \frac{n-3}{2}\neq 1$; in this case $p\mid (n-1)-(n-3)=2$, which means that $p=2$ and there must be $n-3=2$ and $n-1=4$; but then $D_n^{(o)}=D_5^{(o)}=20$.
\end{itemize}
We proved that the equation $D_n^{(o)}=p^k$ with variables $p\in\bbb{P}$, $n,k\in\N_+$ has no solutions.

Now we consider the equation $D_n^{(e)}=p^k$. By Proposition \ref{prop18,3}, if $n$ is odd then $n-2\mid D_n^{(e)}$. Moreover $n-1\mid D_n^{(e)}$. Thus $n-2$ and $n-1$ are powers of $p$, which means that $p=2$ and $n=3$. Indeed we get the solution $(p,n,k)=(2,3,1)$.

If $n$ is even then by Proposition \ref{prop18,3}, $\frac{n-2}{2}\mid D_n^{(e)}$. Then $\frac{n-2}{2}$ and $n-1$ are powers of $p$ and we have two possibilities:
\begin{itemize}
\item $n-1=1$, but then $D_n^{(e)}=D_2^{(e)}=0$;
\item $n-1\neq 1$; then $n-2\neq 1$ (since $n$ is even) and $p\mid (n-1)-(n-2)=1$, a contradiction.
\end{itemize}
\end{proof}

\section{Arithmetic properties of $h$-Schenker sums}\label{sec5}

In this section we will generalize results on arithmetic properties of Schenker sums $a_n = \sum_{j=0}^n \frac{n!}{j!}n^j, n\in\N$, onto some wider class of sequences. We define these sequences by the closed formula
\begin{equation*}
a_n = \sum_{j=0}^n \frac{n!}{j!}h(n)^j, n\in\N,
\end{equation*}
where $h\in\Z[X]$ is fixed. A sequence of such form will be called \emph{the sequence of $h$-Schenker sums}.

Certainly, if $h=X$ then the sequence $(a_n)_{n\in\N}$ is the sequence of Schenker sums. Hence the sequences defined above can be thought as a generalization of the sequence of Schenker sums. This is a motivation of name of these sequences.

Let us notice that for $h=-1$ the sequence $(a_n)_{n\in\N}$ is the sequence of derangements. Moreover, if $h = b\in\Z$ then $(a_n)_{n\in\N}$ is the generalization of the sequence of derangements in the sense of section \ref{sec3}, i.e. it is given by the recurrence relation $a_0 = 1, a_n = na_{n-1} + b^n, n>0$. In particular, for $h=0$ the sequence $(a_n)_{n\in\N} = (n!)_{n\in\N}$ is the sequence of factorials (recall that we set $0^0=1$).

In \cite{AmdCalMo} and \cite{Mi} we can find various results concerning Schenker sums ($p$-adic valuations and infinitude of the set of Schenker primes, i.e. prime numbers $p$ with such property that $p$ divides $a_n$ for some positive integer $n$ not divisible by $p$). Now we will study these properties for our more general $h$-Schenker sums.

\subsection{Divisibility by primes, periodicity modulo $p$ and $p$-adic valuations of $h$-Schenker sums}\label{subsec5.1}

In \cite{AmdCalMo}, Amdeberhan, Callan and Moll showed that for any prime number $p$ and any positive integer $n$ divisible by $p$, $p$-adic valuation of $n$-th Schenker sum is equal to $p$-adic valuation of $n!$ (and by Legendre's formula this value is equal to $\frac{n-s_p(n)}{p-1}$, where $s_p(n)$ denotes the sum of digits of positive integer $n$ in base $p$). Now we will prove a  generalization of the mentioned result.

\begin{thm}\label{thm14}
Let $h\in\Z[X]$ and consider the sequence $(a_n)_{n\in\N}$ of $h$-Schenker sums. If $p$ is a prime number and $n$ is such a nonnegative integer that $p\mid h(n)$ then $v_p(a_n) = v_p(n!) = \frac{n-s_p(n)}{p-1}$.
\end{thm}

\begin{proof}
It suffices to verify that for $j\in\{1,2,...,n\}$, the $j$-th summand in the sum $\sum_{j=0}^n \frac{n!}{j!}h(n)^j$ has $p$-adic valuation strictly greater than the $0$-th summand, equal to $n!$:
\begin{equation*}
\begin{split}
v_p\left(\frac{n!}{j!}h(n)^j\right) & = \frac{n-s_p(n)}{p-1} - \frac{j-s_p(j)}{p-1} + jv_p(h(n)) > \frac{n-s_p(n)}{p-1} - \frac{j}{p-1} + j \geq \\
& \geq \frac{n-s_p(n)}{p-1} = v_p(n!)
\end{split}
\end{equation*}
This means that $v_p(a_n) = v_p(n!) = \frac{n-s_p(n)}{p-1}$.
\end{proof}

Another result given in \cite{AmdCalMo} states that it suffices to check the divisibility of Schenker sums $a_n, n\in\N_+$, by a given prime number $p$ only for indices $n$ less than $p$ because if $n\equiv r\pmod{p}$ then $p\mid a_n\iff p\mid a_r$. We will give a generalization of this fact.

\begin{prop}\label{prop19}
Let $h\in\Z[X]$ and consider the sequence $(a_n)_{n\in\N}$ of $h$-Schenker sums. Let $p$ be a prime number and $k$ be a positive integer. Then if $n_1,n_2\in\N$ are such that $n_1\equiv n_2\pmod{p^k}$ and $p\nmid h(n_1)$ then $\frac{a_{n_1}}{h(n_1)^{n_1}}\equiv \frac{a_{n_2}}{h(n_2)^{n_2}}\pmod{p^k}$. In particular, $p^k\mid a_{n_1}\iff p^k\mid a_{n_2}$.
\end{prop}

\begin{proof}
It suffices to note that
\begin{equation}\label{eq22}
\begin{split}
a_n & = \sum_{j=0}^n \frac{n!}{j!}h(n)^j = \sum_{j=0}^n \frac{n!}{(n-j)!}h(n)^{n-j} = \sum_{j=0}^n h(n)^{n-j}\prod_{i=0}^{j-1} (n-i) \equiv \\
\\ & \equiv \sum_{j=0}^{kp-1} h(n)^{n-j}\prod_{i=0}^{j-1} (n-i) \pmod{p^k},
\end{split}
\end{equation}
since $p^k\mid\frac{n!}{(n-j)!}$ for $j\geq kp$ and $\prod_{i=0}^{j-1} (n-i) = 0$ in case $n<j<kp$. By equivalence (\ref{eq22})
\begin{equation*}
\frac{a_{n_1}}{h(n_1)^{n_1}}\equiv \sum_{j=0}^{kp-1} h(n_1)^{-j}\prod_{i=0}^{j-1} (n_1-i) \pmod{p^k},
\end{equation*}
since $p\nmid h(n_1)$. The congruence $n_1\equiv n_2\pmod{p^k}$ implies that $h(n_1)\equiv h(n_2)\pmod{p^k}$ and as a consequence
\begin{equation*}
\frac{a_{n_1}}{h(n_1)^{n_1}}\equiv \sum_{j=0}^{kp-1} h(n_1)^{-j}\prod_{i=0}^{j-1} (n_1-i)\equiv \sum_{j=0}^{kp-1} h(n_2)^{-j}\prod_{i=0}^{j-1} (n_2-i)\equiv \frac{a_{n_2}}{h(n_2)^{n_2}} \pmod{p^k}.
\end{equation*}
Our proposition is proved.
\end{proof}

\begin{cor}\label{cor5}
Let $h\in\Z[X]$ and consider the sequence $(a_n)_{n\in\N}$ of $h$-Schenker sums. Let $p$ be a prime number, $k$ be a positive integer and $n_0\in\N_+$ be the smallest positive integer such that $p^k\mid n_0!$. Then the sequence $(a_n\pmod{p^k})_{n\in\N_{n_0}}$ is periodic of period $p^k(p-1)$.
\end{cor}

\begin{proof}
If $n\geq n_0$ is such that $p\mid h(n)$ then by Theorem \ref{thm14}, $v_p(a_n)=v_p(n!)\geq k$. Assume now that $n_1,n_2\geq n_0$ are such that $n_1\equiv n_2\pmod{p^k(p-1)}$ and $p\nmid h(n_1)
\cdot h(n_2)$. Since $n_1\equiv n_2\pmod{p^k}$, thus by Proposition \ref{prop19} we have $\frac{a_{n_1}}{h(n_1)^{n_1}}\equiv \frac{a_{n_2}}{h(n_2)^{n_2}}\pmod{p^k}$ and $h(n_1)^{n_1}\equiv h(n_2)^{n_1}\pmod{p^k}$. In addition, $n_1\equiv n_2\pmod{p^{k-1}(p-1)}$, so by Euler's theorem  $h(n_2)^{n_1}\equiv h(n_2)^{n_2}\pmod{p^k}$. Summing up our reasoning we conclude that $a_{n_1}\equiv a_{n_2}\pmod{p^k}$ and corollary follows.
\end{proof}

\begin{rem}
It is possible that the sequence $(a_n\pmod{p^k})_{n\in\N_{n_0-1}}$ is not periodic, but only on condition that $p\mid h(n_0-1)$ and then $v_p(a_{n_0-1})=v_p((n_0-1)!)<k$ (let us observe that if $p\nmid h(n_0-1)$ then the consideration from the proof of Corollary \ref{cor5} allows us to claim that $a_{n_0-1}\equiv a_n\pmod{p^k}$ for any $n\equiv n_0-1\pmod{p^k(p-1)}$).

We claim that if $p\mid h(n_0-1)$ and the basic period of the sequence $(a_n\pmod{p^k})_{n\in\N_{n_0}}$ is divisible by $p$ then the sequence $(a_n\pmod{p^k})_{n\in\N_{n_0-1}}$ is not periodic. If we assume the contrary then the basic period of $(a_n\pmod{p^k})_{n\in\N_{n_0-1}}$ must be divisble by $p$, but $a_n\equiv 0\not\equiv a_{n_0-1}\pmod{p^k}$ for any $n>n_0-1$ such that $n\equiv n_0-1\pmod{p}$. Hence it suffices to give such a sequence $(a_n)_{n\in\N}$ that the sequence $(a_n\pmod{p^k})_{n\in\N_{n_0}}$ has the basic period divisible by $p$ and $p\mid h(n_0-1)$.

Let us consider $a_n = \sum_{j=0}^n \frac{n!}{j!}(n+1)^j, n\in\N$, $p=5$ and $k=2$. Then the basic period of the sequence $(a_n\pmod{5^2})_{n\in\N_{10}}$ is equal to $100 = 5^2\cdot 4$. Therefore, the sequence $(a_n\pmod{5^2})_{n\in\N_9}$ is not periodic. Moreover, this example shows that it is possible that $p^k(p-1)$ is the basic period of the sequence $(a_n\pmod{p^k})_{n\in\N_{n_0}}$.
\end{rem}

Using Remark \ref{rem1} we obtain

\begin{cor}
Let $h\in\Z[X]$ and consider the sequence $(a_n)_{n\in\N}$ of $h$-Schenker sums. Let $d=p_1^{k_1}\cdot ...\cdot p_s^{k_s}$ be a positive integer and $n_0\in\N_+$ be the smallest positive integer such that $p_i^{k_i}\mid n_0!$ for each $i\in\{1,2,...,s\}$. Then the sequence $(a_n\pmod{d})_{n\in\N_{n_0}}$ is periodic of period $\lcm\{p_i^{k_i}(p_i-1): i\in\{1,2,...,s\}\}$.
\end{cor}

\bigskip

Theorem \ref{thm14} allows us to describe $p$-adic valuation of the $h$-Schenker sum $a_n$, $n\in\N$, when $p$ is such a prime number that $p\mid a_n$ only if $p\mid h(n)$. Namely, in this situation we have
\begin{equation*}
v_p(a_n)=
\begin{cases}
v_p(n!)=\frac{n-s_p(n)}{p-1} & \mbox{ if } p\mid h(n) \\
0 & \mbox{ if } p\nmid h(n)
\end{cases}.
\end{equation*}

However, it is possible that $p\mid a_n$ and $p\nmid h(n)$ for some $n\in\N$.

\begin{defi}
Let $h\in\Z[X]$ and a sequence $(a_n)_{n\in\N}$ be the sequence of $h$-Schenker sums. Then a prime number $p$ will be called $h$-Schenker prime if $p\mid a_n$ and $p\nmid h(n)$ for some $n\in\N$.
\end{defi}

In order to verify, if a given prime number $p$ is a $h$-Schenker prime, it suffices to check divisibility of $a_n$ by $p$ for $n\in\{0,1,...,p-1\}$ under the condition $p\nmid h(n)$. By Proposition \ref{prop19}, if $n$ is a positive integer such that $p\nmid h(n)$ and $r$ is remainder of $n$ from division by $p$ then $p\mid a_n\iff p\mid a_r$.

If $p\mid a_{n_1}$ and $p\nmid h(n_1)$ for some $n_1\in\N$ then using Theorem \ref{thm1} we can obtain description of $p$-adic valuation of the number $a_n$, where $n\equiv n_1\pmod{p}$. The congruence $\frac{a_{n_1}}{h(n_1)^{n_1}}\equiv \frac{a_{n_2}}{h(n_2)^{n_2}}\pmod{p^k}$ from the statement of Proposition \ref{prop19} suggests us that the sequence $(a_n)_{n\in\N}$ has pseudo-polynomial decomposition modulo $p$ on the set $\{n\in\N: n\equiv n_1\pmod{p}\}$. Using similar computation as in (\ref{eq22}) we obtain
\begin{equation*}
\begin{split}
a_n & = \sum_{j=0}^n \frac{n!}{j!}h(n)^j = \sum_{j=0}^n \frac{n!}{(n-j)!}h(n)^{n-j} = \sum_{j=0}^n h(n)^{n-j}\prod_{i=0}^{j-1} (n-i) \equiv \\
& \equiv \sum_{j=0}^{p^k} h(n)^{n-j}\prod_{i=0}^{j-1} (n-i) = h(n)^{n-p^k-1}\sum_{j=0}^{p^k} h(n)^{p^k+1-j}\prod_{i=0}^{j-1} (n-i) \pmod{p^k}.
\end{split}
\end{equation*}
Let us put $f_{p,k}(X) = \sum_{j=0}^{p^k} h(X)^{p^k+1-j}\prod_{i=0}^{j-1} (X-i) \in\Z[X]$ and $g_{p,k}(n) = h(n)^{n-p^k-1}$ for $k\in\N_+$. If $n\equiv n_1\pmod{p}$ then $g_{p,k}(n)\in\Z\bs p\Z$. What is more, for any $k\geq 2$ we have
\footnotesize
\begin{equation*}
\begin{split}
f'_{p,k}(n) & = (p^k+1)h(n)^{p^k}h'(n) + \\
& + \sum_{j=1}^{p^k} [(p^k+1-j)h(n)^{p^k-j}h'(n)\prod_{i=0}^{j-1} (n-i) + h(n)^{p^k+1-j}\sum_{s=0}^{j-1}\prod_{i=0, i\neq s}^{j-1} (n-i)] \equiv \\
& \equiv h(n)h'(n) + \sum_{j=1}^{2p-1} [(1-j)h(n)^{1-j}h'(n)\prod_{i=0}^{j-1} (n-i) + h(n)^{2-j}\sum_{s=0}^{j-1}\prod_{i=0, i\neq s}^{j-1} (n-i)]\pmod{p},
\end{split}
\end{equation*}
\normalsize
since $p\nmid h(n)$ for $n\equiv n_1\pmod{p}$ and by Fermat's little theorem $h(n)^{p^k}\equiv h(n)\pmod{p}$. Finally, $(f_{p,k}, g_{p,k})_{k\in\N_2}$ is a pseudo-polynomial decomposition modulo $p$ of the sequence $(a_n)_{n\in\N}$ on the set $\{n\in\N: n\equiv n_1\pmod{p}\}$. Hence, we can apply Theorem \ref{thm1} to describe behavior of $p$-adic valuation of the $h$-Schenker sum $a_n$, where $n\equiv n_1\pmod{p}$.

\begin{thm}
Let $h\in\Z[X]$ and a sequence $(a_n)_{n\in\N}$ be the sequence of $h$-Schenker sums. Let $p$ be a $h$-Schenker prime, $k\in\N_+$ and $n_k\in\N$ be such that $p^k\mid a_{n_k}$ and $p\nmid h(n_k)$. Let us define $\widehat{q}_p(n_k)=\frac{1}{p}(a_{n_k+p}-h(n)^pa_{n_k})$.
\begin{itemize}
\item If $v_p(\widehat{q}_p(n_k))=0$ then there exists a unique $n_{k+1}$ modulo  for which $n_{k+1}\equiv n_k \pmod{p^k}$ and $p^{k+1}\mid{a_n}$ for all $n$ congruent to $n_{k+1}$ modulo $p^{k+1}$. What is more, $n_{k+1} \equiv n_k-\frac{a_{n_kh(n_k)}}{\widehat{q}_p(n_k)} \pmod{p^{k+1}}$.
\item If $v_p(\widehat{q}_p(n_k))>0$ and $p^{k+1}\mid{a_{n_k}}$ then $p^{k+1}\mid{a_n}$ for all $n$ satisfying $n\equiv n_k \pmod{p^k}$.
\item If $v_p(\widehat{q}_p(n_k))>0$ and $p^{k+1}\nmid{a_{n_k}}$ then $p^{k+1}\nmid{a_n}$ for any $n$ satisfying $n\equiv n_k \pmod{p^k}$.
\end{itemize}
In particular, if $k=1$, $p\mid{a_{n_1}}$ and $v_p(\widehat{q}_p(n_1))=0$ then for any $l\in\mathbb{N}_+$ there exists a unique $n_l$ modulo $p^l$ such that $n_l\equiv n_1 \pmod{p}$ and $v_p(a_n)\geq{l}$ for all $n\equiv n_l\pmod{p^l}$. Moreover, the number $n_l$ satisfies the congruence $n_l \equiv n_{l-1}-\frac{a_{n_{l-1}}h(n_{l-1})}{\widehat{q}_p(n_1)} \pmod{p^l}$ for $l>1$.
\end{thm}

\begin{proof}
The number $q_p(n_k)$ as specified in Theorem \ref{thm1} takes the form
\begin{equation*}
q_p(n_k)=\frac{1}{p}\left(\frac{a_{n_k+p}}{h(n_k+p)^{n_k+p-p^k-1}}-\frac{a_{n_k}}{h(n_k)^{n_k-p^k-1}}\right).
\end{equation*}
Hence $q_p(n_k)\equiv\frac{1}{ph(n_k)^{n_k+p-p^k-1}}(a_{n_k+p}-h(n_k)^pa_{n_k})=\frac{\widehat{q}_p(n_k)}{h(n_k)^{n_k+p-p^k-1}}\pmod{p}$ and since $p\nmid h(n_k)$we have $p\mid q_p(n_k)\iff p\mid \widehat{q}_p(n_k)$. Moreover $h(n_k)^{n_k-p^k-1}q_p(n_k)=\frac{\widehat{q}_p(n_k)}{h(n_k)^p}\equiv \frac{\widehat{q}_p(n_k)}{h(n_k)}\pmod{p}$, where the last equality holds by Fermat's little theorem. Then we can use Theorem \ref{thm1} to obtain the statement of our theorem.
\end{proof}

\subsection{Bounds on $h$-Schenker sums and infinitude of the set of $h$-Schenker primes}\label{subsec5.2}

Firstly, we will prove that for any polynomial $h$ the sequence of absolute values of $h$-Schenker sums diverges to $+\infty$. More precisely, we have the following:

\begin{thm}\label{thm15}
Let $h\in\Z[X]$ and consider the sequence $(a_n)_{n\in\N}$ of $h$-Schenker sums.
\begin{enumerate}
\item If $h(n)>n$ for $n\gg 0$ then
\begin{equation*}
(n+1)!<h(n)^n<a_n<(n+1)h(n)^n
\end{equation*}
for $n\gg 0$.
\item If $h=X-b$ for some $b\in\N$ and $n\geq b+2$ then
\begin{equation*}
n!<(n-b)^{n-b}\prod_{i=0}^{b-1} (n-i)<a_n<(n+1)(n-b)^{n-b}\prod_{i=0}^{b-1} (n-i).
\end{equation*}
\item If $h=b$ for some $b\in\Z$ then
\begin{equation*}
\lim_{n\rightarrow +\infty}\frac{a_n}{n!} = e^b.
\end{equation*}
In particular, $a_n=O(n!)$, when $n\rightarrow +\infty$.
\item If $h=-X+b$ for some $b\in\N$ then
\begin{equation*}
n!<(n-b)^{n-b}\prod_{i=0}^{b-3} (n-i)<|a_n|<(n-1)(n-b)^{n-b+1}\prod_{i=0}^{b-2} (n-i)
\end{equation*}
for $n\gg 0$.
\item If $-h(n)>n$ for $n\gg 0$ then
\begin{equation*}
n!<|h(n)|^{n-1}(|h(n)|-n)<|a_n|<(n+1)|h(n)|^n
\end{equation*}
for $n\gg 0$.
\end{enumerate}
In particular, if the leading coefficient of $h$ is positive or $\deg h > 0$ then $|a_n|>n!$ for $n\gg 0$.
\end{thm}

\begin{proof}
\begin{enumerate}
\item If $h(n)>n$ for $n\gg 0$ then the $n$-th summand in the sum $a_n = \sum_{j=0}^n \frac{n!}{j!}h(n)^j$ is the biggest one, thus $a_n < (n+1)h(n)^n$. On the other hand, each summand in the mentioned sum is positive, therefore $a_n > h(n)^n$. Moreover, $h(n)\geq n+1$ for $n\gg 0$, hence $h(n)^n>(n+1)\cdot ...\cdot 2=(n+1)!$.
\item If $h=X-b$ for some $b\in\N$ and $n>b$ then $n-b-1$-st and $n-b$-th summands in the sum $a_n = \sum_{j=0}^n \frac{n!}{j!}h(n)^j$ are the biggest ones and equal to $(n-b)^{n-b}\prod_{i=0}^{b-1} (n-i)$. That is why $a_n\leq (n+1)(n-b)^{n-b}\prod_{i=0}^{b-1} (n-i)$ and the inequality is strict for $n\geq b+2$. On the other hand, each summand in the mentioned sum is positive, therefore $a_n > (n-b)^{n-b}\prod_{i=0}^{b-1} (n-i) > n!$.
\item If $h=b$ for some $b\in\Z$ then $a_n = \sum_{j=0}^n \frac{n!}{j!}b^j = n!\sum_{j=0}^n \frac{b^j}{j!}$, which means that $\lim_{n\rightarrow +\infty}\frac{a_n}{n!} = \lim_{n\rightarrow +\infty}\sum_{j=0}^n \frac{b^j}{j!} = e^b$ and $a_n=O(n!)$, when $n\rightarrow +\infty$.
\item If $h=-X+b$ for some $b\in\N$ then for $1\leq j\leq\left\lfloor\frac{b}{2}\right\rfloor$ we add up $n-b+2j$-th, $n-b+2j-1$-st, $n-b-2j$-th and $n-b-2j-1$-st summands together.
\begin{equation*}
\begin{split}
& (-n+b)^{n-b+2j}\prod_{i=0}^{b-2j-1} (n-i) + (-n+b)^{n-b+2j-1}\prod_{i=0}^{b-2j} (n-i) + \\
& \quad + (-n+b)^{n-b-2j}\prod_{i=0}^{b+2j-1} (n-i) + (-n+b)^{n-b-2j-1}\prod_{i=0}^{b+2j} (n-i) \\
& = (-1)^{n-b}\left[(n-b)^{n-b+2j}\prod_{i=0}^{b-2j-1} (n-i) - (n-b)^{n-b+2j-1}\prod_{i=0}^{b-2j} (n-i) +\right. \\
& \quad\left. + (n-b)^{n-b-2j}\prod_{i=0}^{b+2j-1} (n-i) - (n-b)^{n-b-2j-1}\prod_{i=0}^{b+2j} (n-i)\right] \\
& = (-1)^{n-b}\left[-2j(n-b)^{n-b+2j-1}\prod_{i=0}^{b-2j-1} (n-i) + 2j(n-b)^{n-b-2j-1}\prod_{i=0}^{b+2j-1} (n-i)\right] \\
& = (-1)^{n-b}\cdot 2j(n-b)^{n-b-2j-1}\prod_{i=0}^{b-2j-1} (n-i)\times \\
& \quad\times\left[(n-b+2j)\cdot ...\cdot (n-b-2j+1) - (n-b)^{4j}\right] \\
& = (-1)^{n-b}\cdot 2j(n-b)^{n-b-2j}\prod_{i=0}^{b-2j-1} (n-i)\times \\
& \quad\times\left[(n-b+2j)\prod_{i=1}^{2j-1} ((n-b)^2 - i^2) - (n-b)^{4j-1}\right] \\
& = (-1)^{n-b}\cdot 2j(n-b)^{n-b-2j}\prod_{i=0}^{b-2j-1} (n-i)\times \\
& \quad\times\left[(n-b)^{4j-1} + 2j(n-b)^{4j-2} + p_j(n-b) - (n-b)^{4j-1}\right] \\
& = (-1)^{n-b}\cdot 2j(n-b)^{n-b-2j}\prod_{i=0}^{b-2j-1} (n-i)\cdot [2j(n-b)^{4j-2} + p_j(n-b)],
\end{split}
\end{equation*}
where $p_j\in\Z[X]$ and $\deg p_j\leq 4j-3$. Hence $2j(n-b)^{4j-2} + p_j(n-b) > 0$ for $n\gg 0$. Since there are only finitely values $2j(n-b)^{4j-2} + p_j(n-b) > 0$, $1\leq j\leq\left\lfloor\frac{b}{2}\right\rfloor$, thus all these values are positive for $n\gg 0$.

For $\left\lfloor\frac{b}{2}\right\rfloor + 1\leq j\leq\left\lfloor\frac{n-b-1}{2}\right\rfloor$ we add up $n-b-2j$-th and $n-b-2j-1$-st summands together.
\begin{equation*}
\begin{split}
& (-n+b)^{n-b-2j}\prod_{i=0}^{b+2j-1} (n-i) + (-n+b)^{n-b-2j-1}\prod_{i=0}^{b+2j} (n-i) \\
& = (-1)^{n-b}[(n-b)^{n-b-2j}\prod_{i=0}^{b+2j-1} (n-i) - (n-b)^{n-b-2j-1}\prod_{i=0}^{b+2j} (n-i)] \\
& = (-1)^{n-b}\cdot 2j(n-b)^{n-b-2j-1}\prod_{i=0}^{b+2j-1} (n-i)
\end{split}
\end{equation*}

If $2\nmid b$ then the sign of the $n$-th summand in the sum $a_n = \sum_{j=0}^n \frac{n!}{j!}(-n+b)^j$ is equal to $(-1)^n$, this means it is opposite to $(-1)^{n+b}$. Therefore, if $n\gg 0$ then we obtain (in case, when $2\mid n-b$, we can omit the $0$-th summand in the sum $a_n = \sum_{j=0}^n \frac{n!}{j!}(-n+b)^j$ because its sign is $(-1)^{n-b}=1$):
\begin{equation*}
\begin{split}
& |a_n| \geq \sum_{j=1}^{\lfloor\frac{b}{2}\rfloor} 2j(n-b)^{n-b-2j}\prod_{i=0}^{b-2j-1} (n-i)\cdot [2j(n-b)^{4j-2} + p_j(n-b)] + \\
& \quad + \sum_{j=\lfloor\frac{b}{2}\rfloor + 1}^{\lfloor\frac{n-b-1}{2}\rfloor} 2j(n-b)^{n-b+2j-1}\prod_{i=0}^{b+2j-1} (n-i) - (n-b)^n \geq \\
& \geq 2(n-b)^{n-b-2}\prod_{i=0}^{b-3} (n-i)\cdot [2(n-b)^2 + p_1(n-b)] - (n-b)^n \geq \\
& \geq 2(n-b)^{n-b-2}\prod_{i=0}^{b-3} (n-i)\cdot (n-b)^2 - (n-b)^n \\
& = 2(n-b)^{n-b}\prod_{i=0}^{b-3} (n-i) - (n-b)^n > (n-b)^{n-b}\prod_{i=0}^{b-3} (n-i) \\
& = (n-b)\cdot (n-b)^{n-b-6}\cdot (n-b)^5\prod_{i=0}^{b-3} (n-i) \geq \\
& \geq 3!\cdot 4\cdot ...\cdot (n-b-3)(n-b-2)(n-b-1)\times \\
& \quad\times(n-b)(n-b+1)(n-b+2)\prod_{i=0}^{b-3} (n-i) = n!,
\end{split}
\end{equation*}
where the last inequality holds for $n-b\geq 3!=6$. Then $(n-b)^{n-b-6}\geq 4\cdot ...\cdot (n-b-3)$ and $(n-b)^5\geq (n-b-2)(n-b-1)(n-b)(n-b+1)(n-b+2)$.

If $2\mid b$ then the $n$-th summand in the sum $a_n = \sum_{j=0}^n \frac{n!}{j!}(-n+b)^j$ is added up together with $n-1$-st, $n-2b$-th and $n-2b-1$-st summand and analogous estimation as above allows us to state that
\begin{equation*}
\begin{split}
|a_n| & \geq 2(n-b)^{n-b-2}\prod_{i=0}^{b-3} (n-i)\cdot [2(n-b)^2 + p_1(n-b)]\geq \\
& \geq 2(n-b)^{n-b}\prod_{i=0}^{b-3} (n-i)\geq 2\cdot 3\cdot ...\cdot (n-b-3)\cdot (n-b)^5\prod_{i=0}^{b-3} (n-i)\geq n!
\end{split}
\end{equation*}
for $n\gg 0$.

To obtain the upper bound on $|a_n|$ it suffices to note that if $n>b$ then $n-b-1$-st and $n-b$-th summand in the sum $a_n = \sum_{j=0}^n \frac{n!}{j!}(-n+b)^j$ reduce and the $n-b+1$-st summand has the biggest absolute value among all the remaining summands. Therefore $|a_n|\leq (n-1)(n-b)^{n-b+1}\prod_{i=0}^{b-2} (n-i)$, where the inequality is strict for $n\geq b+2$.
\item If $-h(n)>n$ for $n\gg 0$ then for any $0\leq j\leq\left\lfloor\frac{n-1}{2}\right\rfloor$ the sum of $n-2j$-th and $n-2j-1$-st summand appearing in the sum $a_n = \sum_{j=0}^n \frac{n!}{j!}h(n)^j$ is equal to
\begin{equation*}
\begin{split}
& h(n)^{n-2j}\prod_{i=0}^{2j-1} (n-i) + h(n)^{n-2j-1}\prod_{i=0}^{2j} (n-i) \\
& = (-1)^n[|h(n)|^{n-2j}\prod_{i=0}^{2j-1} (n-i) - |h(n)|^{n-2j-1}\prod_{i=0}^{2j} (n-i)] \\
& = (-1)^n|h(n)|^{n-2j-1}(|h(n)|-n+2j)\prod_{i=0}^{2j-1} (n-i)
\end{split}
\end{equation*}
Hence
\begin{equation*}
|a_n| =
\begin{cases}
1 + \sum_{j=0}^{\frac{n-2}{2}} |h(n)|^{n-2j-1}(|h(n)|-n+2j)\prod_{i=0}^{2j-1} (n-i), & \mbox{ if } 2\mid n \\
\sum_{j=0}^{\frac{n-1}{2}} |h(n)|^{n-2j-1}(|h(n)|-n+2j)\prod_{i=0}^{2j-1} (n-i), & \mbox{ if } 2\nmid n
\end{cases},
\end{equation*}
which implies that $|a_n|>|h(n)|^{n-1}(|h(n)|-n)$ for $n\gg 0$.

In order to obtain the upper bound on $|a_n|$ it suffices to note that the $n$-th summand in the sum $a_n = \sum_{j=0}^n \frac{n!}{j!}h(n)^j$ has the biggest absolute value, thus $|a_n|<(n+1)|h(n)|^n$ for $n\gg 0$.
\end{enumerate}
\end{proof}

Theorem \ref{thm15} allows us to prove that for any nonzero polynomial $h\in\Z[X]$ there are infinitely many $h$-Schenker primes (certainly, if $h=0$ then any prime number $p$ is not an $h$-Schenker prime).

\begin{thm}
For any $h\in\Z[X]\bs\{0\}$ there are infinitely many $h$-Schenker primes.
\end{thm}

\begin{proof}
Let us assume that there are only finitely many $h$-Schenker primes. We consider two cases.
\begin{enumerate}
\item The case of $\deg h >0$. Let $n_0\in\N$ be such that $h(n_0)\cdot a_{n_0}\neq 0$ (we can find such an $n_0$, since $|a_n|>n!$ for $n\gg 0$). Let $p_1,...,p_s$ be all the $h$-Schenker primes that do not divide $h(n_0)$. Let $k_i = v_{p_i}(a_{n_0})$ for $i\in\{1,...,s\}$. Let us put $t=p_1^{k_1+1}\cdot ...\cdot p_s^{k_s+1}$. By Proposition \ref{prop19}, $v_{p_i}(a_{mt+n_0})=k_i$ for any $m\in\N$ and $i\in\{1,...,s\}$, so by Theorem \ref{thm14} and Theorem \ref{thm15} we obtain
\begin{equation*}
(mt+n_0)! < |a_{mt+n_0}| = \prod_{i=1}^s p_i^{k_i} \cdot \prod_{p \mbox{\scriptsize{ prime, }} p\mid h(mt+n_0)} p^{v_p((mt+n_0)!)} < (mt+n_0)!
\end{equation*}
for $m\gg 0$ - a contradiction.
\item The case $\deg h =0$. Then $h=b$ for some $b\in\Z\bs\{0\}$. If some prime number $p$ divides $a_n$ for some $n\in\N$ and $p$ is not an $h$-Schenker prime then $p\mid b$. Let $p_1,...,p_s$ be all the $h$-Schenker primes that do not divide $b$ and $t=p_1\cdot ...\cdot p_s$. Since $a_0=1$, thus $p_1,...,p_s$ do not divide $a_{mt}$ for any $m\in\N$. Hence $|a_{mt}|=\prod_{p \mbox{\scriptsize{ prime, }} p\mid b} p^{v_p((mt)!)}$ and
\begin{equation*}
\lim_{m\rightarrow +\infty}\frac{|a_{mt}|}{(mt)!} = \lim_{m\rightarrow +\infty}\frac{1}{\prod_{p \mbox{\scriptsize{ prime, }} p\nmid b} p^{v_p((mt)!)}} = 0,
\end{equation*}
but by Theorem \ref{thm15}, $\lim_{m\rightarrow +\infty}\frac{|a_{mt}|}{(mt)!} = e^b\neq 0$ - once again we obtain a contradiction.
\end{enumerate}
\end{proof}

\section*{Acknowledgements}

I wish to thank my MSc thesis advisor, Maciej Ulas for suggesting the sequence of derangements and trial of generalizing this sequence as a subject of MSc thesis, for scientific care and for help with edition of this paper. I would like also to thank Jakub Byszewski for discussion which was useful in preparing Section \ref{sec2} and giving the heuristic reasoning to Conjecture \ref{con1}. I wish to thank Maciej Gawron for help with numerical computations, too.

\end{document}